\documentclass[12pt]{amsart}
\usepackage{amscd,amssymb,amsthm,amsmath,amssymb,textcomp,mathrsfs,mathtools,wasysym}
\usepackage[matrix,arrow,curve]{xy}


\newtheorem{theorem}{Theorem}[section]

\newtheorem{lemma}[theorem]{Lemma}
\newtheorem{corollary}[theorem]{Corollary}
\newtheorem*{corollary*}{Corollary}
\newtheorem*{maincorollary*}{Main Corollary}

\newtheorem*{conjecture*}{Conjecture}

\newtheorem*{problem*}{Problem}

\newtheorem*{theorem*}{Theorem}
\newtheorem*{maintheorem*}{Main Theorem}
\newtheorem*{auxiliarytheorem*}{Auxiliary Theorem}
\theoremstyle{definition}
\newtheorem{example}[theorem]{Example}
\newtheorem*{example*}{Example}

\theoremstyle{remark}
\newtheorem{remark}[theorem]{Remark}
\newtheorem*{remark*}{Remark}

\makeatletter\@addtoreset{equation}{section} \makeatother

\title{K-stable smooth Fano threefolds of Picard rank two}

\author{Ivan Cheltsov, Elena Denisova, Kento Fujita}

\address{\emph{Ivan Cheltsov}, \newline \textnormal{University of Edinburgh,  Edinburgh, Scotland}}

\address{\emph{Elena Denisova}, \newline \textnormal{University of Edinburgh,  Edinburgh, Scotland}}

\address{\emph{Kento Fujita}, \newline  \textnormal{Osaka University, Osaka, Japan}}

\thanks{Throughout this paper, all varieties are assumed to be projective and defined over~$\mathbb{C}$.}

\begin{document}
\maketitle

\begin{abstract}
We prove that all smooth Fano threefolds in the families
\textnumero 2.1, \textnumero 2.2, \textnumero 2.3, \textnumero 2.4, \textnumero 2.6 and \textnumero 2.7 are K-stable,
and we also prove that smooth Fano threefolds in the family \textnumero 2.5 that satisfy one very explicit generality condition are K-stable.
\end{abstract}

\tableofcontents

\section{Introduction}
\label{section:introduction}

Recently, we witnessed huge progress on K-stability of Fano varieties \cite{AbbanZhuang,Fujita2019Crelle,Li,Xu},
which resulted in an~immense breakthrough in the solution of the following problem:

\begin{problem*}
Find all K-polystable smooth Fano threefolds.
\end{problem*}

Let us describe what is  done in this direction. To do this, fix a smooth Fano threefold~$X$.
Then $X$ belongs to one of the~$105$ deformation families found by Iskovskikh, Mori, Mukai.
These families are often labeled as \textnumero 1.1, \textnumero 1.2, \textnumero 1.3, \textnumero 1.4, $\ldots$, \textnumero 7.1, \textnumero 8.1, \textnumero 9.1, \textnumero 10.1,
where the first digit stands for the rank of the Picard group of the threefolds in the family.
\iffalse
For instance, the first $4$ deformation families can be described as follows:
\begin{itemize}
\item[\textnumero 1.1:] this family consists of smooth sextic hypersurfaces in $\mathbb{P}(1,1,1,1,3)$;
\item[\textnumero 1.2:] this family consists of smooth quartic threefolds in $\mathbb{P}^4$ and their degenerations,
which both can be described as complete intersections of a quadric and a quartic in the weighted projective space $\mathbb{P}(1,1,1,1,1,2)$;
\item[\textnumero 1.3:] this family consists of complete intersections of a quadric and a cubic in $\mathbb{P}^5$;
\item[\textnumero 1.4:] this family consists of smooth complete intersections of three quadrics in $\mathbb{P}^6$.
\end{itemize}
\fi
For the~detailed descriptions of the $105$ families, we refer the reader to \cite{Book}.

The valuative criterion for K-stability \cite{Fujita2019Crelle,Li} allows us to explain when $X$ is K-stable in relatively simple terms.
To do this, let us remind what are $\beta$-invariant and $\delta$-invariant.
Recall that $E$ is a prime divisor over $X$ if there is a birational morphism $f\colon\widetilde{X}\to X$ with normal $\widetilde{X}$ and a prime divisor $E\subset \widetilde{X}$.
We set $\beta(E)=A_X(E)-S_X(E)$,
where
$$
S_X(E)=\frac{1}{(-K_X)^3}\int_{0}^{\infty}\mathrm{vol}\big(f^*(-K_X)-uE)du,
$$
and $A_X(E)$ is the~log discrepancy of the~divisor $E$. Then, for every point $P\in X$, we let
$$
\delta_P(X)=\inf_{\substack{F/X\\ P\in C_X(F)}}\frac{A_{X}(F)}{S_X(F)},
$$
where the~infimum is taken over all prime divisors over $X$ whose centers on $X$ contain $P$.
Finally, we set
$$
\delta(X)=\inf_{P\in X}\delta_P(X).
$$
The valuative criterion says that $X$ is K-stable (respectively, K-semistable) if $\beta(F)>0$ (respecitvely, $\beta(F)\geqslant 0$) for any prime divisor $F$ over~$X$. Hence, $X$ is K-stable if $\delta(X)>1$.
We say that $X$ is K-unstable if it is not K-semistable.

The simplest way to apply the valuative criterion is to check whether $\beta(S)<0$ or not for some irreducible surface $S$ in the~threefold $X$.
This has been done in \cite{Fujita2019Crelle}.
As a result, we know that $X$ is K-unstable if it belongs to any of the following $26$ families:
\begin{center}
\textnumero 2.23, \textnumero 2.28, \textnumero 2.30, \textnumero 2.31, \textnumero 2.33, \textnumero 2.35, \textnumero 2.36, \textnumero 3.14, \\
\textnumero 3.16, \textnumero 3.18, \textnumero 3.21, \textnumero 3.22, \textnumero 3.23, \textnumero 3.24, \textnumero 3.26, \textnumero 3.28, \textnumero 3.29, \\
\textnumero 3.30, \textnumero 3.31, \textnumero 4.5, \textnumero 4.8, \textnumero 4.9, \textnumero 4.10, \textnumero 4.11, \textnumero 4.12, \textnumero 5.2.
\end{center}
To be precise, if $X$ is contained in one of these $26$ families, then  it contains an~irreducible surface $S$ such that $\beta(S)<0$,
so that $X$ is K-unstable, which implies that $X$ does not admit a K\"ahler--Einstein metric.
Almost in every case, such surface $S$ is not hard~to~find.
For instance, if $X$ is the unique smooth Fano threefold in the deformation family  \textnumero 2.35, then $X$ is a blow up of $\mathbb{P}^3$ at a point,
and $S$ is the exceptional surface of the blow up.

We listed $26$  families of smooth Fano threefolds that contain no K-semistable members.
There is another family that does not contain K-polystable members --- the family \textnumero 2.26.
This family is quite special, see \cite[\S~5.10]{Book} for its detailed description,
and it contains exactly two smooth Fano threefolds.
One of them is K-semistable and not K-polystable, while another one is K-unstable.
Moreover, it follows from \cite[Main~Theorem]{Book} that general members of the remaining $78$ families of smooth Fano threefolds
are K-polystable.

Further, all K-polystable smooth Fano 3-folds in $53$ families among these $78$ families are \mbox{described} in
\cite{AbbanZhuangSeshadri,Book,BelousovLoginov,CheltsovPark,Denisova,LTN,Liu,Malbon,XuLiu,CheltsovFujitaKishimotoOkada,CheltsovFujitaKishimotoPark}.
The remaining $25$ families are:
\begin{center}
\textnumero 1.9, \textnumero 1.10, \textnumero 2.1, \textnumero 2.2, \textnumero 2.3, \textnumero 2.4, \textnumero 2.5, \textnumero 2.6, \textnumero 2.7, \textnumero 2.9, \textnumero 2.10, \textnumero 2.11, \textnumero 2.12, \\
\textnumero 2.13, \textnumero 2.14, \textnumero 2.16, \textnumero 2.17, \textnumero 2.19, \textnumero 2.20, \textnumero 3.2, \textnumero 3.5, \textnumero 3.6, \textnumero 3.7, \textnumero 3.8, \textnumero 3.11.
\end{center}
The deformation families \textnumero 1.10, \textnumero 2.20, \textnumero 3.5, \textnumero 3.8 contain non-K-polystable members,
and \cite[\S~7]{Book} provides conjectures that describe K-polystable smooth Fano threefolds in these $4$ families.
On the other hand, all smooth Fano threefolds in the~$21$ families
\begin{center}
\textnumero 1.9, \textnumero 2.1, \textnumero 2.2, \textnumero 2.3, \textnumero 2.4, \textnumero 2.5, \textnumero 2.6, \\
\textnumero 2.7, \textnumero 2.9, \textnumero 2.10, \textnumero 2.11, \textnumero 2.12, \textnumero 2.13, \textnumero 2.14, \\
\textnumero 2.16, \textnumero 2.17, \textnumero 2.19, \textnumero 3.2, \textnumero 3.6, \textnumero 3.7, \textnumero 3.11
\end{center}
are conjectured to be K-stable \cite{Book}.
The goal of this paper is to prove this for $6$ families:

\begin{maintheorem*}
Let $X$ be a~smooth Fano threefold contained in one of the~following deformation families:
\textnumero 2.1, \textnumero 2.2, \textnumero 2.3, \textnumero 2.4, \textnumero 2.6, \textnumero 2.7.
Then  $X$ is K-stable.
\end{maintheorem*}

Therefore, to solve the problem posed above, it remains to find all K-polystable smooth Fano threefolds
in the following $19$ families:
\begin{center}
\textnumero 1.9, \textnumero 1.10, \textnumero 2.5, \textnumero 2.9, \textnumero 2.10, \textnumero 2.11, \textnumero 2.12, \textnumero 2.13, \textnumero 2.14, \\
\textnumero 2.16, \textnumero 2.17, \textnumero 2.19, \textnumero 2.20, \textnumero 3.2, \textnumero 3.5, \textnumero 3.6, \textnumero 3.7, \textnumero 3.8, \textnumero 3.11.
\end{center}
We hope that this will be done in a nearest future.

To prove Main Theorem, we use Abban--Zhuang theory \cite{AbbanZhuang} and its applications \cite{Book,Fujita2021}.
A similar approach also works for almost all smooth Fano threefolds in the~family \textnumero 2.5.
Namely, if $X$ is a smooth Fano threefold in the~family \textnumero 2.5,
we~show that $\delta_P(X)>1$ for every point $P\in X$ that satisfy certain geometric conditions.
As~a~result, we obtained

\begin{auxiliarytheorem*}
Let $X$ be a~smooth Fano threefold in the~deformation family \textnumero 2.5.
Recall that there exists the following Sarkisov link:
$$
\xymatrix{
&X\ar@{->}[ld]_{\pi}\ar@{->}[rd]^{\phi}&&\\%
V&&\mathbb{P}^1}
$$
where $V$ is a~smooth cubic threefold in $\mathbb{P}^4$,
the morphism $\pi$ is a blow up of a smooth plane cubic curve,
and $\phi$ is a morphism whose fibers are normal cubic surfaces.
Suppose~that
\begin{equation}
\label{equation:generality-condition}\tag{$\bigstar$}
\text{no fiber of the morphism $\phi$ has a Du Val singular point of type $\mathbb{D}_5$ or $\mathbb{E}_6$.}
\end{equation}
Then  $X$ is K-stable.
\end{auxiliarytheorem*}

Let us describe the structure of this paper.
In Section~\ref{section:2-1}, we prove Auxiliary Theorem,
and we prove that all smooth Fano~threefolds in the~families \textnumero 2.1 and \textnumero 2.3 are K-stable.
In Sections~\ref{section:2-2}, \ref{section:2-4}, \ref{section:2-6}, \ref{section:2-7},
we prove that all smooth Fano~threefolds in the~families \textnumero 2.2,  \textnumero 2.4,  \textnumero 2.6, \textnumero 2.7 are K-stable, respectively.
Note that Section~\ref{section:2-7} is very technical and long.

As we already mentioned, we use applications of Abban--Zhuang theory \cite{AbbanZhuang} which have been discovered in \cite{Book,Fujita2021}.
For the~background material, we refer the~reader to~\cite{Book,Fujita2021,Xu}.

\section{Families \textnumero 2.1, \textnumero 2.3, \textnumero 2.5 }
\label{section:2-1}

Fix $d\in\{1,2,3\}$. Let $V$ be one of the~following smooth Fano threefolds:
\begin{enumerate}
\item[$\boxed{d=1}$] a~smooth sextic hypersurface in $\mathbb{P}(1,1,1,2,3)$;
\item[\boxed{d=2}] a~smooth quartic hypersurface in $\mathbb{P}(1,1,1,1,2)$;
\item[\boxed{d=3}] a~smooth cubic threefold in $\mathbb{P}^4$.
\end{enumerate}
Then  $-K_{V}\sim 2H$ for an ample divisor $H\in\mathrm{Pic}(V)$ such that $H^3=d$ and $\mathrm{Pic}(V)=\mathbb{Z}[H]$.
Let $S_1$ and $S_2$ be two distinct surfaces in the~linear system $|H|$, and let $\mathcal{C}=S_1\cap S_2$.
Suppose that the~curve $\mathcal{C}$ is smooth.
Then  $\mathcal{C}$ is an elliptic curve by the~adjunction~formula.
Let $\pi\colon X\to V$ be the~blow up of the~curve $\mathcal{C}$, and let $E$ be the~$\pi$-exceptional surface.
\begin{itemize}
\item If $d=1$, then $X$ is a~smooth Fano threefold in the~deformation family \textnumero 2.1.
\item If $d=2$, then $X$ is a~smooth Fano threefold in the~deformation family \textnumero 2.3.
\item If $d=3$, then $X$ is a~smooth Fano threefold in the~deformation family \textnumero 2.5.
\end{itemize}
Moreover, all smooth Fano threefolds in these families can be obtained in this way.

Note that $(-K_X)^3=4d$. Moreover, we have the~following commutative diagram:
$$
\xymatrix{
&X\ar@{->}[ld]_{\pi}\ar@{->}[rd]^{\phi}&&\\%
V\ar@{-->}[rr]&&\mathbb{P}^1}
$$
where $V\dasharrow\mathbb{P}^1$ is the~rational map given by the~pencil that is generated by  $S_1$ and~$S_2$,
and $\phi$ is a~morphism whose general fiber is a~smooth del Pezzo surface of degree $d$.

The goal of this section is to show that $X$ is K-stable in the case when $d=1$ or $d=2$,
and to show that $X$ is K-stable in the case when $d=3$ and $X$ satisfies the~condition \eqref{equation:generality-condition}.
To show that $X$ is K-stable,  it is enough to show that $\delta_O(X)>1$ for every point $O\in X$.
This follows~from the~valuative criterion for K-stability \cite{Fujita2019Crelle,Li}.

\begin{lemma}
\label{lemma:2-1-2-3-2-5-del-Pezzo-delta}
Let $O$ be a point in $X$, let $A$ be the fiber of the~morphism $\phi$ such that~$O\in A$.
Suppose that $A$ has at most Du Val singularities at the~point $O$. Then
$$
\delta_O(X)\geqslant\left\{\aligned
&\mathrm{min}\Big\{\frac{16}{11},\frac{16}{15}\delta_O(A)\Big\}\ \text{if $O\not\in E$}, \\
&\mathrm{min}\Big\{\frac{16}{11},\frac{16\delta_{O}(A)}{\delta_{O}(A)+15}\Big\}\ \text{if $O\in E$}.
\endaligned
\right.
$$
\end{lemma}

\begin{proof}
Let $u$ be a non-negative real number. Then  $-K_X-uA\sim_{\mathbb{R}} (2-u)A+E$,
which implies that divisor $-K_X-uA$ is pseudoeffective if and only if $u\leqslant 2$.
For every $u\in[0,2]$, let us denote by $P(u)$ the positive part of Zariski decomposition of the divisor $-K_X-uA$,
and let us denote by $N(u)$ its negative part. Then
$$
P(u)=\left\{\aligned
&(2-u)A+E\ \text{if $0\leqslant u\leqslant 1$}, \\
&(2-u)H\ \text{if $1\leqslant u\leqslant 2$},
\endaligned
\right.
$$
and
$$
N(u)=\left\{\aligned
&0\ \text{if $0\leqslant u\leqslant 1$}, \\
&(u-1)E\ \text{if $1\leqslant u\leqslant 2$}.
\endaligned
\right.
$$
Integrating, we get $S_X(A)=\frac{11}{16}$.
Using \cite[Theorem~3.3]{AbbanZhuang} and \cite[Corollary 1.102]{Book}, we get
\begin{equation}
\label{equation:2-1-2-3-2-5-first}
\delta_O(X)\geqslant \min\left\{\frac{1}{S_X(A)},\inf_{\substack{F/A\\ O\in C_A(F)}}\frac{A_A(F)}{S(W_{\bullet,\bullet}^A;F)}\right\}=\min\left\{\frac{16}{11},\inf_{\substack{F/A\\ O\in C_A(F)}}\frac{A_A(F)}{S(W_{\bullet,\bullet}^A;F)}\right\}
\end{equation}
where the infimum is taken by all prime divisors $F$ over the~surface $A$ with~\mbox{$O\in C_A(F)$}.
The value $S(W_{\bullet,\bullet}^A;F)$ can be computed using \cite[Corollary~1.108]{Book} as follows:
$$
S\big(W_{\bullet,\bullet}^A; F\big)=
\frac{3}{4d}\int_1^{2}\big(P(u)\big\vert_{A}\big)^2(u-1)\mathrm{ord}_{F}\big(E\big\vert_{A}\big)du+
\frac{3}{4d}\int_{0}^{2}\int_0^\infty \mathrm{vol}\big(P(u)\big\vert_{A}-vF\big)dvdu.
$$

Now, let $F$ be any prime divisor over the surface $A$ such that $O\in C_A(F)$. Since
$$
P(u)\big\vert_{A}=\left\{\aligned
&-K_A\ \text{if $0\leqslant u\leqslant 1$}, \\
&(2-u)(-K_A)\ \text{if $1\leqslant u\leqslant 2$},
\endaligned
\right.
$$
we have
\begin{multline*}
\quad\quad\quad\quad\quad S\big(W_{\bullet,\bullet}^A; F\big)=
\frac{3}{4d}\int_1^{2}d(2-u)^2(u-1)\mathrm{ord}_{F}\big(E\big\vert_{A}\big)du+\\
+\frac{3}{4d}\int_{0}^{1} \int_{0}^{\infty}\mathrm{vol}\big(-K_A-vF\big)dvdu+\frac{3}{4d}\int_{1}^{2}\int_{0}^{\infty}\mathrm{vol}\big((2-u)(-K_A)-vF\big)dvdu=\\
=\frac{\mathrm{ord}_{F}\big(E\big\vert_{A}\big)}{16}+\frac{3}{4d}\int_{0}^{\infty}\mathrm{vol}\big(-K_A-vF\big)dv+\frac{3}{4d}\int_{1}^{2}(2-u)^3\int_{0}^{\infty}\mathrm{vol}\big(-K_A-vF\big)dvdu=\\
=\frac{\mathrm{ord}_{F}\big(E\big\vert_{A}\big)}{16}+\frac{3}{4d}\int_{0}^{\infty}\mathrm{vol}\big(-K_A-vF\big)dv+\frac{3}{16d}\int_{0}^{\infty}\mathrm{vol}\big(-K_A-vF\big)dv=\\
=\frac{\mathrm{ord}_{F}\big(E\big\vert_{A}\big)}{16}+\frac{15}{16d}\int_{0}^{\infty}\mathrm{vol}\big(-K_A-vF\big)dv=\\
=\frac{\mathrm{ord}_{F}\big(E\big\vert_{A}\big)}{16}+\frac{15}{16}S_A(F)\leqslant\frac{\mathrm{ord}_{F}\big(E\big\vert_{A}\big)}{16}+\frac{15A_{A}(F)}{16\delta_{O}(A)}.\quad\quad\quad\quad
\end{multline*}
Therefore, if $O\not\in E$, then $\mathrm{ord}_{F}(E\vert_{A})=0$, which implies that
$$
S\big(W_{\bullet,\bullet}^A;F\big)\leqslant\frac{15A_{A}(F)}{16\delta_{O}(A)}.
$$
Similarly, if $O\in E$, then $\mathrm{ord}_{F}(E\vert_{A})\leqslant A_A(F)$, because $(A,E\vert_A)$ is log canonical,
so that
$$
S\big(W_{\bullet,\bullet}^A; F\big)=\frac{\mathrm{ord}_{F}\big(E\big\vert_{A}\big)}{16}+\frac{15}{16}S_A(F)\leqslant\frac{A_A(F)}{16}+\frac{15A_{A}(F)}{16\delta_{O}(A)}=\frac{\delta_{O}(A)+15}{16\delta_{O}(A)}A_{A}(F).
$$
Now, using \eqref{equation:2-1-2-3-2-5-first}, we obtain the required inequality.
\end{proof}

Suppose $X$ is not K-stable. Let us seek for a contradiction.
Using the~valuative criterion for K-stability \cite{Fujita2019Crelle,Li},
we see that there exists a~prime divisor $\mathbf{F}$ over $X$~such~that
$$
\beta(\mathbf{F})=A_X(\mathbf{F})-S_X(\mathbf{F})\leqslant 0,
$$
where $A_X(\mathbf{F})$ is a~log discrepancy of the~divisor $\mathbf{F}$, and $S_X(\mathbf{F})$ is defined in \cite{Fujita2019Crelle} or~\mbox{\cite[\S~1.2]{Book}}.
Let $Z$ be the~center of the~divisor $\mathbf{F}$ on $X$.
Then  $Z$ is not a~surface~\mbox{\cite[Theorem~3.17]{Book}}.
We see that $Z$ is an~irreducible curve or a~point. Let $P$ be a~point in $Z$.
Then  $\delta_P(X)\leqslant 1$.

\begin{lemma}
\label{lemma:2-1-P-in-E-d-3}
One has $P\not\in E$.
\end{lemma}

\begin{proof}
Let us compute $S_X(E)$. Note that $S_X(E)<1$ by \cite[Theorem~3.17]{Book}.
Fix $u\in\mathbb{R}_{\geqslant 0}$.
Then  $-K_X-uE$ is pseudoeffective $\iff$ $-K_X-uE$ is nef $\iff$ $u\leqslant 1$.
Thus, we have
$$
S_X(E)=\frac{1}{4d}\int_{0}^{1}\big(-K_X-uE\big)^3du=\frac{1}{4d}\int_{0}^{1}d(2u^3-6u+4)du=\frac{3}{8}.
$$

Suppose that $P\in E$. Let us seek for a~contradiction.

Note that $E\cong\mathcal{C}\times\mathbb{P}^1$.
Let $\mathbf{s}$ be a~fiber of the~projection $\phi\vert_{E}\colon E\to\mathbb{P}^1$ that contains~$P$,
and let $\mathbf{f}$ be a~fiber of the~projection $\pi\vert_{E}\colon E\to\mathcal{C}$.
Fix $u\in[0,1]$ and take $v\in\mathbb{R}_{\geqslant 0}$. Then
$$
-K_X-uE\big\vert_{E}-v\mathbf{s}\equiv (1+u-v)\mathbf{s}+d(1-u)\mathbf{f}.
$$
This implies that
\begin{center}
$(-K_X-uE)\big\vert_{E}-v\mathbf{s}$ is pseudoeffective $\iff$ $(-K_X-uE)\big\vert_{E}-v\mathbf{s}$ is nef  $\iff$ $v\leqslant 1+v$.
\end{center}
Therefore, using  \cite[Corollary~1.109]{Book}, we get
$$
S\left(W_{\bullet,\bullet}^{E};\mathbf{s}\right)=\frac{3}{4d} \int_{0}^{1} \int_{0}^{1+u} 2d(1-u)(1+u-v)dvdu=\frac{11}{16}.
$$
Similarly, using \cite[Theorem~1.112]{Book}, we get
$$
S\left(W_{\bullet,\bullet, \bullet}^{E,\mathbf{s}};P\right)=\frac{3}{4d}\int_{0}^{1} \int_{0}^{1+u}\big(d(1-u)\big)^{2}dvdu=\frac{5d}{16}.
$$
Therefore, it follows from \cite[Theorem~1.112]{Book} that
$$
1\geqslant\frac{A_X(\mathbf{F})}{S_X(\mathbf{F})}\geqslant \min\left\{\frac{1}{S_{X}(E)}, \frac{1}{S\left(W_{\bullet,\bullet}^E,\mathbf{s}\right)},
\frac{1}{S\left(W_{\bullet,\bullet,\bullet}^{E,\mathbf{s}};P\right)}\right\}=\min \left\{\frac{8}{3}, \frac{16}{11}, \frac{16}{5d}\right\}\geqslant\frac{16}{15}>1,
$$
which is a~contradiction.
\end{proof}

Let $A$ be the~fiber of the~del Pezzo fibration $\phi$ such that $A$ passes through the~point~$P$.
Then  $A$ is a~del Pezzo surface of degree $d\in\{1,2,3\}$ that has at most isolated singularities.
In particular, we see that $A$ is normal. Applying Lemmas~\ref{lemma:2-1-2-3-2-5-del-Pezzo-delta} and \ref{lemma:2-1-P-in-E-d-3}, we obtain

\begin{corollary}
\label{corollary:2-1-2-3-2-5-del-Pezzo-delta}
One has $\delta_P(A)\leqslant\frac{15}{16}$.
\end{corollary}

\begin{proof}
Since $1\geqslant\frac{A_X(\mathbf{F})}{S_X(\mathbf{F})}\geqslant\delta_P(X)$ , we get $\delta_P(A)\leqslant\frac{15}{16}$ by Lemmas~\ref{lemma:2-1-2-3-2-5-del-Pezzo-delta} and \ref{lemma:2-1-P-in-E-d-3}.
\end{proof}

\begin{corollary}
\label{corollary:2-1-2-3-2-5-A-singular}
The surface $A$ is singular.
\end{corollary}

\begin{proof}
If $A$ is smooth, then $\delta_P(A)\geqslant\delta(A)\geqslant\frac{3}{2}$ \cite[\S~2]{Book}, which contradicts Corollary~\ref{corollary:2-1-2-3-2-5-del-Pezzo-delta}.
\end{proof}

Let $\overline{S}$ be a~general surface in $|H|$ that passes through  $\pi(P)$,
and let $S$ be the~proper transform on $X$ of the~surface $\overline{S}$. Then
\begin{itemize}
\item the~surface $\overline{S}$ is a~smooth del Pezzo surface of degree $d$,
\item the~surface $\overline{S}$ intersects the~curve $\mathcal{C}$ transversally at $d$ points,
\item the~induced morphism $\pi\vert_{S}\colon S\to\overline{S}$ is a~blow up of the~points $\overline{S}\cap\mathcal{C}$.
\end{itemize}
Observe that  $\phi\vert_S\colon S\to \mathbb{P}^1$ is an elliptic fibration given by the~pencil $|-K_S|$.
Set $C=A\big\vert_{S}$.
Then  $C$ is a~reduced curve of arithmetic genus $1$ in $|-K_S|$ that has at most $d$ components.
In~particular, if $d=1$, then $C$ is irreducible.
Therefore, the~following cases may happen:
\begin{enumerate}
\item the~curve $C$ is irreducible, and $C$ is smooth at $P$,
\item the~curve $C$ is irreducible, and $C$ has an ordinary node at $P$,
\item the~curve $C$ is irreducible, and $C$ has an~ordinary cusp at $P$,
\item the~curve $C$ is reducible.
\end{enumerate}

Fix $u\in\mathbb{R}_{\geqslant 0}$.
Then  $-K_X-uS$ is nef $\iff$ $u\leqslant 1$ $\iff$ $-K_X-uS$ is pseudoeffective.
Using this, we see that
$$
S_X(S)=\frac{1}{4d}\int_{0}^{1}(-K_X-uS)^3du=\frac{1}{4d}\int_{0}^{1}d(4-u)(1-u)^2du=\frac{5}{16}<1,
$$
which also follows from \cite[Theorem~3.17]{Book}. Moreover, if $u\in[0,1]$, then
$$
(-K_X-uS)|_S\sim_{\mathbb{R}}(1-u)(\pi\vert_{S})^*(-K_{\overline{S}})-K_S\sim_{\mathbb{R}}(1-u)\sum_{i=1}^{d}\mathbf{e}_i+(2-u)(-K_S),
$$
where $\mathbf{e}_1,\ldots,\mathbf{e}_d$ are exceptional curves of the~blow up $\pi\vert_{S}\colon S\to\overline{S}$.

\begin{lemma}
\label{lemma:2-1-S-C-smooth-P}
Suppose that $C$ is irreducible. Then  $C$ is singular at the~point $P$.
\end{lemma}

\begin{proof}
As in the proof of Lemma~\ref{lemma:2-1-P-in-E-d-3}, it follows from \cite[Theorem 1.112]{Book} that
$$
1\geqslant \frac{A_X(\mathbf{F})}{S_X(\mathbf{F})}\geqslant \min\left\{\frac{1}{S_X(S)},\frac{1}{S(W_{\bullet,\bullet}^S;C)},
\frac{1}{S(W_{\bullet, \bullet,\bullet}^{S,C};P)}\right\},
$$
where $S(W_{\bullet,\bullet}^S;C)$ and $S(W_{\bullet, \bullet,\bullet}^{S,C};P)$ are defined in \cite[\S~1.7]{Book}.
Since we know that~\mbox{$S_X(S)<1$}, we see that $S(W_{\bullet,\bullet}^S;C)\geqslant 1$ or $S(W_{\bullet, \bullet,\bullet}^{S,C};P)\geqslant 1$.
Let us compute these numbers.

Let $P(u,v)$ be the~positive part of the~Zariski decomposition of $(-K_X-uS)\vert_{S}-vC$,
and let $N(u,v)$ be its~negative part, where $u\in[0,1]$ and $v\in\mathbb{R}_{\geqslant 0}$. Since
$$
(-K_X-uS)\vert_{S}-vC\sim_{\mathbb{R}}(1-u)\sum_{i=1}^{d}\mathbf{e}_i+(2-u-v)C,
$$
we see that $(-K_X-uS)\vert_{S}-vC$ is pseudoeffective $\iff$ $v\leqslant 2-u$.
Moreover, we have
$$
P(u,v)=\left\{\aligned
&(1-u)\sum_{i=1}^{d}\mathbf{e}_i+(2-u-v)C\ \text{if $0\leqslant v\leqslant 1$}, \\
&(2-u-v)\Big(C+\sum_{i=1}^{d}\mathbf{e}_i\Big)\ \text{if $1\leqslant v\leqslant 2-u$}, \\
\endaligned
\right.
$$
and
$$
N(u,v)=\left\{\aligned
&0\ \text{if $0\leqslant v\leqslant 1$}, \\
&(v-1)\sum_{i=1}^{d}\mathbf{e}_i\ \text{if $1\leqslant v\leqslant 2-u$}. \\
\endaligned
\right.
$$
Thus, it follows from \cite[Corollary~1.109]{Book} that
\begin{multline*}
\quad \quad \quad \quad \quad\quad \quad \quad  S\big(W^S_{\bullet,\bullet};C\big)=\frac{3}{4d}\int_0^1\int_0^{2-u}P(u,v)^2dvdu=\\
=\frac{3}{4d}\int_0^1\int_0^1d(1-u)(3-u-2v)dvdu+\frac{3}{4d}\int_0^1\int_1^{2-u}d(2-u-v)^2dvdu=\frac{11}{16}<1.
\end{multline*}
Thus, we conclude that  $S(W_{\bullet, \bullet,\bullet}^{S,C};P)\geqslant 1$.

Since $P\not\in\mathbf{e}_1\cup\cdots\cup\mathbf{e}_d$ by Lemma~\ref{lemma:2-1-P-in-E-d-3}, it follows from \cite[Theorem~1.112]{Book} that
\begin{multline*}
\quad \quad \quad \quad \quad \quad \quad \quad \quad S\big(W_{\bullet, \bullet,\bullet}^{S,C};P\big)=\frac{3}{4d}\int_0^1\int_0^{2-u}\big(P(u,v)\cdot C\big)^2dvdu=\\
=\frac{3}{4d}\int_0^1\int_0^1d^2(u-1)^2dvdu+\frac{3}{4d}\int_0^1\int_1^{2-u}d^2(2-u-v)^2dvdu=\frac{5d}{16}<1,\quad \quad \quad
\end{multline*}
which is a~contradiction.
\end{proof}

Now, let us show that $C$ is reducible for $d\in\{1,2\}$.

\begin{lemma}
\label{lemma:2-1-S-C-reducible}
Suppose that $C$ is irreducible.
Then  $d=3$ and $C$ has a cusp at $P$.
\end{lemma}

\begin{proof}
By Lemma~\ref{lemma:2-1-S-C-smooth-P}, the curve $C$ is singular at the point $P$.

Now, let $\sigma\colon\widetilde{S}\to S$ be the~blow up of the~point $P$, let $\mathbf{f}$ be the~$\sigma$-exceptional curve,
and let $\widetilde{\mathbf{e}}_1,\ldots,\widetilde{\mathbf{e}}_d,\widetilde{C}$ be the~proper transforms on $\widetilde{S}$ of the~curves $\mathbf{e}_1,\ldots,\mathbf{e}_d,C$, respectively.
Then  the~curve $\widetilde{C}$ is smooth, and it follows from \cite[Remark~1.113]{Book} that
$$
1\geqslant \frac{A_X(\mathbf{F})}{S_X(\mathbf{F})}\geqslant\min\left\{\inf_{O\in\mathbf{f}}\frac{1}{S(W_{\bullet,\bullet,\bullet}^{\widetilde{S},\mathbf{f}};O)},
\frac{2}{S(V_{\bullet,\bullet}^{S};\mathbf{f})},\frac{1}{S_X(S)}\right\},
$$
where $S(W_{\bullet, \bullet,\bullet}^{\widetilde{S},\mathbf{f}};O)$ and $S(V_{\bullet,\bullet}^{S};\mathbf{f})$
are defined in \cite[\S~1.7]{Book}.
Since we know that $S_X(S)<1$, we see that $S(V_{\bullet,\bullet}^{S};\mathbf{f})\geqslant 2$ or
there exists a~point $O\in\mathbf{f}$ such that $S(W_{\bullet,\bullet,\bullet}^{\widetilde{S},\mathbf{f}};O)\geqslant 1$.

Let us compute $S(V_{\bullet,\bullet}^{S};\mathbf{f})$.
Fix $u\in[0,1]$ and $v\in\mathbb{R}_{\geqslant 0}$.
Since $\sigma^*(C)\sim\widetilde{C}+2\mathbf{f}$, we get
$$
\sigma^*\big((-K_X-uS)\vert_{S}\big)-v\mathbf{f}\sim_{\mathbb{R}}(2-u)\widetilde{C}+(4-2u-v)\mathbf{f}+(1-u)\sum_{i=1}^{d}\widetilde{\mathbf{e}}_i.
$$
Then the divisor $\sigma^*((-K_X-uS)\vert_{S})-v\mathbf{f}$ is pseudoeffective $\iff$ $v\leqslant 4-2u$.

Let $P(u,v)$ be the~positive part of the~Zariski decomposition of $\sigma^*((-K_X-uS)\vert_{S})-v\mathbf{f}$,
and let $N(u,v)$ be its~negative part.
Then
$$
P(u,v)=\left\{\aligned
&(2-u)\widetilde{C}+(4-2u-v)\mathbf{f}+(1-u)\sum_{i=1}^{d}\widetilde{\mathbf{e}}_i\ \text{if $0\leqslant v\leqslant \frac{d-du}{2}$}, \\
&\frac{8+d-4u-du-2v}{4}\widetilde{C}+(4-2u-v)\mathbf{f}+(1-u)\sum_{i=1}^{d}\widetilde{\mathbf{e}}_i\ \text{if $\frac{d-du}{2}\leqslant v\leqslant \frac{4+d-du}{2}$}, \\
&\frac{4-2u-v}{4-d}\Big(2\widetilde{C}+(4-d)\mathbf{f}+2\sum_{i=1}^{d}\widetilde{\mathbf{e}}_i\Big)\ \text{if $\frac{4+d-du}{2}\leqslant v\leqslant 4-2u$},
\endaligned
\right.
$$
and
$$
N(u,v)=\left\{\aligned
&0\ \text{if $0\leqslant v\leqslant \frac{d-du}{2}$}, \\
&\frac{2v+du-d}{4}\widetilde{C}\ \text{if $\frac{d-du}{2}\leqslant v\leqslant \frac{4+d-du}{2}$}, \\
&\frac{2v+du-2d}{4-d}\widetilde{C}+\frac{2v+du-4-d}{4-d}\sum_{i=1}^{d}\widetilde{\mathbf{e}}_i\ \text{if $\frac{4+d-du}{2}\leqslant v\leqslant 4-2u$}.
\endaligned
\right.
$$
Thus, using \cite[Corollary~1.109]{Book}, we get
\begin{multline*}
S\big(W_{\bullet,\bullet}^S;\mathbf{f}\big)=\frac{3}{4d}\int_0^1\int_0^{\frac{d-du}{2}}\big(du^2-4du-v^2+3d\big)dvdu+\\
+\frac{3}{4d}\int_0^1\int_{\frac{d-du}{2}}^{\frac{4+d-du}{2}}\frac{d(1-u)(12-du+d-4u-4v)}{4}dvdu+\\
+\frac{3}{4d}\int_0^1\int_{\frac{4+d-du}{2}}^{4-2u}\frac{d(4-2u-v)^2}{4-d}dvdu=\frac{44+5d}{32}<2.
\end{multline*}
Therefore, there exists a~point $O\in\mathbf{f}$ such that $S(W_{\bullet,\bullet,\bullet}^{\widetilde{S},\mathbf{f}};O)\geqslant 1$.

Let us compute $S(W_{\bullet,\bullet,\bullet}^{\widetilde{S},\mathbf{f}};O)$. Observe that
$$
P(u,v)\cdot\mathbf{f}=\left\{\aligned
&v\ \text{if $0\leqslant v\leqslant \frac{d-du}{2}$}, \\
&\frac{d-du}{2}\ \text{if $\frac{d-du}{2}\leqslant v\leqslant \frac{4+d-du}{2}$}, \\
&\frac{d(4-2u-v)}{4-d}\ \text{if $\frac{4+d-du}{2}\leqslant v\leqslant 4-2u$}.
\endaligned
\right.
$$
Hence, it follows from \cite[Remark~1.113]{Book} that
\begin{multline*}
\quad\quad\quad\quad S(W_{\bullet,\bullet,\bullet}^{\widetilde{S},\mathbf{f}};O)=\frac{3}{4d}\int_0^1\int_0^{4-2u}\Big(\big(P(u,v)\cdot\mathbf{f}\big)\Big)^2dvdu+\\
+\frac{6}{4d}\int_0^1\int_0^{4-2u}\big(P(u,v)\cdot\mathbf{f}\big)\mathrm{ord}_O\big(N(u,v)\big|_\mathbf{f}\big)dvdu=\\
=\frac{3}{4d}\int_0^1\int_0^{\frac{d-du}{2}}v^2dvdu+\\
+\frac{3}{4d}\int_{\frac{d-du}{2}}^{\frac{4+d-du}{2}}\Big(\frac{d-du}{2}\Big)^2dvdu+
\frac{3}{4d}\int_0^1\int_{\frac{4+d-du}{2}}^{4-2u}\Big(\frac{d(4-2u-v)}{4-d}\Big)^2dvdu+\\
+\frac{6}{4d}\int_0^1\int_0^{4-2u}\big(P(u,v)\cdot\mathbf{f}\big)\mathrm{ord}_O\big(N(u,v)\big|_\mathbf{f}\big)dvdu=\\
=\frac{5d}{32}+\frac{3}{2}\int_0^1\int_0^{4-2u}\big(P(u,v)\cdot\mathbf{f}\big)\mathrm{ord}_O\big(N(u,v)\big|_\mathbf{f}\big)dvdu.\quad\quad\quad
\end{multline*}
Therefore, if $O\not\in\widetilde{C}$, we obtain $S(W_{\bullet,\bullet,\bullet}^{\widetilde{S},\mathbf{f}};O)=\frac{5d}{32}$,
which contradicts to $S(W_{\bullet,\bullet,\bullet}^{\widetilde{S},\mathbf{f}};O)\geqslant 1$.
Similarly, if $O\in\widetilde{C}$ and $\widetilde{C}$ intersects the~curve $\mathbf{f}$ transversally at the~point $O$,
then
$$
S(W_{\bullet,\bullet,\bullet}^{\widetilde{S},\mathbf{f}};O)=\frac{5d}{32}+\frac{6}{4d}\int_0^1\int_0^{4-2u}\big(P(u,v)\cdot\mathbf{f}\big)\mathrm{ord}_O\big(N(u,v)\big|_\mathbf{f}\big)dvdu=\frac{44+5d}{64}<1.
$$
which again contradicts $S(W_{\bullet,\bullet,\bullet}^{\widetilde{S},\mathbf{f}};O)\geqslant 1$.
Therefore, the~curves $\widetilde{C}$ and $\mathbf{f}$ are tangent at~$O$,
which implies that $C$ has a~cusp at the~point $P$.

Thus, to proceed, we may assume that $d=1$ or $d=2$.

Now, let us consider the~following commutative diagram:
$$
\xymatrix{
\widetilde{S}\ar@{->}[d]_{\sigma}&&\widehat{S}\ar@{->}[ll]_{\rho}&&\overline{S}\ar@{->}[ll]_{\eta}\ar@{->}[d]^{\psi}\\%
S &&&& \mathscr{S}\ar@{->}[llll]_{\upsilon}}
$$
where
\begin{itemize}
\item $\rho$ is the~blow up of the~point $\widetilde{C}\cap\mathbf{f}$,
\item $\eta$ is the~blow up of the~intersection point of the~$\rho$-exceptional curve
and the proper transform of the~curve $\widetilde{C}$,
\item $\psi$ is the~contraction of the~proper transforms of both~$(\sigma\circ\rho)$-exceptional curves,
\item $\upsilon$ is the~birational contraction of the~proper transform of the~$\eta$-exceptional curve.
\end{itemize}
Let~$\mathscr{F}$ be the~$\upsilon$-exceptional curve, let $\mathscr{C}$ be the~proper transform on $\mathscr{S}$ of the~curve~$C$.
Then  $\mathscr{F}$ and $\mathscr{C}$ are smooth,
$\mathscr{C}^2=-6$, $\mathscr{F}^{2}=-\frac{1}{6}$, $\mathscr{C}\cdot\mathscr{F}=1$, and $\upsilon^*(C)=\mathscr{C}+6\mathscr{F}$.

Observe that $\mathscr{F}$ contains two singular points of the~surfaces $\mathscr{S}$,
which are quotient singular points of type $\frac{1}{2}(1,1)$ and $\frac{1}{3}(1,1)$.
Denote these points by $Q_2$ and $Q_3$, respectively.
Note that  $\mathscr{C}$ does not contain $Q_2$ and $Q_3$.
Write $\Delta_{\mathscr{F}}=\frac{1}{2}Q_2+\frac{2}{3}Q_3$.
Then, since $A_S(\mathscr{F})=5$, it follows from \cite[Remark~1.113]{Book} that
$$
1\geqslant \frac{A_X(\mathbf{F})}{S_X(\mathbf{F})}\geqslant\min\left\{\inf_{Q\in\mathscr{F}}\frac{1-\mathrm{ord}_{Q}(\Delta_{\mathscr{F}})}{S(W_{\bullet,\bullet,\bullet}^{\mathscr{S},\mathscr{F}};Q)},
\frac{5}{S(V_{\bullet,\bullet}^{S};\mathscr{F})},\frac{1}{S_X(S)}\right\}.
$$
But we already proved that $S_X(S)<1$. Hence, we conclude that $S(V_{\bullet,\bullet}^{S};\mathscr{F})\geqslant 5$ or
there exists a~point $Q\in\mathscr{F}$ such that
$S(W_{\bullet,\bullet,\bullet}^{\mathscr{S},\mathscr{F}};Q)\geqslant 1-\mathrm{ord}_{Q}(\Delta_{\mathscr{F}})$.

Let us compute $S(V_{\bullet,\bullet}^{S};\mathscr{F})$. Take $v\in\mathbb{R}_{\geqslant 0}$. Set $P(u)=-K_X-u S$. Then
$$
\upsilon^*\big(P(u)|_S\big)-v\mathscr{F}\sim_{\mathbb{R}}(2-u)\mathscr{C}+(1-u)\sum_{i=1}^{d}\mathscr{E}_i+(12-6u-v)\mathscr{F},
$$
where $\mathscr{E}_i$ is the~proper transform on $\mathscr{S}$ of the~$(-1)$-curve $\mathbf{e}_i$.
Using this, we conclude that the~divisor $\upsilon^*(P(u)|_S)-v\mathscr{F}$ is pseudoeffective $\iff$ $v\leqslant 12-6u$.

Let $\mathscr{P}(u,v)$ be the~positive part of the~Zariski decomposition of $\upsilon^*(P(u)|_S)-v\mathscr{F}$,
and let $\mathscr{N}(u,v)$ be its~negative part.
Then
$$
\mathscr{P}(u,v)=\left\{\aligned
&(2-u)\mathscr{C}+(1-u)\sum_{i=1}^{d}\mathscr{E}_i+(12-6u-v)\mathscr{F}\ \text{if $0\leqslant v\leqslant d(1-u)$}, \\
&\frac{12+d-(6+d)u-v}{6}\mathscr{C}+(1-u)\sum_{i=1}^{d}\mathscr{E}_i+(12-6u-v)\mathscr{F}\ \text{if $d(1-u)\leqslant v\leqslant 6+d-du$}, \\
&\frac{12-6u-v}{6-d}\Big(\mathscr{C}+\sum_{i=1}^{d}\mathscr{E}_i+(6-d)\mathscr{F}\Big)\ \text{if $6+d-du\leqslant v\leqslant 12-6u$},
\endaligned
\right.
$$
and
$$
\mathscr{N}(u,v)=\left\{\aligned
&0\ \text{if $0\leqslant v\leqslant d(1-u)$}, \\
&\frac{v+d(u-1)}{6}\mathscr{C}\ \text{if $d(1-u)\leqslant v\leqslant 6+d-du$}, \\
&\frac{v+du-2d}{6-d}\mathscr{C}+\frac{v+du-6-d}{6-d}\sum_{i=1}^{d}\mathscr{E}_i\ \text{if $6+d-du\leqslant v\leqslant 12-6u$}.
\endaligned
\right.
$$
This gives
$$
\mathscr{P}(u,v)^2=\left\{\aligned
&\frac{6du^2-24du-v^2+18d}{6}\ \text{if $0\leqslant v\leqslant d(1-u)$}, \\
&\frac{d(1-u)(18-(d+6)u+d-2v)}{6}\ \text{if $d(1-u)\leqslant v\leqslant 6+d-du$}, \\
&\frac{d(12-6u-v)^2}{6(6-d)}\ \text{if $6+d-du\leqslant v\leqslant 12-6u$}.
\endaligned
\right.
$$
Thus, using \cite[Corollary~1.109]{Book} and integrating, we get
$$
S(W_{\bullet,\bullet}^\mathscr{S};\mathscr{F})=\frac{3}{4d}\int_0^1\int_0^{12-6u}\mathscr{P}(u,v)^2dvdu=\frac{66+5d}{16}\leqslant\frac{19}{4}<5=A_S(\mathscr{F}),
$$
since $d=1$ or $d=2$.
Thus, there is a~point $Q\in\mathscr{F}$ such that $S(W_{\bullet,\bullet,\bullet}^{\mathscr{S},\mathscr{F}};Q)\geqslant 1-\mathrm{ord}_{Q}(\Delta_{\mathscr{F}})$.

Now, using \cite[Remark~1.103]{Book} again, we see that
\begin{multline*}
\quad \quad \quad \quad \quad S\big(W_{\bullet,\bullet,\bullet}^{\mathscr{S},\mathscr{F}};Q\big)=\frac{3}{4d}\int_0^1\int_0^{12-6u}\Big(\big(\mathscr{P}(u,v)\cdot\mathscr{F}\big)\Big)^2dvdu+\\
+\frac{6}{4d}\int_0^1\int_0^{12-6u}\big(P(u,v)\cdot\mathscr{F}\big)\mathrm{ord}_Q\big(\mathscr{N}(u,v)\big|_{\mathscr{F}}\big)dvdu.\quad \quad \quad \quad \quad
\end{multline*}
On the~other hand, we have
$$
\mathscr{P}(u,v)\cdot\mathscr{F}=\left\{\aligned
&\frac{v}{6}\ \text{if $0\leqslant v\leqslant d(1-u)$}, \\
&\frac{d(1-u)}{6}\ \text{if $d(1-u)\leqslant v\leqslant 6+d-du$}, \\
&\frac{d(12-6u-v)}{6(6-d)}\ \text{if $6+d-du\leqslant v\leqslant 12-6u$},
\endaligned
\right.
$$
and
$$
\mathscr{N}(u,v)\cdot\mathscr{F}=\left\{\aligned
&0\ \text{if $0\leqslant v\leqslant d(1-u)$}, \\
&\frac{v-d(1-u)}{6}\ \text{if $d(1-u)\leqslant v\leqslant 6+d-du$}, \\
&\frac{v+du-2d}{6-d}\ \text{if $6+d-du\leqslant v\leqslant 12-6u$}.
\endaligned
\right.
$$
In particular, we have
$$
S(W_{\bullet,\bullet,\bullet}^{\mathscr{S},\mathscr{F}};Q)=\frac{5d}{96}+\frac{6}{4d}\int_0^1\int_0^{12-6u}\big(P(u,v)\cdot\mathscr{F}\big)\mathrm{ord}_Q\big(\mathscr{N}(u,v)\big|_{\mathscr{F}}\big)dvdu.
$$
Hence, if $Q\not\in\mathscr{C}$, then
$\frac{1}{3}\leqslant 1-\mathrm{ord}_{Q}(\Delta_{\mathscr{F}})\leqslant S(W_{\bullet,\bullet,\bullet}^{\mathscr{S},\mathscr{F}};Q)=\frac{5d}{96}<\frac{1}{3}$,
which is absurd.
Thus, we conclude that $Q=\mathscr{C}\cap\mathscr{F}$. Then
\begin{multline*}
S(W_{\bullet,\bullet,\bullet}^{\mathscr{S},\mathscr{F}};Q)=\frac{5d}{96}+\frac{6}{4d}\int_0^1\int_0^{12-6u}\big(P(u,v)\cdot\mathscr{F}\big)\mathrm{ord}_Q\big(\mathscr{N}(u,v)\big|_{\mathscr{F}}\big)dvdu\leqslant\\
\leqslant\frac{5d}{96}+\frac{6}{4d}\int_0^1\int_0^{12-6u}\big(P(u,v)\cdot\mathscr{F}\big)\big(\mathscr{N}(u,v)\cdot\mathscr{F}\big)dvdu=\frac{11}{16}<1,
\end{multline*}
which is a~contradiction,
since $S(W_{\bullet,\bullet,\bullet}^{\mathscr{S},\mathscr{F}};Q)\geqslant 1-\mathrm{ord}_{Q}(\Delta_{\mathscr{F}})=1$.
\end{proof}

In particular, we conclude that either $d=2$ or $d=3$.

\begin{corollary}
\label{corollary:2-1-C-nodal}
All smooth Fano threefolds in the~family \textnumero 2.1 are K-stable.
\end{corollary}

Recall that $A$ is the~fiber of the~del Pezzo fibration $\phi\colon X\to\mathbb{P}^1$ that passes through~$P$.
Note also that we have the~following possibilities:
\begin{itemize}
\item $d=2$, and $A$ is a~double cover of $\mathbb{P}^2$ branched over a~reduced quartic curve;
\item $d=3$, and $A$ is a~normal cubic surface in $\mathbb{P}^3$.
\end{itemize}
Observe that $C=S\cap A$, where $S$ is a~general surface in $|\pi^*(H)|$ that contains the~point~$P$.
Since $C$ is singular at $P$, the surface $A$ must be singular at $P$, which confirms Corollary~\ref{corollary:2-1-2-3-2-5-A-singular}.
Now, using classifications of reduced singular plane quartic curves and singular normal cubic surfaces \cite{BruceWall},
we see that $P=\mathrm{Sing}(A)$, and one of the following three cases holds:
\begin{itemize}
\item $d=2$, and $A$ is a~double cover of $\mathbb{P}^2$ branched
over $4$ lines intersecting in a point;
\item $d=3$, and $A$ is a~cone in $\mathbb{P}^3$ over a smooth plane cubic curve;
\item $d=3$, and $A$ has Du Val singular point of type $\mathbb{D}_4$, $\mathbb{D}_5$, or $\mathbb{E}_6$.
\end{itemize}
Let us show that the first case is impossible.

\begin{lemma}
\label{lemma:2-1-d-2}
One has $d=3$.
\end{lemma}

\begin{proof}
Suppose that $d=2$.
Then  $P=\mathrm{Sing}(A)$, and $A$ is a~double cover of $\mathbb{P}^2$ branched
over a~reduced reducible plane quartic curve that is a~union of $4$ distinct lines passing through one point.
Let us seek for a contradiction.

Let $\alpha\colon\widetilde{X}\to X$ be the~blow up of the~point $P$,
let $E_P$ be the~$\alpha$-exceptional divisor, and let $\widetilde{A}$ be the~proper transform on $\widetilde{X}$ of the~surface $A$.
Then  $\widetilde{A}\cap E_P$ is a~line $L\subset E_P\cong\mathbb{P}^2$,
and the~surface $\widetilde{A}$ is singular along this line.
Let $\beta\colon\overline{X}\to\widetilde{X}$ be the~blow up of the~line $L$,
let~$E_L$ be the~$\beta$-exceptional divisor,
let $\overline{A}$ be the~proper transform on $\overline{X}$ of the~surface~$\widetilde{A}$,
and let $\overline{E}_P$ be the~proper transforms on $\overline{X}$  of the~surface $E_P$.
Then
\begin{itemize}
\item $E_L\cong\mathbb{F}_2$,
\item the~intersection $\overline{E}_P\cap E_L$ is the~$(-2)$-curve in $E_L$,
\item the~surface $\overline{A}$ is smooth, and there~exists a~$\mathbb{P}^1$-bundle $\overline{A}\to\mathcal{C}$,
\item $\overline{A}\cap E_L$ is a~smooth elliptic curve that is a~section of the~$\mathbb{P}^1$-bundle $\overline{A}\to\mathcal{C}$,
\item the~surfaces $\overline{A}$ and $\overline{E}_P$ are disjoint,
\item $\overline{E}_P\cong\mathbb{P}^2$ and $\overline{E}_P\vert_{\overline{E}_P}\cong\mathcal{O}_{\mathbb{P}^2}(-2)$.
\end{itemize}
There is a~birational contraction $\gamma\colon\overline{X}\to\widehat{X}$ of the surface $\overline{E}_P$ such that
$\widehat{X}$ is a~projective threefold that has one singular point $O=\gamma(\overline{E}_P)$,
which is a~terminal cyclic quotient singularities of type $\frac{1}{2}(1,1,1)$.
Thus, there exists the~following commutative diagram
$$\xymatrix{
\overline{X}\ar[rr]^{\beta}\ar[d]_{\gamma}&&\widetilde{X}\ar[d]^{\alpha}\\
\widehat{X}\ar[rr]_{\sigma}&&X}
$$
where $\sigma$ is a~birational morphism that contracts the~surface $\gamma(E_L)$ to the~point $P$.

Let $G=\gamma(E_L)$, let $\widehat{A}=\gamma(\overline{A})$, and  let $\widehat{E}$ be the~proper transform on $\widehat{X}$ of the~surface~$E$.
Then  $A_{X}(G)=4$. Moreover, we have
\begin{align*}
\sigma^{*}(-K_X)&\sim 2\widehat{A}+\widehat{E}+8G, \\
\sigma^{*}(A)&\sim\widehat{A}+4G.
\end{align*}

Note that $\widehat{A}\cong\overline{A}$ and $G\cong\mathbb{P}(1,1,2)$,
so we can identify $G$ with a~quadric cone in~$\mathbb{P}^3$.
Note also that $O$ is the~vertex of the~cone $G$. Moreover, by construction, we have $O\not\in\widehat{A}$.
Furthermore, the~exists a~$\mathbb{P}^1$-bundle $\widehat{A}\to\mathcal{C}$
such that $G\vert_{\widehat{A}}$ is its section.

Let $\mathbf{g}$ be a~ruling of the~quadric cone $G$, let $\mathbf{l}$ be a~fiber of the~$\mathbb{P}^1$-bundle $\widehat{A}\to\mathcal{C}$,
and let~$\mathbf{f}$ be a~fiber of the~$\mathbb{P}^1$-bundle $\pi\circ\sigma\vert_{\widehat{E}}\colon\widehat{E}\to\mathcal{C}$.
Then  $G\vert_{G}\sim_{\mathbb{Q}}-\mathbf{g}$ and $\widehat{A}\vert_{G}\sim_{\mathbb{Q}}4\mathbf{g}$.
Moreover, the~intersections of the~surfaces $G$, $\widehat{A}$, $\widehat{E}$ with  the~curves $\mathbf{g}$, $\mathbf{l}$, $\mathbf{f}$
are given here:
\begin{center}
\renewcommand\arraystretch{1.4}
\begin{tabular}{|c||c|c|c|}
\hline & $G$ & $\widehat{A}$ & $\widehat{E}$ \\
\hline
\hline
$\mathbf{g}$ & $-\frac{1}{2}$ & 2 & 0 \\
\hline
$\mathbf{l}$ & 1 & $-4$ & 1 \\
\hline
$\mathbf{f}$ & 0 & 1 & $-1$ \\
\hline
\end{tabular}
\end{center}

Fix a~non-negative real number $u$.
We have $\sigma^{*}(-K_X)-uG\sim_{\mathbb{R}} 2\widehat{A}+\widehat{E}+(8-u)G$,
which implies that  $\sigma^{*}(-K_X)-uG$ is pseudo-effective $\iff$ $u\in[0,8]$.
Furthermore, if~$u\in[0,8]$, then the~Zariski decomposition of the~divisor $\sigma^{*}(-K_X)-uG$ can be described as follows:
$$
P(u)=\left\{\aligned
&2\widehat{A}+\widehat{E}+(8-u)G\ \text{if $0\leqslant u\leqslant 1$}, \\
&\frac{9-u}{4}\widehat{A}+\widehat{E}+(8-u)G \ \text{if $1\leqslant u\leqslant 5$}, \\
&\frac{8-u}{3}\big(\widehat{A}+\widehat{E}+3G\big)\ \text{if $5\leqslant u\leqslant 8$}, \\
\endaligned
\right.
$$
and
$$
N(u)=\left\{\aligned
&0\ \text{if $0\leqslant u\leqslant 1$}, \\
&\frac{u-1}{4}\widehat{A}\ \text{if $1\leqslant u\leqslant 5$}, \\
&\frac{u-2}{3}\widehat{A}+\frac{u-5}{3}\widehat{E}\ \text{if $5\leqslant u\leqslant 8$}, \\
\endaligned
\right.
$$
where $P(u)$ and $N(u)$ are the~positive and the~negative parts of the~decomposition. Then
$$
\mathrm{vol}\big(\sigma^{*}(-K_X)-uG\big)=P(u)^3=\left\{\aligned
&8-\frac{u^3}{8}\ \text{if $0\leqslant u\leqslant 1$}, \\
&\frac{18-3u}{2}\ \text{if $1\leqslant u\leqslant 5$}, \\
&\frac{(8-u)^{3}}{18}\ \text{if $5\leqslant u\leqslant 8$}. \\
\endaligned
\right.
$$
Integrating, we get $S_{X}(G)=\frac{27}{8}<4=A_{X}(G)$.
But \cite[Corollary~4.18]{Fujita2021} gives
$$
1\geqslant \frac{A_X(\mathbf{F})}{S_X(\mathbf{F})}\geqslant\min\left\{\frac{A_{X}(G)}{S_{X}(G)},
\inf_{Q\in G}\delta_{Q}\big(G,V^G_{\bullet,\bullet}\big)\right\}=\min\left\{\frac{32}{27},\inf_{Q\in G}\delta_{Q}\big(G,V^G_{\bullet,\bullet}\big)\right\},
$$
where $\delta_{Q}(G,V^G_{\bullet,\bullet})$ is defined in \cite{Fujita2021}.
Moreover, if $Q$ is a~point in $G$ and $Z$ is a~smooth curve in $G$ that passes through $Q$,
then it follows from \cite[Corollary~4.18]{Fujita2021} that
$$
\delta_{Q}\big(G,V^G_{\bullet,\bullet}\big)
\geqslant\min\left\{\frac{1}{S\big(V^G_{\bullet,\bullet};Z\big)},\frac{1-\mathrm{ord}_Q(\Delta_Z)}{S\big(W^{G,Z}_{\bullet,\bullet,\bullet};Q\big)}\right\},
$$
where $S(V^G_{\bullet,\bullet};Z)$ and $S(W^{G,Z}_{\bullet,\bullet,\bullet};Q)$ are defined in \cite{Fujita2021}, and
$$
\Delta_Z=\left\{\aligned
&0\ \text{if $O\not\in Z$}, \\
&\frac{1}{2}O\ \text{if $O\in Z$}.
\endaligned
\right.
$$
Let us show that $\delta_{Q}\big(G,V^G_{\bullet,\bullet}\big)>1$ for every $Q\in G$, which would imply a contradiction.

Let $\mathscr{C}=\widehat{A}\vert_{G}$, and let $\ell$ be the~curve in $|\mathbf{f}|$ that contains $Q$.
Then  $O\not\in\mathscr{C}$ and $O\in\ell$, so that $\Delta_{\mathscr{C}}=0$ and $\Delta_{\ell}=\frac{1}{2}O$.
Take $v\in\mathbb{R}_{\geqslant 0}$. Then
$$
P(u)\big\vert_{G}-v\ell\sim_{\mathbb{R}}\left\{\aligned
&(u-v)\mathbf{g}\ \text{if $0\leqslant u\leqslant 1$}, \\
&(1-v)\mathbf{g}\ \text{if $1\leqslant u\leqslant 5$}, \\
&\frac{8-u-3v}{3}\mathbf{g}\ \text{if $5\leqslant u\leqslant 8$}.\\
\endaligned
\right.
$$
Now, using \cite[Theorem~4.8]{Fujita2021}, we get
\begin{multline*}
S\big(W^{G}_{\bullet,\bullet};\ell\big)=
\frac{3}{8}\int_0^8\int_0^\infty \mathrm{vol}\big(P(u)\big\vert_{G}-v\ell\big)dvdu=\frac{3}{8}\int_{0}^{1}\int_{0}^{u}\frac{(u-v)^{2}}{2}dvdu+\\
+\frac{3}{12}\int_{1}^{5}\int_{0}^{1}\frac{(1-v)^{2}}{4}dvdu+\frac{3}{12}\int_{5}^{8}\int_{0}^{\frac{8-u}{3}}\frac{(8-u-3v)^2}{18}dvdu=\frac{5}{16}.
\end{multline*}
Similarly, if $Q\not\in\mathscr{C}$, then it follows from \cite[Theorem~4.17]{Fujita2021} that
\begin{multline*}
S\big(W_{\bullet,\bullet,\bullet}^{G,\ell};Q\big)=
\frac{3}{8}\int_{0}^{1}\int_{0}^{u}\Big(\big(P(u)\big\vert_{G}-v\ell\big)\cdot\ell\Big)^2dvdu+
\frac{3}{8}\int_{1}^{5}\int_{0}^{1}\Big(\big(P(u)\big\vert_{G}-v\ell\big)\cdot\ell\Big)^2dvdu+\\
+\frac{3}{8}\int_{5}^{8}\int_{0}^{\frac{8-u}{3}}\Big(\big(P(u)\big\vert_{G}-v\ell\big)\cdot\ell\Big)^2dvdu+F_Q=
\frac{3}{8}\int_{0}^{1}\int_{0}^{u}\frac{(u-v)^{2}}{4}dvdu+\\
+\frac{3}{8}\int_{1}^{5}\int_{0}^{1}\frac{(1-v)^{2}}{4}dvdu+
\frac{3}{8}\int_{5}^{8}\int_{0}^{\frac{8-u}{3}}\frac{(8-u-3v)^2}{36}dvdu=\frac{5}{32}.
\end{multline*}
so that $S(W_{\bullet,\bullet,\bullet}^{G,\ell};Q)=\frac{5}{32}$,
which implies that $\delta_{Q}(G,V^G_{\bullet,\bullet})\geqslant\frac{16}{5}$. Likewise, if $Q\in\mathscr{C}$, then
\begin{multline*}
S\big(W^{G}_{\bullet,\bullet};\mathscr{C}\big)=
\frac{3}{8}\int_{0}^{8}\big(P(u)^{2}\cdot G\big)\cdot\mathrm{ord}_{\mathscr{C}}\Big(N(u)\big\vert_{G}\Big)du+
\frac{3}{8}\int_0^8\int_0^\infty \mathrm{vol}\big(P(u)\big\vert_{G}-v\mathscr{C}\big)dvdu=\\
=\frac{3}{8}\int_{1}^{5}\frac{u-1}{8}du+\frac{3}{8}\int_{5}^{8}\frac{(u-2)(8-u)^{2}}{54}du+\frac{3}{8}\int_0^1\int_{0}^{\frac{u}{4}}\frac{(u-4v)^{2}}{2}dvdu+\\
+\frac{3}{8}\int_{1}^{5}\int_{0}^{\frac{1}{4}} \frac{(1-4v)^{2}}{2}dvdu+\frac{3}{8}\int_{5}^{8}\int_{0}^{\frac{8-u}{12}}\frac{(8-u-12v)^{2}}{18}dvdu=\frac{11}{16}
\end{multline*}
and
\begin{multline*}
\quad\quad\quad\quad S\big(W_{\bullet,\bullet,\bullet}^{G,\mathscr{C}};Q\big)=
\frac{3}{8}\int_{0}^{1}\int_{0}^{\frac{u}{4}}\Big(\big(P(u)\big\vert_{G}-v\mathscr{C}\big)\cdot\mathscr{C}\Big)^2dvdu+\\
+\frac{3}{8}\int_{1}^{5}\int_{0}^{\frac{1}{4}}\Big(\big(P(u)\big\vert_{G}-v\mathscr{C}\big)\cdot\mathscr{C}\Big)^2dvdu+\frac{3}{8}\int_{5}^{8}\int_{0}^{\frac{8-u}{12}}\Big(\big(P(u)\big\vert_{G}-v\mathscr{C}\big)\cdot\mathscr{C}\Big)^2dvdu=\\
=\frac{3}{8}\int_{0}^{1} \int_{0}^{\frac{u}{4}}(2u-8v)^{2}dvdu+\frac{3}{8}\int_{1}^{5} \int_{0}^{\frac{1}{4}}(2-8v)^{2}dvdu+\\
+\frac{3}{8}\int_{5}^{8}\int_{0}^{\frac{8-u}{12}}(\frac{(16-2u-24v)^2}{9}dvdu=\frac{5}{8}.\quad\quad\quad\quad\quad
\end{multline*}
This implies that $\delta_{Q}(G,V^G_{\bullet,\bullet})\geqslant\min\left\{\frac{16}{11},\frac{8}{5}\right\}=\frac{16}{11}$,
which is a~contradiction.
\end{proof}

\begin{corollary}
\label{corollary:2-3}
All smooth Fano threefolds in the~family \textnumero 2.3 are K-stable.
\end{corollary}

We see that $d=3$, so that $A$ is a singular cubic surface in $\mathbb{P}^3$ such that $P=\mathrm{Sing}(A)$.
Let $\sigma\colon\widehat{X}\to X$ be the~blow up of the~point $P$, and let $G$ be the~$\sigma$-exceptional surface.
Denote by $\widehat{A}$ and $\widehat{E}$ the~proper transforms on $\widehat{X}$ of the~surfaces $A$ and $E$, respectively.

\begin{lemma}
\label{lemma:2-1-d-3}
The surface $A$ has Du Val singularities.
\end{lemma}

\begin{proof}
Suppose that $A$ is a~cone in $\mathbb{P}^3$ with vertex $P$. Take $u\in\mathbb{R}_{\geqslant 0}$. Then
$$
\sigma^*(-K_X)-vG\sim_{\mathbb{R}}2\widehat{A}+\widehat{E}+(6-u)G.
$$
Thus, the~divisor $\sigma^*(-K_X)-vG$ is pseudo-effective $\iff$ $u\in[0,6]$.
Moreover, if $u\in[0,6]$, then the~Zariski decomposition of the~divisor $\sigma^*(-K_X)-vG$ can be described as follows:
$$
P(u)=\left\{\aligned
&2\widehat{A}+E+(6-u)G\ \text{if $0\leqslant u\leqslant 1$}, \\
&\frac{7-u}{3}\widehat{A}+\widehat{E}+(6-u)G\ \text{if $1\leqslant u\leqslant 4$}, \\
&\frac{6-u}{2}(\widehat{A}+\widehat{E}+2G)\ \text{if $4\leqslant u\leqslant 6$}, \\
\endaligned
\right.
$$
and
$$
N(u)=\left\{\aligned
&0\ \text{if $0\leqslant u\leqslant 1$}, \\
&\frac{u-1}{3}\widehat{A}\ \text{if $1\leqslant u\leqslant 4$}, \\
&\frac{u-2}{2}\widehat{A}+\frac{u-4}{2}\widehat{E}\ \text{if $4\leqslant u\leqslant 6$}, \\
\endaligned
\right.
$$
where $P(u)$ and $N(u)$ are the~positive and the~negative parts of the~Zariski decomposition, respectively.
Using this, we compute
$$
S_{X}(G)=\frac{3}{12}\int_{0}^{1}u^{3}du+\frac{3}{12}\int_{1}^{4}udu+\frac{3}{12}\int_{4}^{6}\frac{u(6-u)^{2}}{4}du=\frac{43}{16}<3=A_X(G).
$$

Let us apply \cite[Theorem~4.8]{Fujita2021},  \mbox{\cite[Corollary~4.17]{Fujita2021}}, \cite[Corollary~4.18]{Fujita2021}
using notations introduced in \cite[\S~4]{Fujita2021}. To start with, we apply \cite[Corollary~4.18]{Fujita2021} to get
\begin{equation}
\label{equation:2-5-d-3-cone}
1\geqslant \frac{A_X(\mathbf{F})}{S_X(\mathbf{F})}\geqslant\min\left\{\frac{A_{X}(G)}{S_{X}(G)},
\inf_{Q\in G}\delta_{Q}\big(G,V^G_{\bullet,\bullet}\big)\right\}=\min\left\{\frac{48}{43},\inf_{Q\in G}\delta_{Q}\big(G,V^G_{\bullet,\bullet}\big)\right\},
\end{equation}
where $\delta_{Q}(G,V^G_{\bullet,\bullet})$ is defined in \cite[\S~4]{Fujita2021}.
Let $Q$ be an~arbitrary~point in the surface $G$, and let $\ell$ is a~general line in $G\cong\mathbb{P}^2$ that contains~$Q$.
Then  \cite[Corollary~4.18]{Fujita2021} gives
$$
\delta_{Q}\big(G,V^G_{\bullet,\bullet}\big)
\geqslant\min\left\{\frac{1}{S\big(V^G_{\bullet,\bullet};\ell\big)},\frac{1}{S\big(W^{G,\ell}_{\bullet,\bullet,\bullet};Q\big)}\right\},
$$
where $S(V^F_{\bullet,\bullet};\ell)$ and $S(W^{G,\ell}_{\bullet,\bullet,\bullet};Q)$ are defined in \cite[\S~4]{Fujita2021}.
Take $v\in\mathbb{R}_{\geqslant 0}$. Then
$$
P(u)\big\vert_{G}-v\ell\sim_{\mathbb{R}}\left\{\aligned
&(u-v)\ell\ \text{if $0\leqslant u\leqslant 1$}, \\
&(1-v)\ell\ \text{if $1\leqslant u\leqslant 4$}, \\
&\frac{6-u-2v}{2}\ell\ \text{if $4\leqslant u\leqslant 6$}.\\
\endaligned
\right.
$$
Let $\mathscr{C}=\widehat{A}\vert_{G}$. Then  $\mathscr{C}$ is a~smooth cubic curve in $G\cong\mathbb{P}^2$. Let
$$
N^\prime(u)=N(u)\big\vert_{G}=\left\{\aligned
&0\ \text{if $0\leqslant u\leqslant 1$}, \\
&\frac{u-1}{3}\mathscr{C}\ \text{if $1\leqslant u\leqslant 4$}, \\
&\frac{u-2}{2}\mathscr{C}\ \text{if $4\leqslant u\leqslant 6$}.
\endaligned
\right.
$$
Now, using \cite[Theorem~4.8]{Fujita2021}, we get
\begin{multline*}
S\big(W^{G}_{\bullet,\bullet};\ell\big)=\frac{3}{12}\int_0^6\int_0^\infty \mathrm{vol}\big(P(u)\big\vert_{G}-v\ell\big)dvdu=\frac{3}{12}\int_{0}^{1}\int_{0}^{u}(u-v)^{2}dvdu+\\
+\frac{3}{12}\int_{1}^{4}\int_{0}^{1}(1-v)^{2}dvdu+\frac{3}{12}\int_{4}^{6}\int_{0}^{\frac{6-u}{2}}\left(\frac{6-u-2v}{2}\right)^{2}dvdu=\frac{5}{16}.
\end{multline*}
Similarly, it follows from \cite[Theorem~4.17]{Fujita2021} that
\begin{multline*}
S\big(W_{\bullet,\bullet,\bullet}^{G,\ell};Q\big)=
\frac{3}{12}\int_{0}^{1}\int_{0}^{u}\Big(\big(P(u)\big\vert_{G}-v\ell\big)\cdot\ell\Big)^2dvdu+
\frac{3}{12}\int_{1}^{4}\int_{0}^{1}\Big(\big(P(u)\big\vert_{G}-v\ell\big)\cdot\ell\Big)^2dvdu+\\
+\frac{3}{12}\int_{4}^{6}\int_{0}^{\frac{6-u}{2}}\Big(\big(P(u)\big\vert_{G}-v\ell\big)\cdot\ell\Big)^2dvdu+F_Q=\frac{3}{12}\int_{0}^{1}\int_{0}^{u}(u-v)^{2}dvdu+\\
+\frac{3}{12}\int_{1}^{4}\int_{0}^{1}(1-v)^{2}dvdu+\frac{3}{12}\int_{4}^{6}\int_{0}^{\frac{6-u}{2}}\left(\frac{6-u-2v}{2}\right)^{2}dvdu+F_Q=\frac{5}{16}+F_Q,
\end{multline*}
where
\begin{multline*}
F_Q=\frac{6}{12}\int_{1}^{4}\int_{0}^{1}\Big(\big(P(u)\big\vert_{G}-v\ell\big)\cdot\ell\Big)\mathrm{ord}_Q\big(N^\prime(u)\big|_\ell\big)dvdu+\\
+\frac{6}{12}\int_{4}^{6} \int_{0}^{\frac{6-u}{2}}\Big(\big(P(u)\big\vert_{G}-v\ell\big)\cdot\ell\Big)\mathrm{ord}_Q\big(N^\prime(u)\big|_\ell\big)dvdu\leqslant\\
\leqslant\frac{6}{12}\int_{1}^{4} \int_{0}^{1}\frac{(1-v)(u-1)}{3}dvdu+
\frac{6}{12}\int_{4}^{6} \int_{0}^{\frac{6-u}{2}}\frac{(6-u-2v)(u-2)}{4}dvdu=\frac{7}{12}.
\end{multline*}
So, we have $S(W_{\bullet,\bullet,\bullet}^{G,\ell};Q)\leqslant\frac{43}{48}$.
Then  $\delta_{Q}(G,V^G_{\bullet,\bullet})>1$, which contradicts \eqref{equation:2-5-d-3-cone}.
\end{proof}

Thus, we see that $P$ is a Du Val singular point of the surface $A$ of type $\mathbb{D}_4$, $\mathbb{D}_5$, $\mathbb{E}_6$.
Now, arguing as in the proof of \cite[Lemma~9.11]{Xu}, we see that $\beta(\mathbf{F})>0$ if
\begin{enumerate}
\item the~inequality $\beta(G)>0$ holds,
\item and for every prime divisor $\mathbf{E}$ over $X$ such that $C_X(\mathbf{E})$ is a curve containing $P$, the following inequality holds:
$$
\frac{A_X(\mathbf{E})}{S_X(\mathbf{E})}\geqslant\frac{4}{3}.
$$
\end{enumerate}
Since $\beta(\mathbf{F})\leqslant 0$ by our assumption, we see that at least one of these conditions must fail.

\begin{lemma}
\label{lemma:2-5-beta-blow-up}
One has $\beta(G)\geqslant\frac{465}{2048}$.
\end{lemma}

\begin{proof}
Let $\widehat{A}$ and $\widehat{E}$ be the proper transforms on $\widehat{X}$ of the surfaces $A$ and $E$, respectively.
Take $u\in\mathbb{R}_{\geqslant 0}$. Then
$$
\sigma^*(-K_{X})-uG\sim\sigma^*(2H-E)-uG\sim\sigma^*(2A+E)-uG\sim 2\widehat{A}+\widehat{E}+(4-u)G,
$$
which easily implies that the divisor $-K_{\widehat{X}}-G$ is pseudoeffective $\iff$ $u\leqslant 4$,
because we can contract the surfaces $\widehat{A}$ and $\widehat{E}$ simultaneously after flops. Then
$$
\beta(G)=A_X(G)-S_X(G)=3-\frac{1}{12}\int_0^4\mathrm{vol}\big(\sigma^*(-K_{X})-uG\big)du.
$$

Note that $\sigma^*(-K_{X})-uG$ is nef for $u\in[0,1]$, because the divisor $-K_X$ is very ample.
Thus, if $u\in[0,1]$, then
$$
\mathrm{vol}\big(\sigma^*(-K_{X})-uG\big)=\big(\sigma^*(-K_{X})-uG\big)^3=12-u^3.
$$
Similarly, if $1\leqslant u\leqslant\frac{3}{2}$, then
$$
\mathrm{vol}\big(\sigma^*(-K_{X})-uG\big)\leqslant\mathrm{vol}\big(\sigma^*(-K_{X})-G\big)=\big(\sigma^*(-K_{X})-G\big)^3=11.
$$
Finally, let us estimate $\mathrm{vol}(\sigma^*(-K_{X})-uG)$ in the case when $4\geqslant u>\frac{3}{2}$.

Let $Z$ be a general hyperplane section of the cubic surface $A$ that passes through $P$,
and let $\widehat{Z}$ be its proper transform on the threefold $\widehat{X}$.
Then  $Z$ is an irreducible cuspidal cubic curve, and $\widehat{Z}\subset\widehat{A}$.
Observe that $(\sigma^*(-K_{X})-uG)\cdot\widehat{Z}=3-2u$ and $\widehat{A}\cdot\widehat{Z}=-4$,
so $\widehat{A}$ is contained in the asymptotic base locus of the divisor $\sigma^*(-K_{X})-uG$ for $u>\frac{3}{2}$.
Moreover, if $\sigma^*(-K_{X})-uG\sim_{\mathbb{R}}\widehat{D}+\lambda \widehat{A}$ for $\lambda\in\mathbb{R}_{\geqslant 0}$
and an effective $\mathbb{R}$-divisor $\widehat{D}$ whose support does not contain $\widehat{A}$,
then $\widehat{Z}\not\subset\widehat{D}$, which implies that
$$
0\leqslant \widehat{D}\cdot\widehat{Z}=\Big(\sigma^*(-K_{X})-uG-\lambda \widehat{A}\Big)\cdot\widehat{Z}=3-2u-4\lambda,
$$
so that $\lambda\geqslant\frac{3-2u}{4}$. Thus, if $4\geqslant u>\frac{3}{2}$, then
$$
\mathrm{vol}\big(\sigma^*(-K_{X})-uG\big)\leqslant\mathrm{vol}\Big(\sigma^*(2H-E)-uG-\frac{2u-3}{4}\widehat{A}\Big).
$$
Moreover, if $4\geqslant u>\frac{3}{2}$, then
$$
\sigma^*(2H-E)-uG-\frac{2u-3}{4}\widehat{A}\sim_{\mathbb{R}}\frac{11-2u}{4}\sigma^{*}(H)-\frac{7-2u}{2}\sigma^*(E)-\frac{3}{2}G.
$$
Therefore, if $4\geqslant u>\frac{3}{2}$, then
$$
\mathrm{vol}\big(\sigma^*(-K_{X})-uG\big)\leqslant\mathrm{vol}\Big(\frac{11-2u}{4}\sigma^{*}(H)-\frac{7-2u}{2}\sigma^*(E)\Big)=\mathrm{vol}\Big(\frac{11-2u}{4}H-\frac{7-2u}{2}E\Big).
$$
Furthermore, if $\frac{7}{2}\geqslant u>\frac{3}{2}$, then $\frac{11-2u}{4}H-\frac{7-2u}{2}E$ is nef, so that
$$
\mathrm{vol}\Big(\frac{11-2u}{4}H-\frac{7-2u}{2}E\Big)=\Big(\frac{11-2u}{4}H-\frac{7-2u}{2}E\Big)^3=\frac{25-6u}{16}.
$$
Similarly, if $4\geqslant u>\frac{7}{2}$, then
$$
\mathrm{vol}\Big(\frac{11-2u}{4}H-\frac{7-2u}{2}E\Big)=\Big(\frac{11-2u}{4}H\Big)^3=\frac{(11-2u)^3}{256}.
$$

Now, we can estimate $\beta(G)$ as follows
\begin{multline*}
\beta(G)=3-\frac{1}{12}\int_0^4\mathrm{vol}\big(\sigma^*(-K_{X})-uG\big)du\geqslant
3-\frac{1}{12}\int_0^{1}(12-u^3)du-\frac{1}{12}\int_1^{\frac{3}{2}}11du-\\
-\frac{1}{12}\int_{\frac{3}{2}}^{\frac{7}{2}}\frac{25-6u}{16}du-\frac{1}{12}\int_{\frac{7}{2}}^4\frac{(11-2u)^3}{256}du=3-\frac{5679}{2048}=\frac{465}{2048}\quad\quad\quad\quad\quad
\end{multline*}
as claimed.
\end{proof}

Therefore, there exists a prime divisor $\mathbf{E}$ over $X$ such that $C_X(\mathbf{E})$ is a curve, $P\in C_X(\mathbf{E})$, and
$A_X(\mathbf{E})<\frac{4}{3}S_X(\mathbf{E})$. Set $Z=C_X(\mathbf{E})$. Then  $\delta_O(X)<\frac{4}{3}$ for every point $O\in Z$.

\begin{lemma}
\label{lemma:2-5-Z-curve}
One has $Z\subset A$, and $Z$ is a line in the cubic surface $A$.
\end{lemma}

\begin{proof}
Let $O$ be a general point in $Z$, and let $A_O$ be the fiber of $\phi$ that passes through $O$.
If $Z\not\subset A$, then $A_O$ is smooth, so that $\delta_O(A_O)\geqslant\frac{3}{2}$ by \cite[Lemma~2.13]{Book},
which gives
$$
\frac{4}{3}>\frac{A_X(\mathbf{E})}{S_X(\mathbf{E})}\geqslant\delta_O(X)\geqslant\mathrm{min}\Bigg\{\frac{16}{11},\frac{16\delta_O(A_O)}{\delta_O(A_O)+15}\Bigg\}\geqslant\mathrm{min}\Bigg\{\frac{16}{11},\frac{16\times\frac{3}{2}}{\frac{3}{2}+15}\Bigg\}=\frac{16}{11}>\frac{4}{3},
$$
by Lemma~\ref{lemma:2-1-2-3-2-5-del-Pezzo-delta}. This shows that $Z\subset A$ and $A_O=A$.

To complete the proof of the lemma, we have to show that $Z$ is a line in the surface~$A$.
Suppose that $Z$ is not a line. Then  the point $O$ is not contained in a line in the surface $A$, because $A$ contains finitely many lines \cite{BruceWall}.
Now, arguing as in the proof of \cite[Lemma 2.13]{Book},
we get  $\delta_O(A)\geqslant\frac{3}{2}$.
So, applying Lemma~\ref{lemma:2-1-2-3-2-5-del-Pezzo-delta} again, we get a contradiction as above.
\end{proof}

Now, our Auxiliary Theorem follows from the following lemma: % D5 $tz^2+zx^2+y^2x=0$

\begin{lemma}
\label{lemma:2-5-D-4}
The surface $A$ does not have a singular point of type $\mathbb{D}_4$.
\end{lemma}

\begin{proof}
Suppose $A$ has singularity of type $\mathbb{D}_4$.
Then  it follows from \cite{BruceWall} that,
for a suitable choice of coordinates $x$, $y$, $z$, $t$ on the~projective space $\mathbb{P}^3$,
one of the following cases hold:
\begin{itemize}
\item[(A)] $A=\{tx^2=y^3-z^3\}\subset \mathbb{P}^3$,
\item[(B)] $A=\{tx^2=y^3-z^3+xyz\}\subset \mathbb{P}^3$.
\end{itemize}
Note that $P=[0:0:0:1]$, and $A$ contains $6$ lines  \cite{BruceWall}.
In case (A), these lines are
\begin{align*}
L_1&=\{x=y-z=0\},\\
L_2&=\{x=y-\omega_3z=0\},\\
L_3&=\{x=y+\omega_3^2z=0\},\\
L_4&=\{t=y-z=0\},\\
L_5&=\{t=y+\omega_3z=0\},\\
L_6&=\{t=y+\omega_3^2z=0\},
\end{align*}
where $\omega_3$ is a primitive cube root of unity. In case (B), these lines are
\begin{align*}
L_1&=\{x=y-z=0\},\\
L_2&=\{x=y-\omega_3z=0\},\\
L_3&=\{x=y-\omega_3^2z=0\},\\
L_4&=\{x+3(y-z)=y-z-9t=0\},\\
L_5&=\{x+3\omega_3(y-\omega_3z)=\omega_3y-\omega_3^2z-9t=0\},\\
L_6&=\{x+3\omega_3^2(y-\omega_3^2z)=\omega_3^2y-\omega_3z-9t=0\}.
\end{align*}
Note that $P=L_1\cap L_3\cap L_3$, $P\not\in L_4\cup L_5\cup L_6$ and
$-K_A\sim 2L_1+L_4\sim 2L_2+L_5\sim 2L_3+L_6$.

By Lemma~\ref{lemma:2-5-Z-curve}, we may assume that $Z=L_1$.

Recall that $S_X(A)=\frac{11}{16}$, see the proof of Lemma~\ref{lemma:2-1-2-3-2-5-del-Pezzo-delta}.
Using \cite[Theorem~1.112]{Book}, we get
$$
\frac{4}{3}>\frac{A_X(\mathbf{E})}{S_X(\mathbf{E})}\geqslant \min\left\{\frac{1}{S_X(A)},\frac{1}{S(W_{\bullet,\bullet}^A;L_1)}\right\}=\min\left\{\frac{16}{11},\frac{1}{S(W_{\bullet,\bullet}^A;L_1)}\right\},
$$
where $S(W_{\bullet,\bullet}^A;L_1)$ is defined in \cite[\S~1.7]{Book}. Therefore, we conclude that $S(W_{\bullet,\bullet}^A;L_1)<\frac{4}{3}$.
Let us compute $S(W_{\bullet,\bullet}^A;L_1)$ using \cite[Corollary~1.109]{Book}.

To do this, we use notations introduced in the proof of Lemma~\ref{lemma:2-1-2-3-2-5-del-Pezzo-delta} applied to $O=P$.
Then  using \cite[Corollary~1.109]{Book} and computations from the proof of Lemma~\ref{lemma:2-1-2-3-2-5-del-Pezzo-delta}, we get
$$
S\big(W_{\bullet,\bullet}^A; L_1\big)=\frac{1}{4}\int_{0}^{1}\int_0^\infty \mathrm{vol}\big(-K_A-vL_1\big)dvdu+\frac{1}{4}\int_{1}^{2}\int_0^\infty \mathrm{vol}\big((2-u)(-K_A)-vL_1\big)dvdu,
$$
since $L_1\not\subset\mathrm{Supp}(N(u))$, since $L_1\not\subset E$. Let us compute $S(W_{\bullet,\bullet}^A; L_1)$. Take $v\in\mathbb{R}_{\geqslant 0}$.
Then
$$
-K_A-vL_1\sim_{\mathbb{R}}(2-v)L_1+L_4.
$$
Thus, the divisor $-K_A-vL_1$ is pseudoeffective $\iff$ $v\leqslant 2$, since $L_4^2=-1$. Fix $v\in[0,2]$.
Let $P(u,v)$ be the~positive part of the Zariski decomposition of the divisor $-K_A-vL_1$,
and let $N(u,v)$ be its negative part. Then
$$
P(u,v)=\left\{\aligned
&(2-v)L_1+L_4\ \text{if $0\leqslant v\leqslant 1$}, \\
&(2-v)(L_1+L_4)\ \text{if $1\leqslant v\leqslant 2$}, \\
\endaligned
\right.
$$
and
$$
N(u,v)=\left\{\aligned
&0\ \text{if $0\leqslant v\leqslant 1$}, \\
&(v-1)L_4\ \text{if $1\leqslant v\leqslant 2$}.
\endaligned
\right.
$$
Thus, if $0\leqslant v\leqslant 1$, then $\mathrm{vol}(-K_A-vL_1)=3-2v$, because $L_1^2=0$ and $L_1\cdot L_4=0$.
Similarly, if $1\leqslant v\leqslant 2$, then $\mathrm{vol}(-K_A-vL_1)=(v-2)^2$. This gives
$$
\frac{1}{4}\int_{0}^{1}\int_0^\infty \mathrm{vol}\big(-K_A-vL_1\big)dvdu=
\frac{1}{4}\int_{0}^{1}\int_0^1(3-2v)dvdu+\frac{1}{4}\int_{0}^{1}\int_1^2(v-2)^2dvdu=\frac{7}{12}
$$
and
\begin{multline*}
\frac{1}{4}\int_{1}^{2}\int_0^\infty \mathrm{vol}\big((2-u)(-K_A)-vL_1\big)dvdu=
\frac{1}{4}\int_{1}^{2}\int_0^\infty(2-u)^3\mathrm{vol}\big((-K_A)-vL_1\big)dvdu=\\
=\frac{1}{4}\int_{1}^{2}\int_0^1(2-u)^3(3-2v)dvdu+\frac{1}{4}\int_{1}^{2}\int_1^2(2-u)^3(v-2)^2dvdu=\frac{7}{48}.\quad\quad\quad
\end{multline*}
Combining, we get $S(W_{\bullet,\bullet}^A;L_1)=\frac{35}{48}<\frac{3}{4}$. This is a contradiction.
\end{proof}

\section{Family \textnumero 2.2.}
\label{section:2-2}

Let $R$ be a~smooth surface of degree $(2,4)$ in $\mathbb{P}^1\times\mathbb{P}^2$,
let $\pi\colon X\to \mathbb{P}^1\times\mathbb{P}^2$ be a~double cover ramified over the~surface $R$.
Then  $X$ is a~smooth Fano threefold in the~family~\textnumero~2.2.
Moreover, all smooth Fano threefolds in this family can be obtained this way.

Let $\mathrm{pr}_1\colon\mathbb{P}^1\times\mathbb{P}^2\to\mathbb{P}^1$
and $\mathrm{pr}_2\colon\mathbb{P}^1\times\mathbb{P}^2\to\mathbb{P}^2$ be the~projections
to the~first and the~second factors, respectively.
Set $p_1=\mathrm{pr}_1\circ\pi$ and $p_2=\mathrm{pr}_2\circ\pi$.
We have the~following commutative diagram:
$$\xymatrix{& X\ar[d]^{\pi}\ar@/^-1.7pc/[ddl]_{p_1}\ar@/^1.7pc/[ddr]^{p_2}&\\
&\mathbb{P}^1\times\mathbb{P}^2\ar[dl]_{\mathrm{pr}_1}\ar[dr]^{\mathrm{pr}_2}&\\
\mathbb{P}^1 && \mathbb{P}^2}
$$
where $p_1$ is a~fibration into del Pezzo surfaces of degree $2$,
and $p_2$ is a~conic bundle.

\begin{lemma}
\label{2-2-normal-surface}
Let $S$ be a~fiber of the~morphism $p_1$. Then  $S$ is irreducible and normal.
\end{lemma}

\begin{proof}
Since $p_1\colon X\to\mathbb{P}^1$ is a Mori fiber space, any fiber of $p_1$ is
irreducible and reduced. Moreover, any fiber of $\mathrm{pr}_1|_R$ is reduced
by local computations. Thus, the assertion follows.
\end{proof}

\begin{lemma}
\label{2-2-dP-conic-bundle}
Let $S$ be a~fiber of the~morphism $p_1$, let $C$ be a~fiber of the~morphism $p_2$,
and let $P$ be a~point in $S\cap C$.
Then  $S$ or $C$ is smooth at $P$.
\end{lemma}

\begin{proof}
Local computations.
\end{proof}

Now, we are ready to prove that $X$ is K-stable.
Recall from \cite{CheltsovShramovPrzyjalkowski} that $\mathrm{Aut}(X)$ is finite.
Thus, the~threefold $X$ is K-stable if and only if it is K-polystable \cite{Xu}.

Let $\tau$ be the~Galois involution of the~double cover $\pi\colon X\to\mathbb{P}^1\times\mathbb{P}^2$, and let $G=\langle\tau\rangle$.
Suppose that $X$ is not $K$-polystable.
Then  it follows from \cite[Corollary~4.14]{Zhuang} that there exists a~$G$-invariant prime divisor $\mathbf{F}$ over $X$ such that
$$
\beta(\mathbf{F})=A_X(\mathbf{F})-S_X(\mathbf{F})\leqslant 0.
$$
Let $Z$ be the~center of this~divisor on $X$. Then  $Z$ is not  a~surface by \cite[Theorem~3.17]{Book}.
Hence, we see that either $Z$ is a~$G$-invariant irreducible curve, or $Z$ is a~$G$-fixed point.
Let us seek for a~contradiction.

Let $P$ be a~general point in $Z$, and let $S$ be the~fiber of $p_1$ that passes through $P$.

\begin{lemma}
\label{lemma:2-2-S-singular}
The surface $S$ is singular at $P$.
\end{lemma}

\begin{proof}
Suppose that $S$ is smooth at $P$.
Let $B$ be a~general curve in $|-K_S|$ that contains~$P$.
Then  $B$ is a smooth curve in $S$.
Applying \cite[Theorem~1.112]{Book}, we~get
$$
1\geqslant \frac{A_X(\mathbf{F})}{S_X(\mathbf{F})}\geqslant \min\left\{\frac{1}{S_X(S)},\frac{1}{S(W_{\bullet,\bullet}^S;B)},\frac{1}{S(W_{\bullet, \bullet,\bullet}^{S,B};P)}\right\}.
$$
Since~\mbox{$S_X(S)<1$} by \cite[Theorem~3.17]{Book},
we see that $S(W_{\bullet,\bullet}^S;B)\geqslant 1$ or $S(W_{\bullet, \bullet,\bullet}^{S,B};P)\geqslant 1$.
We refer the~reader to \cite[\S~1.7]{Book} for definitions of $S_X(S)$, $S(W_{\bullet,\bullet}^S;B)$, $S(W_{\bullet, \bullet,\bullet}^{S,B};P)$.

Note that \cite[Theorem~1.112]{Book} requires $S$ to have Du Val singularities, but $S$ may have non-Du Val singularities.
Nevertheless, we still can apply \cite[Theorem~1.112]{Book} here, since the~proof of \cite[Theorem 1.112]{Book}
remains valid in our case, because $S$ is smooth along~$B$.

Let us compute $S(W_{\bullet,\bullet}^S;B)$ and $S(W_{\bullet, \bullet,\bullet}^{S,B};P)$.
Take $u\in\mathbb{R}_{\geqslant 0}$ and $v\in\mathbb{R}_{\geqslant 0}$.
Then
\begin{center}
$-K_X-uS$ is nef $\iff$ $-K_X-uS$ is pseudoeffective $\iff$ $u\leqslant 1$,
\end{center}
Similarly, if $u\in[0,1]$, then $(-K_X-uS)\vert_{S}-vB\sim_{\mathbb{R}}(1-v)(-K_S)$, so
\begin{center}
$(-K_X-uS)\vert_{S}-vB$ is nef $\iff$ $(-K_X-uS)\vert_{S}-vB$ is pseudoeffective $\iff$ $v\leqslant 1$.
\end{center}
Now, applying \cite[Corollary~1.109]{Book}, we get
$$
S(W_{\bullet,\bullet}^S;B)=\frac{3}{6}\int_0^1\int_0^1\big((1-v)(-K_S)\big)^2dvdu=\frac{1}{2}\int_0^1\int_0^12(1-v)^2dv=\frac{1}{3}<1.
$$
Similarly, using \cite[Theorem~1.112]{Book}, we get
$$
S(W_{\bullet,\bullet,\bullet}^{S,B},P)=\frac{3}{6}\int_0^1\int_0^1\big((1-v)(-K_S)\cdot B\big)^2dvdu=\frac{2}{3}<1.
$$
But we already know that $S(W_{\bullet,\bullet}^S;B)\geqslant 1$ or $S(W_{\bullet, \bullet,\bullet}^{S,B};P)\geqslant 1$.
This is a~contradiction.
\end{proof}

If $Z$ is a~curve, then $S$ is smooth at $P$ by Lemma \ref{2-2-normal-surface},
because $P$ is a~general point in~$Z$.
Hence, we conclude that $Z=P$, because $S$ is singular at the~point $P$ by Lemma~\ref{lemma:2-2-S-singular}.
Recall that $Z$ is $G$-invariant. This implies that $\tau(P)\in R$.

Let $C$ be the~fiber of $p_2$ that passes through $P$.
Then  $C$ is smooth at $P$ by Lemma~\ref{2-2-dP-conic-bundle}, because $S$ is singular at $P$.
Since $\tau(P)\in R$, we see that $C$ is irreducible and smooth.

Let $T$ be a~sufficiently general surface in linear system $|p_2^*(\mathcal{O}_{\mathbb{P}^2}(1))|$ that contains $C$.
Since $C$ is smooth, it follows from Bertini's theorem that the~surface $T$ is smooth.

As in the~proof of Lemma~\ref{lemma:2-2-S-singular},
it follows from \cite[Theorem~1.112]{Book} that
$$
1\geqslant \frac{A_X(\mathbf{F})}{S_X(\mathbf{F})}\geqslant\min\left\{\frac{1}{S_X(T)},\frac{1}{S(W_{\bullet,\bullet}^T;C)},\frac{1}{S(W_{\bullet, \bullet,\bullet}^{T,C};P)}\right\}.
$$
Moreover, it follows from \cite[Theorem~3.17]{Book} that \mbox{$S_X(T)<1$}.
Thus, we conclude that
$$
\max\big\{S(W_{\bullet,\bullet}^T;C),S(W_{\bullet, \bullet,\bullet}^{T,C};P)\big\}\geqslant 1.
$$
In fact, since $P$ is the~center of the~divisor $\mathbf{F}$ on $X$,
\cite[Theorem~3.17]{Book} gives
\begin{equation}
\label{equation:2-2-strict}
\max\big\{S(W_{\bullet,\bullet}^T;C),S(W_{\bullet, \bullet,\bullet}^{T,C};P)\big\}>1.
\end{equation}
Now, let us compute $S(W_{\bullet,\bullet}^T;C)$ and $S(W_{\bullet, \bullet,\bullet}^{T,C};P)$ using the~results obtained in \cite[\S~1.7]{Book}.

Take $u\in\mathbb{R}_{\geqslant 0}$ and $v\in\mathbb{R}_{\geqslant 0}$. Then
\begin{center}
$-K_X-uT$ is nef $\iff$ $-K_X-uT$ is pseudoeffective $\iff$ $u\leqslant 1$,
\end{center}
Similarly, if $u\in[0,1]$, then
\begin{center}
$(-K_X-uT)\vert_{T}-vC$ is nef $\iff$ $(-K_X-uT)\vert_{T}-vC$ is pseudoeffective $\iff$ $v\leqslant 1-u$,
\end{center}
because  $(-K_X-uT)\vert_{T}-vC\sim_{\mathbb{R}}S|_T+(1-u-v)C$.
So, using \cite[Corollary~1.109]{Book},~we~get
$$
S(W_{\bullet,\bullet}^T;C)=\frac{3}{6}\int_0^1\int_0^{1-u}\big(S|_T+(1-u-v)C\big)^2dvdu=\frac{1}{2}\int_0^1\int_0^{1-u}4(1-u-v)dvdu=\frac{1}{3}<1.
$$
Hence, it follows from \eqref{equation:2-2-strict} that $S(W_{\bullet,\bullet,\bullet}^{T,C},P)>1$.
Now, using \cite[Theorem~1.112]{Book}, we get
$$
S(W_{\bullet,\bullet,\bullet}^{T,C},P)=\frac{3}{6}\int_0^1\int_0^{1-u}\Big(\big(S|_T+(1-u-v)C\big)\cdot C\Big)^2dvdu=\frac{3}{6}\int_0^1\int_0^{1-u}4\,dvdu=1,
$$
which is a~contradiction. This shows that $X$ is K-stable.

\begin{corollary}
\label{corollary:2-2-final}
All smooth Fano threefolds in the~family \textnumero 2.2 are K-stable.
\end{corollary}

\section{Family \textnumero 2.4.}
\label{section:2-4}

Let $\mathscr{S}$ and $\mathscr{S}^\prime$ be smooth cubic surfaces in $\mathbb{P}^3$ such that their~intersection is a~smooth curve of genus $10$.
Set~$\mathscr{C}=\mathscr{S}\cap \mathscr{S}^\prime$, and let $\pi\colon X\to \mathbb{P}^3$ be the~blow up of the~curve $\mathscr{C}$.
Then  $X$ is a~smooth Fano threefold in the~family~\textnumero 2.4,
and every smooth Fano threefold in this family can be obtained in this way.
Moreover, there exists a~commutative diagram
$$
\xymatrix{
&X\ar@{->}[ld]_{\pi}\ar@{->}[rd]^{\phi}&\\%
\mathbb{P}^3\ar@{-->}[rr]&&\mathbb{P}^1}
$$
where $\mathbb{P}^3\dasharrow\mathbb{P}^1$ is a~map that is given by the~pencil generated by the~surfaces $\mathscr{S}$~and~$\mathscr{S}^\prime$,
and $\phi$ is a~fibration into cubic surfaces. Note that $-K_X^3=10$ and $\mathrm{Aut}(X)$ is finite \cite{CheltsovShramovPrzyjalkowski}.

Let $H=\pi^*(\mathcal{O}_{\mathbb{P}^3}(1))$, and let $E$ be the~$\pi$-exceptional surface.
Then  $-K_X\sim 4H-E$, the~morphism $\phi$ is given by the~linear system $|3H-E|$, and $E\cong\mathscr{S}\times\mathbb{P}^1$.

The goal of this section is to prove that $X$ is K-stable.
Suppose that $X$ is not K-stable. Let us seek for a~contradiction.
First, using the~valuative criterion for K-stability \cite{Fujita2019Crelle,Li},
we see that there~exists a~prime divisor $\mathbf{F}$ over $X$ such that
$$
\beta(\mathbf{F})=A_X(\mathbf{F})-S_X(\mathbf{F})\leqslant 0.
$$
Let $Z$ be the~center of the~divisor $\mathbf{F}$ on $X$.
Then $Z$ is not a~surface~by~\mbox{\cite[Theorem~3.17]{Book}}.
Therefore, either $Z$ is an~irreducible curve or $Z$ is a~point.
Fix a~point $P\in Z$.

Let $A$ be the~surface in $|3H-E|$ that contains $P$.
Fix $u\in\mathbb{R}_{\geqslant 0}$.
Let $\mathscr{P}(u)$ be the~positive part of the~Zariski decomposition of $-K_X-uA$,
and let $\mathscr{N}(u)$ be its negative part. Then
$$
-K_X-uA\sim_{\mathbb{R}}(4-3u)H-(1-u)E\sim_{\mathbb{R}}\Big(\frac{4}{3}-u\Big)A+\frac{1}{3}E.
$$
This implies that $-K_X-uA$ is pseudoeffective $\iff$ $u\leqslant\frac{4}{3}$.
Moreover, we have
$$
\mathscr{P}(u)=\left\{\aligned
&(4-3u)H-(1-u)E\ \text{if $0\leqslant u\leqslant 1$}, \\
&(4-3u)H\ \text{if $1\leqslant v\leqslant \frac{4}{3}$}, \\
\endaligned
\right.
$$
and
$$
\mathscr{N}(u)=\left\{\aligned
&0\ \text{if $0\leqslant u\leqslant 1$}, \\
&(u-1)E\ \text{if $1\leqslant u\leqslant \frac{4}{3}$}. \\
\endaligned
\right.
$$
Integrating, we obtain $S_X(A)=\frac{67}{120}<1$, which also follows from \mbox{\cite[Theorem~3.17]{Book}}.

Note that $\pi(A)$ is a~normal cubic surface in $\mathbb{P}^3$, and $\pi(A)$ is smooth along the~curve~$\mathscr{C}$.
In particular, we see that $A\cong\pi(A)$, and $A$ is smooth along the~intersection $E\cap A$.

\begin{lemma}
\label{lemma:2-4-A-singular-at-P}
The surface $A$ is singular at the~point $P$.
\end{lemma}

\begin{proof}
Suppose that $A$ is smooth at $P$.
Let $C$ be a~general curve in $|-K_A|$ that passes through the~point $P$.
Then  $C$ is a~smooth irreducible elliptic curve. Take $v\in\mathbb{R}_{\geqslant 0}$.
Then
$$
\mathscr{P}(u)\vert_{S}-vC\sim_{\mathbb{R}}\left\{\aligned
&(1-v)C\ \text{if $0\leqslant u\leqslant 1$}, \\
&(4-3u-v)C\ \text{if $1\leqslant u\leqslant \frac{4}{3}$}.
\endaligned
\right.
$$
Therefore, using \cite[Corollary~1.109]{Book}, we obtain
$$
S\big(W_{\bullet,\bullet}^A;C\big)=\frac{3}{10}\int_0^1\int_0^13(1-v)^2dvdu+\frac{3}{10}\int_1^{\frac{4}{3}}\int_0^{4-3u}3(4-3u-v)^2dvdu=\frac{13}{40}.
$$
Similarly, using \cite[Theorem~1.112]{Book}, we obtain
\begin{multline*}
S\big(W_{\bullet,\bullet,\bullet}^{A,C};P\big)\leqslant
\frac{3}{10}\int_0^1\int_0^1\big(3(1-v)\big)^2dvdu+\frac{3}{10}\int_1^{\frac{4}{3}}\int_0^{4-3u}\big(3(4-3u-v)\big)^2dvdu+\\
+\underbrace{\frac{6}{10}\int_1^{\frac{4}{3}}\int_0^{4-3u}3(4-3u-v)(u-1)dvdu}_{\text{if } P\in E}=\frac{39}{40}+\frac{1}{120}=\frac{59}{60}.
\end{multline*}
Therefore, it follows from \cite[Theorem~1.112]{Book} that
$$
1\geqslant \frac{A_X(\mathbf{F})}{S_X(\mathbf{F})}\geqslant\min\left\{\frac{1}{S_X(A)},\frac{1}{S(W_{\bullet,\bullet}^A;C)},\frac{1}{S(W_{\bullet, \bullet,\bullet}^{A,C};P)}\right\}
\geqslant\min\left\{\frac{120}{67},\frac{40}{13},\frac{60}{59}\right\}>1,
$$
which is absurd.
\end{proof}

\begin{corollary}
\label{corollary:2-4-A-singular-at-P}
The point $P$ is not contained in the~surface $E$.
\end{corollary}

Since $A\cong\pi(A)$, we may consider $A$ as a~cubic surface in $\mathbb{P}^3$.
Then
\begin{itemize}
\item either $\mathrm{mult}_P(A)=2$ and $A$ has Du Val singularities.
\item or $\mathrm{mult}_P(A)=3$ and $A$ is a~cone over a~plane smooth cubic curve with vertex $P$.
\end{itemize}

\begin{lemma}
\label{lemma:2-4-cone}
One has $\mathrm{mult}_P(A)\ne 3$.
\end{lemma}

\begin{proof}
Let $\sigma\colon\widehat{X}\to X$ be a~blow up of the~point $P$, and let $G$ be the~$\sigma$-exceptional surface.
Denote by $\widehat{A}$ and $\widehat{E}$ the~proper transforms on $\widehat{X}$ of the~surfaces $A$ and $E$, respectively.
Suppose that $\mathrm{mult}_P(A)=3$.  Take $u\in\mathbb{R}_{\geqslant 0}$. Then
$$
\sigma^*(-K_X)-vG\sim_{\mathbb{R}}\frac{1}{3}\widehat{E}+\frac{4}{3}\widehat{A}+(4-u)G.
$$
Thus, the~divisor $\sigma^*(-K_X)-vG$ is pseudo-effective $\iff$ $u\in[0,4]$.
Moreover, if $u\in[0,4]$, then the~Zariski decomposition of the~divisor $\sigma^*(-K_X)-vG$ can be described as follows:
$$
P(u)=\left\{\aligned
&\frac{1}{3}\widehat{E}+\frac{4}{3}\widehat{A}+(4-u)G\ \text{if $0\leqslant u\leqslant 1$}, \\
&\frac{1}{3}\widehat{E}+\frac{5-u}{3}\widehat{A}+(4-u)G\ \text{if $1\leqslant u\leqslant 4$}, \\
\endaligned
\right.
$$
and
$$
N(u)=\left\{\aligned
&0\ \text{if $0\leqslant u\leqslant 1$}, \\
&\frac{u-1}{3}\widehat{A}\ \text{if $1\leqslant u\leqslant 4$}, \\
\endaligned
\right.
$$
where $P(u)$ and $N(u)$ are the~positive and the~negative parts of the~Zariski decomposition, respectively.
Using this, we compute
$$
S_{X}(G)=\frac{3}{10}\int_{0}^{1}u^{3},du+\frac{3}{10}\int_{1}^{4}udu=\frac{93}{40}<3=A_{X}(G).
$$

As in the~proof of Lemma~\ref{lemma:2-1-d-3},
let us use results from \cite[\S~4]{Fujita2021} to get a~contradiction.
Namely, applying \cite[Corollary~4.18]{Fujita2021}, we get
$$
1\geqslant \frac{A_X(\mathbf{F})}{S_X(\mathbf{F})}\geqslant\min\left\{\frac{A_{X}(G)}{S_{X}(G)},
\inf_{Q\in G}\delta_{Q}\big(G,V^G_{\bullet,\bullet}\big)\right\}=
\min\left\{\frac{40}{31},\inf_{Q\in G}\delta_{Q}\big(G,V^G_{\bullet,\bullet}\big)\right\},
$$
where $\delta_{Q}(G,V^G_{\bullet,\bullet})$ is defined in \cite[\S~4]{Fujita2021}.
So, there is $Q\in G$ such that $\delta_{Q}(G,V^G_{\bullet,\bullet})<\frac{40}{31}$.

Let $\ell$ is a~general line in $G\cong\mathbb{P}^2$ that contains~$Q$.
Then  \cite[Corollary~4.18]{Fujita2021} gives
$$
\delta_{Q}\big(G,V^G_{\bullet,\bullet}\big)
\geqslant\min\left\{\frac{1}{S\big(V^F_{\bullet,\bullet};\ell\big)},\frac{1}{S\big(W^{G,\ell}_{\bullet,\bullet,\bullet};Q\big)}\right\}.
$$
Let us compute $S(V^G_{\bullet,\bullet};\ell)$ and $S(W^{G,\ell}_{\bullet,\bullet,\bullet};Q)$.
Take $v\in\mathbb{R}_{\geqslant 0}$. Then
$$
P(u)\big\vert_{G}-v\ell\sim_{\mathbb{R}}\left\{\aligned
&(u-v)\ell\ \text{if $0\leqslant u\leqslant 1$}, \\
&(1-v)\ell\ \text{if $1\leqslant u\leqslant 4$}.
\endaligned
\right.
$$
Let $\widehat{\mathscr{C}}=\widehat{A}\vert_{G}$. Then  $\widehat{\mathscr{C}}$
is a~smooth cubic curve in $G\cong\mathbb{P}^2$. Let
$$
N^\prime(u)=N(u)\big\vert_{G}=\left\{\aligned
&0\ \text{if $0\leqslant u\leqslant 1$}, \\
&\frac{u-1}{3}\widehat{\mathscr{C}}\ \text{if $1\leqslant u\leqslant 4$}.
\endaligned
\right.
$$
Now, using \cite[Theorem~4.8]{Fujita2021}, we get
\begin{multline*}
S\big(W^{G}_{\bullet,\bullet};\ell\big)=\frac{3}{10}\int_0^4\int_0^\infty \mathrm{vol}\big(P(u)\big\vert_{G}-v\ell\big)dvdu=\\
=\frac{3}{10}\int_{0}^{1}\int_{0}^{u}(u-v)^{2}dvdu+\frac{3}{10}\int_{1}^{4}\int_{0}^{1}(1-v)^{2}dvdu=\frac{13}{40}.
\end{multline*}
Similarly, it follows from \cite[Theorem~4.17]{Fujita2021} that $S(W_{\bullet,\bullet,\bullet}^{G,\ell};Q)$ can be computes as follows:
\begin{multline*}
\frac{3}{10}\int_{0}^{1}\int_{0}^{u}\Big(\big(P(u)\big\vert_{G}-v\ell\big)\cdot\ell\Big)^2dvdu+\frac{3}{10}\int_{1}^{4}\int_{0}^{1}\Big(\big(P(u)\big\vert_{G}-v\ell\big)\cdot\ell\Big)^2dvdu+F_Q=\\
=\frac{3}{12}\int_{0}^{1}\int_{0}^{u}(u-v)^{2}dvdu+\frac{3}{10}\int_{1}^{4}\int_{0}^{1}(1-v)^{2}dvdu+F_Q=\frac{13}{40}+F_Q,
\end{multline*}
where $F_Q=0$ if $Q\not\in\widehat{A}\vert_{G}$, and
$$
F_Q=\frac{6}{10} \int_{1}^{4} \int_{0}^{1}\frac{(1-v)(u-1)}{3}dvdu=\frac{9}{20}
$$
otherwise. This gives $S(W_{\bullet,\bullet,\bullet}^{G,\ell};Q)\leqslant\frac{31}{40}$.
Combining the~estimates, we get \mbox{$\delta_{Q}(G,V^G_{\bullet,\bullet})\geqslant\frac{40}{31}$}, which is a contradiction.
This completes the~proof of the~lemma.
\end{proof}

Hence, we see that the~surface $A$ has Du Val singularities.
Let $S$ be a~general surface in the~linear system $|H|$ that contains $P$. Then  $S$ is smooth, and $-K_X-uS\sim_{\mathbb{R}} (4-u)H-E$.
Hence, the divisor $-K_X-uS$ is pseudoeffective $\iff$ it is nef $\iff$ $u\leqslant 1$. Then
$$
S_X(S)=\frac{3}{10}\int_{0}^{1}(-K_X-uS)^3du=\frac{3}{10}\int_{0}^{1}u(1-u)(7-u)du=\frac{13}{40}<1.
$$

Let $C=A\vert_{S}$.
Then  $C$ is a~reduced curve in $|-K_S|$ that is singular at $P$, and $C\cong\pi(C)$.
Moreover, the~curve $\pi(C)$ is a~general hyperplane section of the~cubic surface $\pi(A)\subset\mathbb{P}^3$ that passes through the~point~$\pi(P)$.
Therefore, since $\pi(A)$ is not a cone by Lemma~\ref{lemma:2-4-cone}, we conclude that the~curve $C$ is irreducible.
Hence, one of the~following two cases holds:
\begin{enumerate}
\item the~curve $C$ has an ordinary node at $P$,
\item the~curve $C$ has an~ordinary cusp at $P$.
\end{enumerate}

Let $\Pi=\pi(S)$.
Then  $\Pi$ is a~plane in $\mathbb{P}^3$ such that $\pi(P)\in\Pi$ and $\Pi\cap\pi(A)=\pi(C)$,
and the~morphism $\pi\vert_{S}\colon S\to\Pi$ is a~composition of blow ups of $9$ intersection points $\Pi\cap\mathscr{C}$,
which we denote by $O_1,\ldots,O_9$.
Note that $\pi(C)$ is a~reduced plane cubic curve that passes through these nine points,
and $\pi(C)$ is smooth away from $\pi(P)$.

\begin{lemma}
\label{lemma:2-4-C-nodal}
The curve $C$ cannot have an ordinary double point at the~point $P$.
\end{lemma}

\begin{proof}
For~each $i\in\{1,\ldots,9\}$, let $L_i$ be the~proper transform on $S$ of the~line in $\Pi$ that passes through  $P$ and $O_i$.
Then  $L_i\ne L_j$ for $i\ne j$, since $\pi(C)$ is irreducible. We have
$$
(-K_X-uS)\big\vert_{S}\sim_{\mathbb{R}}\frac{1-u}{6}\sum_{i=1}^9L_i
$$

Let $\sigma\colon\widehat{S}\to S$ be the~blow up of $S$ at the~point $P$,
let $\mathbf{f}$ be the~$\sigma$-exceptional curve,
let~$\widehat{C},\widehat{L}_1,\ldots,\widehat{L}_{9}$ be  the~proper transforms on $\widehat{S}$ of the~curves $C,L_1,\ldots,L_9$, respectively.
Then  $\widehat{L}_1,\ldots,\widehat{L}_{9}$ are disjoint.
On the surface $\widehat{S}$, we have
$$
\widehat{C}^2=-4, \mathbf{f}^2=\widehat{L}_1^2=\cdots=\widehat{L}_9^2=-1, \widehat{L}_1\cdot\mathbf{f}=\cdots=\widehat{L}_9\cdot\mathbf{f}=1, \widehat{C}\cdot\mathbf{f}=2,
$$
and $\widehat{C}\cdot \widehat{L}_1=\cdots=\widehat{C}\cdot \widehat{L}_9=0$.

Fix $u\in[0,1]$. Let $v$ be a~non-negative real number. Then
\begin{equation}
\label{equation:2-4-blow-up-nodal}
\sigma^*\big((-K_X-uS)\big\vert_S\big)-v\mathbf{f}\sim_{\mathbb{R}}\frac{5+u}{6}\widehat{C}+\frac{1-u}{6}\sum_i^9\widehat{L}_i+\frac{19-7u-6v}{6}\mathbf{f}.
\end{equation}
So, the divisor $\sigma^*((-K_X-uS)|_S)-v\mathbf{f}$ is pseudoeffective $\iff$ it is nef $\iff$ $v\leqslant\frac{19-7u}{6}$.
Let $P(u,v)$ and $N(u,v)$ be the~positive and the~negative parts of its Zariski decomposition.
Then, using \eqref{equation:2-4-blow-up-nodal}, we compute
$$
P(u,v)=\left\{\aligned
&\frac{5+u}{6}\widehat{C}+\frac{1-u}{6}\sum_i^9\widehat{L}_i+\frac{19-7u-6v}{6}\mathbf{f}\ \text{if $0\leqslant v\leqslant \frac{3-3u}{2}$}, \\
&\frac{19-7u-6v}{12}\widehat{C}+\frac{1-u}{6}\sum_i^9\widehat{L}_i+\frac{19-7u-6v}{6}\mathbf{f}\ \text{if $\frac{3-3u}{2}\leqslant v\leqslant 3-u$}, \\
&\frac{19-7u-6v}{12}\big(\widehat{C}+2\sum_i^9\widehat{L}_i+2\mathbf{f}\big)\ \text{if $3-u\leqslant v\leqslant \frac{19-7u}{6}$},
\endaligned
\right.
$$
$$
N(u,v)=\left\{\aligned
&0\ \text{if $0\leqslant v\leqslant \frac{3-3u}{2}$}, \\
&\frac{-3+3u+2v}{4}\widehat{C}\ \text{if $\frac{3-3u}{2}\leqslant v\leqslant 3-u$}, \\
&\frac{2v+3u-3}{4}\widehat{C}+(v+u-3)\sum_i^9\widehat{L}_i\ \text{if $3-u\leqslant v\leqslant \frac{19-7u}{6}$},
\endaligned
\right.
$$
$$
\mathrm{vol}\Big(\sigma^*\big((-K_X-uS)\big|_S\big)-v\mathbf{f}\Big)=\left\{\aligned
&(1-u)(7-u)-v^2\ \text{if $0\leqslant v\leqslant \frac{3-3u}{2}$}, \\
&\frac{(1-u)(37-13u-12v)}{4}\ \text{if $\frac{3-3u}{2}\leqslant v\leqslant 3-u$}, \\
&\frac{(19-7u-6v)^2}{4}\ \text{if $3-u\leqslant v\leqslant \frac{19-7u}{6}$},
\endaligned
\right.
$$
$$
P(u,v)\cdot\mathbf{f}=\left\{\aligned
&v \ \text{if $0\leqslant v\leqslant \frac{3-3u}{2}$}, \\
&\frac{3(1-u)}{2} \ \text{if $\frac{3-3u}{2}\leqslant v\leqslant 3-u$}, \\
&\frac{3(19-7u-6v)}{2}\ \text{if $3-u\leqslant v\leqslant \frac{19-7u}{6}$}.
\endaligned
\right.
$$
Indeed, the divisor $P(u,v)$ above is effective, and it has a non-negative intersection with every irreducible component of $\operatorname{Supp}\left(P(u,v)\right)$,
which implies that it is nef in every case.
Now, using \cite[Corollary~1.109]{Book}, we get
\begin{multline*}
S\big(W_{\bullet,\bullet}^{\widehat{S}};\mathbf{f}\big)=
\frac{3}{10}\int_0^1\int_{0}^{\frac{19-7u}{6}}\mathrm{vol}\Big(\sigma^*\big((-K_X-uS)\big|_S\big)-v\mathbf{f}\Big)dudv=\\
=\frac{3}{10}\int_0^1\int_0^{\frac{3-3u}{2}}\big((1-u)(7-u)-v^2\big)dv+\frac{3}{10}\int_0^1\int_{\frac{3-3u}{2}}^{3-u}\frac{(1-u)(37-13u-12v)}{4}dv+\\
+\frac{3}{10}\int_0^1\int_{3-u}^{\frac{19-7u}{6}}\frac{(19-7u-6v)^2}{4}dvdu=\frac{767}{480}<2=A_S(\mathbf{f}).
\end{multline*}
Moreover, if $Q$ is a~point in $\mathbf{f}$, then \cite[Remark~1.113]{Book} gives
\begin{multline*}
S\big(W_{\bullet,\bullet,\bullet}^{\widehat{S},\mathbf{f}};Q\big)=
F_{Q}+\frac{3}{10}\int_{0}^{1}\int_{0}^{\frac{19-7u}{6}}\big(P(u,v)\cdot\mathbf{f}\big)^2dvdu=F_Q+\frac{3}{10}\int_0^1\int_0^{\frac{3-3u}{2}}v^2dv+\\
+\frac{3}{10}\int_0^1\int_{\frac{3-3u}{2}}^{3-u}\Bigg(\frac{3(1-u)}{2}\Bigg)^2dv+\frac{3}{10}\int_0^1\int_{3-u}^{\frac{19-7u}{6}}\Bigg(\frac{3(19-7u-6v)}{2}\Bigg)^2dvdu=F_Q+\frac{147}{320},
\end{multline*}
where
$$
F_Q=\frac{6}{10}\int_{0}^{1}\int_{0}^{\frac{19-7u}{6}}\big(P(u,v)\cdot\mathbf{f}\big)\mathrm{ord}_Q\Big(N(u,v)\big\vert_{\mathbf{f}}\Big)dvdu,
$$
which implies the~following assertions:
\begin{itemize}
\item if $Q\not\in\widehat{C}\cup\widehat{L}_1\cup\cdots\cup\widehat{L}_{9}$, then $F_Q=0$;
\item if $Q\in\widehat{L}_1\cup\cdots\cup\widehat{L}_{9}$, then
$$
F_Q=\frac{6}{10}\int_0^1\int_{3-u}^{\frac{19-7u}{6}}\frac{3(19-7u-6v)(v+u-3)}{2}dvdu=\frac{1}{960};
$$
\item if $Q\in\widehat{C}$ and $\widehat{C}$ intersects $\mathbf{f}$ transversally at $P$, then
\begin{multline*}
\quad \quad \quad \quad F_Q=\frac{6}{10}\int_0^1\int_{\frac{3-3u}{2}}^{3-u}\frac{3(1-u)(2v+3u-3)}{8}dvdu+\\
\frac{6}{10}\int_0^1\int_{3-u}^{\frac{19-7u}{6}}\frac{3(19-7u-6v)(2v+3u-3)}{8}dvdu=\frac{643}{1920};\quad \quad \quad \quad \quad
\end{multline*}
\item if $Q\in\widehat{C}$ and $\widehat{C}$ is tangent to $\mathbf{f}$ at the point $P$, then $F_Q=\frac{643}{960}$.
\end{itemize}
Thus, if $C$ has a node at $P$, then $S(W_{\bullet,\bullet,\bullet}^{\widehat{S},\mathbf{f}};Q)\leqslant\frac{305}{384}$,
so \cite[Remark~1.113]{Book} gives
$$
1\geqslant \frac{A_X(\mathbf{F})}{S_X(\mathbf{F})}\geqslant\min\left\{\inf_{Q\in\mathbf{f}}\frac{1}{S(W_{\bullet,\bullet,\bullet}^{\widehat{S},\mathbf{f}};Q)},
\frac{2}{S(V_{\bullet,\bullet}^{S};\mathbf{f})},\frac{1}{S_X(S)}\right\}
\geqslant\min\left\{\frac{384}{305},\frac{960}{767},\frac{40}{13}\right\}=\frac{960}{767}>1,
$$
which is a~contradiction.
\end{proof}

\begin{lemma}
\label{lemma:2-4-C-cuspidal}
The curve $C$ cannot have an ordinary cusp at the~point $P$.
\end{lemma}

\begin{proof}
Suppose $C$ has a~cusp. Let $L$ be an irreducible curve in $S$ such that $\pi(L)$ is a~line and $\pi(L)\cap\pi(C)=\pi(P)$.
Then  $(-K_X-uS)\vert_{S}\sim_{\mathbb{R}}(1-u)L+C$.

Now, we consider the~following commutative diagram:
$$
\xymatrix{
S_1\ar@{->}[d]_{\sigma_1}&&S_2\ar@{->}[ll]_{\sigma_2}&&S_3\ar@{->}[ll]_{\sigma_3}\ar@{->}[d]^{\upsilon}\\%
S &&&& \widehat{S}\ar@{->}[llll]_{\sigma}}
$$
where $\sigma_1$ is the~blow up of $P$,
$\sigma_2$ is the~blow up of the~point in the~$\sigma_1$-exceptional curve contained in the~proper transform of $C$,
$\sigma_3$ is the~blow up of the~point in the~$\sigma_2$-exceptional curve contained in the~proper transform of $C$,
$\upsilon$ is the~birational contraction of the~proper transforms of~$\sigma_1\circ\sigma_2$-exceptional curves,
and $\sigma$ is the~birational contraction of the~proper transform of the~$\sigma_3$-exceptional curve.
Then  $\widehat{S}$ has two singular points:
\begin{enumerate}
\item a~cyclic quotient singularity of type $\frac{1}{2}(1,1)$, which we denote by $Q_2$;
\item a~cyclic quotient singularity of~type~$\frac{1}{3}(1,1)$, which we denote by $Q_3$.
\end{enumerate}

Let $\mathbf{f}$ be the~$\sigma$-exceptional curve,
let $\widehat{C}$ be the~proper transform on $\widehat{S}$ of the~curve $C$,
and let $\widehat{L}$ be the~proper transform of the~curve $L$.
Then  the~curves $\mathbf{f}$, $\widehat{C}$, $\widehat{L}$ are smooth.
Moreover, it is not very difficult to check that $Q_2\in\mathbf{f}\ni Q_3$, $Q_2\not\in\widehat{C}\not\ni Q_3$, $Q_2\in\widehat{L}\not\ni Q_3$.
Further, we have $A_S(\mathbf{f})=5$, $\sigma^*(C)\sim\widehat{C}+6\mathbf{f}$, $\sigma^*(L)\sim\widehat{L}+3\mathbf{f}$.
On the surface $\widehat{S}$, we have
$$
\widehat{L}^2=-\frac{1}{2}, \widehat{L}\cdot\widehat{C}=0, \widehat{C}\cdot \mathbf{f}=\frac{1}{2}, \widehat{C}^2=-6, \widehat{C}\cdot \mathbf{f}=1, \mathbf{f}^2=-\frac{1}{6}.
$$
Note that  $Q_2=\mathbf{f}\cap\widehat{L}$, and $\widehat{C}$ intersects $\mathbf{f}$ transversally by one point.

Fix $u\in[0,1]$. Let $v$ be a~non-negative real number. Then
\begin{equation}
\label{equation:2-4-blow-up-cuspidal}
\sigma^*\big((-K_X-uS)\big\vert_{S}\big)-v\mathbf{f}\sim_{\mathbb{R}}(1-u)\widehat{L}+\widehat{C}+(9-3u-v)\mathbf{f}.
\end{equation}
Thus, the divisor $\sigma^*((-K_X-uS)|_S)-v\mathbf{f}$ is pseudoeffective $\iff$ it is nef $\iff$ $v\leqslant 9-3u$.
Let $P(u,v)$ be the~positive part of the Zariski decomposition of $\sigma^*((-K_X-uS)|_S)-v\mathbf{f}$, and let  $N(u,v)$ be the negative part.
Then, using \eqref{equation:2-4-blow-up-cuspidal}, we compute
$$
P(u,v)=\left\{\aligned
&(1-u)\widehat{L}+\widehat{C}+(9-3u-v)\mathbf{f}\ \text{if $0\leqslant v\leqslant 3-3u$}, \\
&(1-u)\widehat{L}+\frac{9-3u-v}{6}\widehat{C}+(9-3u-v)\mathbf{f}\ \text{if $3-3u\leqslant v\leqslant 8-2u$}, \\
&\frac{9-3u-v}{6}\big(6\widehat{L}+\widehat{C}+6\mathbf{f}\big)\ \text{if $8-2u\leqslant v\leqslant 9-3u$},
\endaligned
\right.
$$
and
$$
N(u,v)=\left\{\aligned
&0\ \text{if $0\leqslant v\leqslant 3-3u$}, \\
&\frac{v+3u-3}{6}\widehat{C}\ \text{if $3-3u\leqslant v\leqslant 8-2u$}, \\
&\frac{v+3u-3}{6}\widehat{C}+(v+2u-8)\widehat{L}\ \text{if $8-2u\leqslant v\leqslant 9-3u$}.
\endaligned
\right.
$$
This gives
$$
\mathrm{vol}\Big(\sigma^*\big((-K_X-uS)\big|_S\big)-v\mathbf{f}\Big)=\left\{\aligned
&(1-u)(7-u)-\frac{v^2}{6}\ \text{if $0\leqslant v\leqslant 3-3u$}, \\
&\frac{(1-u)(17-5u-2v)}{2}\ \text{if $3-3u\leqslant v\leqslant 8-2u$}, \\
&\frac{(9-3u-v)^2}{2}\ \text{if $8-2u\leqslant v\leqslant 9-3u$},
\endaligned
\right.
$$
and
$$
P(u,v)\cdot\mathbf{f}=\left\{\aligned
&\frac{v}{6}\ \text{if $0\leqslant v\leqslant 3-3u$}, \\
&\frac{1-u}{2}\ \text{if $3-3u\leqslant v\leqslant 8-2u$}, \\
&\frac{9-3u-v}{2}\ \text{if $8-2u\leqslant v\leqslant 9-3u$}.
\endaligned
\right.
$$
Now, we use \cite[Corollary~1.109]{Book} to get
\begin{multline*}
S\big(W_{\bullet,\bullet}^{\widehat{S}};\mathbf{f}\big)=\frac{3}{10}\int_0^1\int_0^{3-3u}\Big((1-u)(7-u)-\frac{v^2}{6}\Big)dv+\\
+\frac{3}{10}\int_0^1\int_{3-3u}^{8-2u}\frac{(1-u)(17-5u-2v)}{2}dv+\frac{3}{10}\int_0^1\int_{8-2u}^{9-3u}\frac{(9-3u-v)^2}{2}dvdu=\frac{173}{40}.
\end{multline*}
Similarly, if $Q$ is a~point in $\mathbf{f}$, then \cite[Remark~1.113]{Book} gives
\begin{multline*}
S\big(W_{\bullet,\bullet,\bullet}^{\widehat{S},\mathbf{f}};Q\big)=F_{Q}+\frac{3}{10}\int_{0}^{1}\int_{0}^{9-3u}\big(P(u,v)\cdot\mathbf{f}\big)^2dvdu=\\
=F_Q+\frac{3}{10}\int_0^1\int_0^{3-3u}\Big(\frac{v}{6}\Big)^2dvdu+\frac{3}{10}\int_0^1\int_{3-3u}^{8-2u}\Big(\frac{1-u}{2}\Big)^2dvdu+\\
+\frac{3}{10}\int_0^1\int_{8-2u}^{9-3u}\Big(\frac{9-3u-v}{2}\Big)^2dvdu=F_Q+\frac{5}{32},
\end{multline*}
where $F_Q$ can be computes as follows:
\begin{itemize}
\item if $Q\ne\widehat{C}\cap\mathbf{f}$ and $Q\ne\widehat{L}\cap\mathbf{f}$, then $F_Q=0$;
\item if $Q=\widehat{L}\cap\mathbf{f}$, then
$$
F_Q=\frac{6}{10}\int_0^1\int_{8-2u}^{9-3u}\frac{(9-3u-v)(v+2u-8)}{2}dvdu=\frac{1}{80};
$$
\item if $Q=\widehat{C}\cap\mathbf{f}$ , then
\begin{multline*}
\quad \quad \quad \quad
F_Q=\frac{6}{10}\int_0^1\int_{3-3u}^{8-2u}\frac{(1-u)(v+3u-3)}{12}dv+\\
+\frac{6}{10}\int_0^1\int_{8-2u}^{9-3u}\frac{(9-3u-v)(v+3u-3)}{12}dudv=\frac{193}{480}.\quad \quad \quad \quad\quad
\end{multline*}
\end{itemize}
Therefore, we conclude that
$$
S\big(W_{\bullet,\bullet,\bullet}^{\widehat{S},\mathbf{f}};Q\big)=
\left\{\aligned
&\frac{5}{32}\ \text{if $Q\not\in\widehat{C}\cup\widehat{L}$}, \\
&\frac{27}{80} \ \text{if $Q=\widehat{L}\cap\mathbf{f}$}, \\
&\frac{67}{120}\ \text{if $Q=\widehat{C}\cap\mathbf{f}$}.
\endaligned
\right.
$$

Let $\Delta_\mathbf{f}=\frac{1}{2}Q_2+\frac{2}{3}Q_3$. Then, using \cite[Remark~1.113]{Book} and our computations, we get
$$
1\geqslant \frac{A_X(\mathbf{F})}{S_X(\mathbf{F})}\geqslant\min\left\{\inf_{Q\in\mathbf{f}}\frac{1-\mathrm{ord}_Q\big(\Delta_\mathbf{f}\big)}{S(W_{\bullet,\bullet,\bullet}^{\widehat{S},\mathbf{f}};Q)},
\frac{5}{S(V_{\bullet,\bullet}^{S};\mathbf{f})},\frac{1}{S_X(S)}\right\}
\geqslant \min \left\{\frac{40}{13}, \frac{200}{173}, \frac{120}{67}\right\}=\frac{200}{173}>1,
$$
which is a~contradiction.
\end{proof}

Combining Lemmas~\ref{lemma:2-4-C-nodal} and \ref{lemma:2-4-C-cuspidal}, we obtain a contradiction.

\begin{corollary}
\label{corollary:2-4-final}
All smooth Fano threefolds in the~family \textnumero 2.4 are K-stable.
\end{corollary}

\section{Family \textnumero 2.6 (Verra threefolds)}
\label{section:2-6}

Smooth Fano threefolds in the~family \textnumero 2.6 can be described as follows:
\begin{itemize}
\item[(a)] smooth divisors of degree $(2,2)$ in $\mathbb{P}^2\times \mathbb{P}^2$, which are known as Verra threefolds,

\item[(b)] double covers of the~(unique) smooth divisor in $\mathbb{P}^2\times \mathbb{P}^2$ of degree $(1,1)$ branched over smooth anticanonical K3 surfaces.
\end{itemize}
Note that every double cover of a~smooth divisor in $\mathbb{P}^2\times \mathbb{P}^2$ of degree $(1,1)$ branched over a~smooth anticanonical surface is K-stable \cite[Example~4.4]{Dervan2}.
In fact, this also implies that general Verra threefold is K-stable \cite[Example~3.14]{Book}.

The goal of this section is to prove that all smooth Verra threefolds are K-stable.

Let $X$ be a~smooth divisor of degree $(2,2)$ in $\mathbb{P}^2\times \mathbb{P}^2$,
let $\pi_1\colon X\to \mathbb{P}^2$ be the~projection to the~first factor of $\mathbb{P}^2\times \mathbb{P}^2$,
and let $\pi_2\colon X\to \mathbb{P}^2$ be the~projection to the~second factor.
Then  $\pi_1$ and $\pi_2$ are conic bundles \cite{Prokhorov}.
Set $H_1=\pi_1^*(\mathcal{O}_{\mathbb{P}^2}(1))$ and $H_2=\pi_2^*(\mathcal{O}_{\mathbb{P}^2}(1))$.
Then
$$
-K_X\sim H_1+H_2,
$$
and the~group $\mathrm{Pic}(X)$ is generated by $H_1$ and $H_2$.
Note that $\mathrm{Aut}(X)$ is finite~\cite{CheltsovShramovPrzyjalkowski}.
Thus, the~threefold $X$ is K-stable $\iff$ it is K-polystable. See \cite{Xu} for details.

\begin{lemma}
\label{lemma:2-6-two-fibers}
Fix a~point $P\in X$. Let $C_1$ be the~fiber of the~conic bundle $\pi_1$ that contains~$P$,
and let $C_2$ be the~fiber of $\pi_2$ that contains $P$.
Then  $C_1$ or $C_2$ is smooth at~$P$.
\end{lemma}

\begin{proof}
Local computations.
\end{proof}

Let $\Delta_1$ and $\Delta_2$ be the~discriminant curves of the~conic bundles $\pi_1$ and~$\pi_2$, respectively.
Then  $\Delta_1$ and $\Delta_2$ are reduced curves of degree $6$
that have at most ordinary double points as singularities.
For basic properties of the~discriminant curves  $\Delta_1$ and $\Delta_2$, see~\mbox{\cite[\S~3.8]{Prokhorov}}.
In particular, we know that no line in $\mathbb{P}^2$ can be an irreducible component of these curves.

\begin{lemma}
\label{lemma:2-6-del-Pezzo-degree-two}
Fix a~point $P\in X$. Let $C_2$ be the~fiber of the~conic bundle $\pi_2$ that contains~$P$,
let $S$ be a~general surface in $|H_2|$ that contains $C_2$.
Then  one of the~following cases holds:
\begin{enumerate}
\item $C_2$ is smooth, $S$ is smooth, the~divisor $-K_S$ is ample;

\item $C_2$ is smooth, $S$ is smooth, the~divisor $-K_S$ is nef, the~surface $S$~has exactly four irreducible curves that
have trivial intersection with the~divisor $-K_S$, these curves are disjoint and none of them passes through $P$,
and $C_2\subset\mathrm{Sing}(\pi_1^{-1}(\Delta_1))$;

\item $C_2$ is singular and reduced, $S$ is smooth, the~divisor $-K_S$ is ample;

\item $C_2$ is not reduced, $P\not\in\mathrm{Sing}(S)\subset\mathrm{Supp}(C_2)$,
and $\mathrm{Sing}(S)$ consists of two~points, which are ordinary double points of the~surface $S$,
and $-K_S$ is ample.
\end{enumerate}
\end{lemma}

\begin{proof}
The assertion about the~singularities of the~surface $S$ are local and well-known.

We have $-K_S\sim H_1\vert_S$ and $K_S^2=2$, so  $S$ is a~weak del Pezzo surface of degree $2$,
and the~restriction $\pi_1\vert_{S}\colon S\to\mathbb{P}^2$ is the~anticanonical morphism.
Let $\ell$ be an irreducible curve in the~surface $S$ such that
$$
-K_S\cdot \ell=0.
$$
Then  $\ell$ is an irreducible component of the~fiber $\pi_1^{-1}(\pi_1(\ell))$.
But $\pi_2(\ell)$ is the~line $\pi_2(S)$, which implies that the~scheme fiber $\pi_1^{-1}(\pi_1(\ell))$ is singular,
which implies that $\pi_1(\ell)\in\Delta_1$.
Since $\ell\cap C_2\ne\varnothing$ and $\pi_2(S)$ is a~general line in $\mathbb{P}^2$ that passes through the~point $\pi_2(P)$,
we see that $\pi_1(C_2)\subset\Delta_1$. This implies that $C_2$ is irreducible and reduced.

Let $R=\pi_1^{-1}(\pi_1(C_2))$. Then  the~surface $R$ is singular along a~curve that is isomorphic to the~conic $\pi_1(C_2)\cong C_2$.
Let $f\colon\widetilde{R}\to R$ be the~blow up of this curve. Then  $\widetilde{R}$ is smooth,
and the~composition morphism $\pi_1\circ f$ induces a~$\mathbb{P}^1$-bundle $\widetilde{R}\to\widetilde{C}_2$,
where $\widetilde{C}_2$ is double cover of the~conic $\pi_1(C_2)$ that is branched over the~eight points $\pi_1(C_2)\cap(\Delta_1-\pi_1(C_2))$.
In particular, we see that $\widetilde{C}_2$ is an irrational curve, which implies that
$C_2=\mathrm{Sing}(R)$.

Vice versa, if the~fiber $C_2$ is a~smooth conic, and the~conic $\pi_1(C_2)$ is an irreducible component of the~discriminant curve $\Delta_1$,
then it follows from the~Bertini theorem that
$$
S\cdot\pi_1^{-1}(\pi_1(C_2))=2C_2+\ell_1+\ell_2+\cdots+\ell_k
$$
where $\ell_1,\ell_2,\ldots,\ell_k$ are $k$ distinct irreducible reduced curves in $X$ that are irreducible components of the~fibers
of the~natural projection  $\pi_1^{-1}(\pi_1(C_2))\to \pi_1(C_2)$.
Since
$$
4=2H_2^2\cdot H_1=H_2\cdot S\cdot\pi_1^{-1}(\pi_1(C_2))=H_2\cdot\Big(2C_2+\sum_{i=1}^k\ell_i\Big)=k,
$$
we see that $S$ contains exactly $4$ irreducible curves that intersects trivially with $-K_S$.
Now, the~generality in the~choice of $S$ implies that none of these curves contains $P$.
\end{proof}

\begin{example}
\label{example:2-6}
Actually, the case (2) in Lemma~\ref{lemma:2-6-del-Pezzo-degree-two} can happen.
Indeed, in the~assumption and notations of Lemma~\ref{lemma:2-6-del-Pezzo-degree-two}, let $P=([0:0:1],[0:0:1])$,
and suppose that $X$ is given~by
$$
(u^2+2uw+v^2+2w^2)x^2+(uv-w^2)xy+(uw-2uv+3v^2+w^2)y^2+(uw+v^2)z^2=0,
$$
where $([u:v:w],[x:y:z])$ are coordinates on $\mathbb{P}^2\times\mathbb{P}^2$.
Then  the~threefold $X$ is smooth.
For instance, this can be checked using the~following Magma script:
\begin{verbatim}
Q:=RationalField();
PxP<x,y,z,u,v,w>:=ProductProjectiveSpace(Q,[2,2]);
X:=Scheme(PxP,[(u^2+2*u*w+v^2+2*w^2)*x^2+(u*v-w^2)*x*y+
               ((-2*v+w)*u+3*v^2+w^2)*y^2+(u*w+v^2)*z^2]);
IsNonsingular(X);
\end{verbatim}
Observe that the~fiber $C_1$ is a~singular reduced curve given by $u=v=2x^2-xy+y^2=0$,
the fiber $C_2$ is a~smooth curve that is given by $x=y=uw+v^2=0$,
and the~discriminant curve $\Delta_1$ is a union of the~conic $\pi_1(C_2)$ and the~irreducible quartic plane curve given by
\begin{multline*}
8u^3v-4u^3w-11u^2v^2+16u^2vw-12u^2w^2+8uv^3-\\
-28uv^2w+14uvw^2-16uw^3-12v^4-28v^2w^2-7w^4=0.
\end{multline*}
As in case (2) in Lemma~\ref{lemma:2-6-del-Pezzo-degree-two}, we have $C_2=\mathrm{Sing}(\pi_1^{-1}(\pi_1(C_2)))$.
\end{example}

Let us prove that $X$ is K-stable. Suppose it is not.
Using the~valuative criterion \cite{Fujita2019Crelle,Li},
we see that there exists a~prime divisor $\mathbf{F}$ over $X$ such that
$$
\beta(\mathbf{F})=A_X(\mathbf{F})-S_X(\mathbf{F})\leqslant 0.
$$
Let $Z$ be the~center of the~divisor $\mathbf{F}$ on $X$.
Then  $Z$ is not a~surface by \cite[Theorem~3.17]{Book}.

Let $P$ be any point in $Z$,
let $C_1$ be the~fiber of the~conic bundle $\pi_1$ that contains~$P$,
and let~$C_2$ be the~fiber of the~conic bundle $\pi_2$ that contains $P$.
By Lemma~\ref{lemma:2-6-two-fibers}, at least one curve among $C_1$ or $C_2$ is smooth at~$P$.
We may assume that $C_2$ is smooth at $P$.

Let $S$ be a~general surface in $|H_2|$ that contains $C_2$.
Then  $S$ is smooth by~Lemma~\ref{lemma:2-6-del-Pezzo-degree-two}.
Moreover, one of the~following three cases holds:
\begin{enumerate}
\item $C_2$ is smooth, the~divisor $-K_S$ is ample;

\item $C_2$ is smooth, $\pi_1(C_2)\subset\Delta_1$, the~divisor $-K_S$ is nef, the~surface $S$~has exactly four irreducible curves that
have trivial intersection with the~divisor $-K_S$, these curves are disjoint and none of them passes through $P$;

\item $C_2$ is singular and reduced, the~divisor $-K_S$ is ample.
\end{enumerate}
Let $C$ be the~curve in $X$ that is defined as follows:
\begin{itemize}
\item if $C_2$ is smooth and irreducible, we let $C=C_2$.
\item if $C_2$ is reducible, we let $C$ be its irreducible component that contains $P$.
\end{itemize}
Note that $-K_S\sim H_1\vert_{S}$ and $K_S^2=2$.
Let $\eta\colon S\to\mathbb{P}^2$ be the~restriction morphism~$\pi_1\vert_{S}$.
Then  $\eta$ is an anticanonical morphism of the~surface $S$, which is generically two-to-one.
Hence, the morphism $\eta$ induces an involution $\tau\in\mathrm{Aut}(S)$. We let $C^\prime=\tau(C)$.

Now, let $u$ be a non-negative real number. Then  we have $-K_X-uS\sim_{\mathbb{R}} H_1+(1-u)H_2$,
so $-K_X-uS$ is nef $\iff$ $-K_X-uS$ is pseudoeffective $\iff$ $u\leqslant 1$. Then  $S_X(S)=\frac{5}{12}$.
Now, let us use notations introduced in \cite[\S~1.7]{Book}.
Using \cite[Theorem~1.112]{Book}, we get
$$
1\geqslant \frac{A_X(\mathbf{F})}{S_X(\mathbf{F})}\geqslant \min\left\{\frac{1}{S_X(S)},\frac{1}{S(W_{\bullet,\bullet}^S;C)},\frac{1}{S(W_{\bullet, \bullet,\bullet}^{S,C};P)}\right\},
$$
where $S(W_{\bullet,\bullet}^S;C)$ and $S(W_{\bullet, \bullet,\bullet}^{S,C};P)$ are defined in \cite[\S~1.7]{Book}.
Hence, since $S_X(S)<1$, we~get
\begin{equation}
\label{equation:2-6-main-inequality}
\max\left\{S(W_{\bullet,\bullet}^S;C),S(W_{\bullet, \bullet,\bullet}^{S,C};P)\right\}\geqslant 1.
\end{equation}
Moreover, if $Z=P$, then it also follows from \cite[Theorem~1.112]{Book}~that
\begin{equation}
\label{equation:2-6-main-inequality-refined}
\max\left\{S(W_{\bullet,\bullet}^S;C),S(W_{\bullet, \bullet,\bullet}^{S,C};P)\right\}>1.
\end{equation}
Let us estimate  $S(W_{\bullet,\bullet}^S;C)$ and $S(W_{\bullet, \bullet,\bullet}^{S,C};P)$ using results obtained in \cite[\S~1.7]{Book}.

Let $P(u)=-K_X-uS$. Then  \cite[Corollary~1.109]{Book} gives
$$
S\big(W^S_{\bullet,\bullet};C\big)=\frac{1}{4}\int_0^1\int_0^\infty \mathrm{vol}\big(P(u)\big\vert_{S}-vC\big)dvdu,
$$
Let $P(u,v)$ be the~positive part of the~Zariski decomposition of the~divisor $P(u)\vert_{S}-vC$,
and let $N(u,v)$ be its~negative part,  where $u\in[0,1]$ and $v\in\mathbb{R}_{\geqslant 0}$.
Then
$$
S\big(W_{\bullet,\bullet,\bullet}^{S,C};P\big)=F_{P}+\frac{1}{4}\int_{0}^{1}\int_{0}^{\infty}\big(P(u,v)\cdot C\big)^2dvdu
$$
by \cite[Theorem~1.112]{Book}, where
$$
F_P=\frac{1}{2}\int_{0}^{1}\int_{0}^{\infty}\big(P(u,v)\cdot C\big)\mathrm{ord}_P\Big(N(u,v)\big\vert_{C}\Big)dvdu.
$$

\begin{lemma}
\label{lemma:2-6-smooth-conic-del-Pezzo}
Suppose that $C_2$ is smooth, and $-K_S$ is ample. Then
\begin{align*}
S(W_{\bullet,\bullet}^S;C)&=\frac{13}{24},\\
S(W_{\bullet, \bullet,\bullet}^{S,C};P)&=1.
\end{align*}
\end{lemma}

\begin{proof}
We have $C\cdot C^\prime=4$ and $(C^\prime)^2=C^2=0$. Fix $u\in[0,1]$ and $v\in\mathbb{R}_{\geqslant 0}$. Then
$$
P(u)\vert_{S}-vC\sim_{\mathbb{Q}}\Big(\frac{3}{2}-u-v\Big)C+\frac{1}{2}C^\prime.
$$
which implies that $P(u)\vert_{S}-vC$ is pseudoeffective $\iff$ $P(u)\vert_{S}-vC$ is nef $\iff$ $v\leqslant \frac{3}{2}-u$.
If $0\leqslant u\leqslant 1$ and $0\leqslant v\leqslant \frac{3}{2}-u$, then $P(u,v)=(\frac{3}{2}-u-v)C+\frac{1}{2}C^\prime$ and $N(u,v)=0$, so
$$
S\big(W^S_{\bullet,\bullet};C\big)=\frac{1}{4}\int_0^1\int_0^{\frac{3}{2}-u}\Bigg(\Big(\frac{3}{2}-u-v)C+\frac{1}{2}C^\prime\Bigg)^2dvdu=
\frac{1}{4}\int_0^1\int_0^{\frac{3}{2}-u}6-4u-4v\,dvdu=\frac{13}{24}.
$$
Similarly, we see that $F_P=0$ and $P(u,v)\cdot C=2$, which gives $S(W_{\bullet,\bullet,\bullet}^{S,C};P)=1$.
\end{proof}

\begin{lemma}
\label{lemma:2-6-smooth-conic-weak-del-Pezzo}
Suppose that $C_2$ is smooth, $-K_S$ is not ample. Then
\begin{align*}
S(W_{\bullet,\bullet}^S;C)&=\frac{7}{12},\\
S(W_{\bullet, \bullet,\bullet}^{S,C};P)&=\frac{5}{6}.
\end{align*}
\end{lemma}

\begin{proof}
In this case, we have the~following commutative diagram:
$$
\xymatrix{
&S\ar@{->}[ld]_{\phi}\ar@{->}[rd]^{\eta}&\\%
\overline{S}\ar@{->}[rr]^{\pi}&&\mathbb{P}^2}
$$
where $\phi$ is a~birational map that contracts four disjoint $(-2)$-curves,
and $\overline{S}$ is a~del Pezzo surface of degree $2$ that has $4$ isolated ordinary double points,
and $\pi$ is a~double cover that is ramified in a~reducible quartic curve that is a~union of two irreducible conics
such that $\eta(C)$ is one of these two conics.
In particular, we have $C=\tau(C)$.

Let $E_1$, $E_2$, $E_3$, $E_4$ be the~$\phi$-exceptional curves.
Fix $u\in[0,1]$ and $v\in\mathbb{R}_{\geqslant 0}$. Then
$$
P(u)\vert_{S}-vC\sim_{\mathbb{Q}}(2-u-v)C+\frac{1}{2}(E_1+E_2+E_3+E_4)
$$
so the divisor $P(u)\vert_{S}-vC$ is pseudoeffective $\iff$ $v\leqslant 2-u$.
Moreover, we have
$$
P(u,v)=\left\{\aligned
&(2-u-v)C+\frac{1}{2}(E_1+E_2+E_3+E_4)\ \text{if $0\leqslant v\leqslant 1-u$}, \\
&(2-u-v)\Big(C+\frac{1}{2}(E_1+E_2+E_3+E_4)\Big)\ \text{if $1-u\leqslant v\leqslant 2-u$}, \\
\endaligned
\right.
$$
$$
N(u,v)=\left\{\aligned
&0\ \text{if $0\leqslant v\leqslant 1-u$}, \\
&\frac{u+v-1}{2}(E_1+E_2+E_3+E_4)\ \text{if $1-u\leqslant v\leqslant 2-u$}, \\
\endaligned
\right.
$$
$$
\mathrm{vol}\big(P(u)\vert_{S}-vC\big)=
\left\{\aligned
&(6-4u)-4\ \text{if $0\leqslant v\leqslant 1-u$}, \\
&2(2-u-v)^2\ \text{if $1-u\leqslant v\leqslant 2-u$}. \\
\endaligned
\right.
$$
Now, integrating $\mathrm{vol}(P(u)\vert_{S}-vC)$, we obtain $S(W_{\bullet,\bullet}^S;C)=\frac{7}{12}$.

To compute $S(W_{\bullet,\bullet,\bullet}^{S,C};P)$, we first observe that $F_{P}=0$,
because $P\not\in E_1\cup E_2\cup E_3\cup E_4$, since $S$ is a~general surface in $|H_2|$ that contains  $C_2$.
On the~other hand, we have
$$
P(u,v)\cdot C=\left\{\aligned
&2\ \text{if $0\leqslant v\leqslant 1-u$}, \\
&4-2u-2v\ \text{if $1-u\leqslant v\leqslant 2-u$}. \\
\endaligned
\right.
$$
Hence, integrating $(P(u,v)\cdot C)^2$, we get $S(W_{\bullet,\bullet,\bullet}^{S,C};P)=\frac{5}{6}$ as required.
\end{proof}

\begin{lemma}
\label{lemma:2-6-reducible-conic}
Suppose that $C_2$ is singular. Then
\begin{align*}
S(W_{\bullet,\bullet}^S;C)&=\frac{3}{4},\\
S(W_{\bullet, \bullet,\bullet}^{S,C};P)&\leqslant \frac{11}{12}.
\end{align*}
\end{lemma}

\begin{proof}
The curve $C_2$ consists of two irreducible components: the~curve $C$ and another curve, which we denote by $L$.
The curves $C$ and $L$ are smooth and intersects transversally at one point.
Note that $P\ne C\cap L$, since $C_2$ is smooth at the~point $P$ by assumption.

The intersections of the~curves $C$, $L$ and $C^\prime=\tau(C)$ on $S$ are given in the~table below.
\begin{center}
\renewcommand\arraystretch{1.4}
\begin{tabular}{|c||c|c|c|}
\hline
$\bullet$  & $C$ & $L$ & $C^\prime$ \\
\hline\hline
$C$        & $-1$ & $1$ & $2$ \\
\hline
$L$        & $1$ & $-1$ & $0$\\
\hline
$C^\prime$ & $2$ & $0$ & $-1$ \\
\hline
\end{tabular}
\end{center}

Fix $u\in[0,1]$ and $v\in\mathbb{R}_{\geqslant 0}$. Since $C+C^\prime\sim -K_S$, we have
$$
P(u)\vert_{S}-vC\sim_{\mathbb{Q}}(2-u-v)C+(1-u)L+C^\prime,
$$
so $P(u)\vert_{S}-vC$ is pseudoeffective $\iff$ $v\leqslant 2-u$.
Moreover, if $0\leqslant u\leqslant\frac{1}{2}$, then
$$
P(u,v)=\left\{\aligned
&(2-u-v)C+(1-u)L+C^\prime\ \text{if $0\leqslant v\leqslant 1$}, \\
&(2-u-v)(C+L)+C^\prime\ \text{if $1\leqslant v\leqslant \frac{3}{2}-u$}, \\
&(2-u-v)(C+L+C^\prime)\ \text{if $\frac{3}{2}-u\leqslant v\leqslant 2-u$},
\endaligned
\right.
$$
$$
N(u,v)=\left\{\aligned
&0\ \text{if $0\leqslant v\leqslant 1$}, \\
&(v-1)L\ \text{if $1\leqslant v\leqslant \frac{3}{2}-u$}, \\
&(v-1)L+(2v+2u-3)C^\prime\ \text{if $\frac{3}{2}-u\leqslant v\leqslant 2-u$},
\endaligned
\right.
$$
$$
\mathrm{vol}\big(P(u)\vert_{S}-vC\big)=
\left\{\aligned
&6-v^2-4u-2v\ \text{if $0\leqslant v\leqslant 1$}, \\
&7-4u-4v \ \text{if $1\leqslant v\leqslant \frac{3}{2}-u$}, \\
&4(u+v-2)^2\ \text{if $\frac{3}{2}-u\leqslant v\leqslant 2-u$}.
\endaligned
\right.
$$
Similarly, if $\frac{1}{2}\leqslant u\leqslant 1$, then
$$
P(u,v)=\left\{\aligned
&(2-u-v)C+(1-u)L+C^\prime\ \text{if $0\leqslant v\leqslant \frac{3}{2}-u$}, \\
&(2-u-v)(C+2C^\prime)+(1-u)L\ \text{if $\frac{3}{2}-u\leqslant v\leqslant 1$}, \\
&(2-u-v)(C+L+C^\prime)\ \text{if $1\leqslant v\leqslant 2-u$},
\endaligned
\right.
$$
$$
N(u,v)=\left\{\aligned
&0\ \text{if $0\leqslant v\leqslant \frac{3}{2}-u$}, \\
&(2v+2u-3)C^\prime\ \text{if $\frac{3}{2}-u\leqslant v\leqslant 1$}, \\
&(v-1)L+(2v+2u-3)C^\prime\ \text{if $1\leqslant v\leqslant 2-u$},
\endaligned
\right.
$$
$$
\mathrm{vol}\big(P(u)\vert_{S}-vC\big)=
\left\{\aligned
&6-v^2-4u-2v\ \text{if $0\leqslant v\leqslant \frac{3}{2}-u$}, \\
&(5-2u-3v)(3-2u-v)\ \text{if $\frac{3}{2}-u\leqslant v\leqslant 1$}, \\
&4(u+v-2)^2\ \text{if $1\leqslant v\leqslant 2-u$}.
\endaligned
\right.
$$
Hence, integrating $\mathrm{vol}(P(u)\vert_{S}-vC)$, we get $S(W_{\bullet,\bullet}^S;C)=\frac{3}{4}$.

Now, let us compute $S(W_{\bullet, \bullet,\bullet}^{S,C};P)$.
If $0\leqslant u\leqslant \frac{1}{2}$, then
$$
P(u,v)\cdot C=\left\{\aligned
&1+v\ \text{if $0\leqslant v\leqslant 1$}, \\
&2\ \text{if $1\leqslant v\leqslant \frac{3}{2}-u$}, \\
&8-4u-4v\ \text{if $\frac{3}{2}-u\leqslant v\leqslant 2-u$}.
\endaligned
\right.
$$
Similarly, if $\frac{1}{2}\leqslant u\leqslant 1$, then
$$
P(u,v)\cdot C=\left\{\aligned
&1+v\ \text{if $0\leqslant v\leqslant \frac{3}{2}-u$}, \\
&7-4u-3v\ \text{if $\frac{3}{2}-u\leqslant v\leqslant 1$}, \\
&8-4u-4v\ \text{if $1\leqslant v\leqslant 2-u$}.
\endaligned
\right.
$$
Then  integrating, we get $S(W_{\bullet,\bullet,\bullet}^{S,C};P)=\frac{145}{192}+F_{P}$.
To compute $F_{P}$, let us recall that~$P\not\in L$.
Hence, if $P\not\in C^\prime$, then $F_{P}=0$,
which implies that $S(W_{\bullet,\bullet,\bullet}^{S,C};P)=\frac{145}{192}<\frac{11}{12}$ as required.

We may assume that $P\in C\cap C^\prime$.
If $C^\prime$ intersects $C$ transversally at $P$, then
$$
\mathrm{ord}_P\Big(N(u,v)\big\vert_{C}\Big)=\left\{\aligned
&0\ \text{if $0\leqslant v\leqslant\frac{3}{2}-u$}, \\
&2u+2v-3\ \text{if $\frac{3}{2}-u\leqslant v\leqslant 2-u$},
\endaligned
\right.
$$
which gives
\begin{multline*}
\quad \quad \quad \quad F_P=\frac{1}{2}\int_{0}^{\frac{1}{2}}\int_{\frac{3}{2}-u}^{2-u}(8-4u-4v)(2u+2v-3)dvdu+\\
+\frac{1}{2}\int_{\frac{1}{2}}^{1}\int_{1}^{\frac{3}{2}-u}(7-4u-3v)(2u+2v-3)dvdu+\frac{1}{2}\int_{\frac{1}{2}}^{1}\int_{\frac{3}{2}-u}^{2-u}(8-4u-4v)(2u+2v-3)dvdu=\frac{31}{384},
\end{multline*}
so $S(W_{\bullet,\bullet,\bullet}^{S,C};P)=\frac{145}{192}+\frac{31}{384}=\frac{107}{128}<\frac{11}{12}$.
If $C^\prime$ is tangent to $C$ at the~point $P$, then
$$
\mathrm{ord}_P\Big(N(u,v)\big\vert_{C}\Big)=\left\{\aligned
&0\ \text{if $0\leqslant v\leqslant\frac{3}{2}-u$}, \\
&2(2u+2v-3)\ \text{if $\frac{3}{2}-u\leqslant v\leqslant 2-u$},\\
\endaligned
\right.
$$
which gives $F_P=\frac{31}{192}$, so $S(W_{\bullet,\bullet}^{S,C};P)=\frac{145}{192}+\frac{31}{192}=\frac{11}{12}$.
\end{proof}

Now, using \eqref{equation:2-6-main-inequality-refined} and Lemmas~\ref{lemma:2-6-smooth-conic-del-Pezzo}, \ref{lemma:2-6-smooth-conic-weak-del-Pezzo}, \ref{lemma:2-6-reducible-conic},
we see that $Z$ is a~curve.

\begin{lemma}
\label{lemma:2-6-curve-H1-H2-intersection}
One has $H_1\cdot Z\geqslant 1$ and $H_2\cdot Z\geqslant 1$.
\end{lemma}

\begin{proof}
If $H_2\cdot Z=0$, then $Z=C$, which is impossible by \eqref{equation:2-6-main-inequality} and Lemmas~\ref{lemma:2-6-smooth-conic-del-Pezzo}, \ref{lemma:2-6-smooth-conic-weak-del-Pezzo},~\ref{lemma:2-6-reducible-conic}.
Hence, we see that $H_2\cdot Z\geqslant 1$ and $\pi_2(Z)$ is a~curve.
Let us show that $H_1\cdot Z\geqslant 1$.

Suppose that $H_1\cdot Z=0$. Then  $Z$ must be an irreducible component of the~curve $C_1$.
If~$C_1$ is reduced, then arguing exactly as in the~proofs of Lemmas~\ref{lemma:2-6-smooth-conic-del-Pezzo}, \ref{lemma:2-6-smooth-conic-weak-del-Pezzo}, \ref{lemma:2-6-reducible-conic},
we obtain a~contradiction with \cite[Corollary~1.109]{Book}. Thus, we see that $C_1$ is not reduced.

So far, the~point $P$ was a~point in $Z$. Let us choose $P\in Z$ such that $\pi_2(P)\in\pi_2(Z)\cap\Delta_2$.
Then  $C_2$ is singular. But it is smooth at $P$ by Lemma~\ref{lemma:2-6-two-fibers},
which fits our assumption~above.
Then  $S(W_{\bullet,\bullet}^S;C)=\frac{3}{4}$ and $S(W_{\bullet, \bullet,\bullet}^{S,C};P)\leqslant \frac{11}{12}$ by Lemma~\ref{lemma:2-6-reducible-conic}, which contradicts \eqref{equation:2-6-main-inequality}.
\end{proof}

Both $\pi_1(Z)$ and $\pi_2(Z)$ are curves.
Similar to what we did in the~proof of Lemma~\ref{lemma:2-6-curve-H1-H2-intersection}, let us choose the~point $P\in Z$ such that $\pi_1(P)\in\Delta_1$.
Then  $C_1$ is singular at $P$, which implies that $C_2$ is smooth at $P$ by Lemma~\ref{lemma:2-6-two-fibers}.
Now, using \eqref{equation:2-6-main-inequality} and Lemmas~\ref{lemma:2-6-smooth-conic-weak-del-Pezzo} and~\ref{lemma:2-6-reducible-conic}, we see that
$C=C_2$, the~curve $C_2$ is smooth, the~divisor $-K_S$ is ample.

We see that $S$ is a~del Pezzo surface, and $\eta\colon S\to\mathbb{P}^2$ is a~double cover ramified in a~smooth quartic curve,
so we are almost in the~same position as in the~proof of Lemma~\ref{lemma:2-6-smooth-conic-del-Pezzo}.
But now, we have one small advantage: the~point $\eta(P)$ is contained in the~ramification divisor of the~double cover $\eta$,
because $C_1$ is singular at $P$.
Then  $|-K_S|$ contains a~unique curve that is singular at~$P$.
Denote this curve by~$R$. We have the~following possibilities:
\begin{enumerate}
\item[(a)] $R$ is an irreducible curve that has a~nodal singularity at $P$;

\item[(b)] $R$ is an irreducible curve that has a~cuspidal singularity at $P$;

\item[(c)] $R=R_1+R_2$ for two $(-1)$-curves in $S$ that intersect transversally at $P$;

\item[(d)] $R=R_1+R_2$ for two $(-1)$-curves in $S$ that are tangent at $P$.
\end{enumerate}

Let $f\colon\widetilde{S}\to S$ be the~blow up of the~point $P$.
Denote by $\widetilde{R}$ and $\widetilde{C}$ the~proper transforms on the~surface $\widetilde{S}$ of the~curves $R$ and $C$, respectively.
Fix $u\in[0,1]$ and $v\in\mathbb{R}_{\geqslant 0}$. Then
$$
f^*\big(P(u)\vert_{S}\big)-vE\sim_{\mathbb{Q}} \widetilde{R}+(1-u)\widetilde{C}+(3-u-v)E.
$$
Let $\widetilde{P}(u,v)$ be the~positive part of the~Zariski decomposition of $\widetilde{R}+(1-u)\widetilde{C}+(3-u-v)E$,
and let $\widetilde{N}(u,v)$ its negative part.
Then  it follows from \cite[Remark~1.113]{Book} that
\begin{equation}
\label{equation:2-6-final}
1\geqslant \frac{A_X(\mathbf{F})}{S_X(\mathbf{F})}\geqslant\min\left\{\frac{1}{S_X(S)},
\frac{2}{S(W_{\bullet,\bullet}^{\widetilde{S}};E)},\inf_{O\in E}\frac{1}{S(W_{\bullet, \bullet,\bullet}^{\widetilde{S},E};O)}\right\},
\end{equation}
where $S(W_{\bullet,\bullet}^{\widetilde{S}};E)$ and $S(W_{\bullet, \bullet,\bullet}^{\widetilde{S},E};O)$
are defined  in \cite{Book} similar to $S(W_{\bullet,\bullet}^S;C)$ and $S(W_{\bullet, \bullet,\bullet}^{S,C};P)$.
These two numbers can be computed using   \cite[Remark~1.113]{Book}. Namely, we have
$$
S\big(W_{\bullet,\bullet}^{\widetilde{S}};E\big)=\frac{1}{4}\int_0^1\int_0^{\infty}\big(\widetilde{P}(u,v)\big)^2dvdu
$$
and
$$
S\big(W_{\bullet, \bullet,\bullet}^{\widetilde{S},E};O\big)=\frac{1}{4}\int_0^1\int_0^{\infty}\Big(\big(\widetilde{P}(u,v)\cdot E\big)\Big)^2dvdu+
F_O,
$$
where $O$ is a~point in $E$ and
$$
F_O=\frac{1}{2}\int_0^1\int_0^{\infty}\big(\widetilde{P}(u,v)\cdot E\big)\mathrm{ord}_O\Big(\widetilde{N}(u,v)\big|_{E}\Big)dvdu.
$$
Let us use these formulas to estimate $S(W_{\bullet,\bullet}^{\widetilde{S}};E)$ and $S(W_{\bullet, \bullet,\bullet}^{\widetilde{S},E};O)$.

If the curve $R$ is irreducible, the~intersections of the~curves $\widetilde{C}$, $\widetilde{R}$, $E$ can be computed as follows:
$\widetilde{C}^2=-1$, $\widetilde{C}\cdot\widetilde{R}=0$, $\widetilde{C}\cdot E=1$, $\widetilde{R}^2=-2$, $\widetilde{R}\cdot E=2$, $E^2=-1$.
If $R$ is reducible, then $\widetilde{R}=\widetilde{R}_1+\widetilde{R}_2$ for two smooth irreducible curves $\widetilde{R}_1+\widetilde{R}_2$
such that the~intersection form of the~curves $\widetilde{C}$, $\widetilde{R}_1$,  $\widetilde{R}_1$ and $E$ is given in the~following table:
\begin{center}
\renewcommand\arraystretch{1.4}
\begin{tabular}{|c||c|c|c|c|}
\hline
$\bullet$  & $\widetilde{C}$ & $\widetilde{R}_1$ & $\widetilde{R}_1$ & $E$ \\
\hline\hline
$\widetilde{C}$        & $-1$ & $0$ & $0$ & $1$ \\
\hline
$\widetilde{R}_1$        & $0$ & $-2$ & $1$ & $1$\\
\hline
$\widetilde{R}_1$        & $0$ & $1$ & $-2$ & $1$\\
\hline
$E$ & $1$ & $1$ & $1$ & $-1$ \\
\hline
\end{tabular}
\end{center}
Then  $\widetilde{R}+(1-u)\widetilde{C}+(3-u-v)E$ is pseudoeffective $\iff$ $v\leqslant 3-u$.
Moreover, we have
$$
\widetilde{P}(u,v)=\left\{\aligned
&\widetilde{R}+(1-u)\widetilde{C}+(3-u-v)E\ \text{if $0\leqslant v\leqslant 2-u$}, \\
&(1-u)\widetilde{C}+(3-u-v)(E+\widetilde{R})\ \text{if $2-u\leqslant v\leqslant 2$}, \\
&(3-u-v)(E+\widetilde{R}+\widetilde{C})\ \text{if $2\leqslant v\leqslant 3-u$}, \\
\endaligned
\right.
$$
$$
\widetilde{N}(u,v)=\left\{\aligned
&0\ \text{if $0\leqslant v\leqslant 2-u$}, \\
&(v+u-2)\widetilde{R}\ \text{if $2-u\leqslant v\leqslant 2$}, \\
&(v+u-2)\widetilde{R}+(v-2)\widetilde{C}\ \text{if $2\leqslant v\leqslant 3-u$},\\
\endaligned
\right.
$$
$$
\big(\widetilde{P}(u,v)\big)^2=
\left\{\aligned
&6-v^2-4u\ \text{if $0\leqslant v\leqslant 2-u$}, \\
&14+2u^2+4uv+v^2-12u-8v\ \text{if $2-u\leqslant v\leqslant 2$}, \\
&2(3-u-v)^2\ \text{if $2\leqslant v\leqslant 3-u$}, \\
\endaligned
\right.
$$
Now, integrating, we obtain $S(W_{\bullet,\bullet}^{\widetilde{S}};E)=\frac{17}{12}$.

Fix a~point $O\in E$. To estimate $S(W_{\bullet, \bullet,\bullet}^{\widetilde{S},E};O)$, first we observe that
$$
\widetilde{P}(u,v)\cdot E=
\left\{\aligned
&v\ \text{if $0\leqslant v\leqslant 2-u$}, \\
&4-2u-v\ \text{if $2-u\leqslant v\leqslant 2$}, \\
&6-2u-2v\ \text{if $2\leqslant v\leqslant 3-u$}.\\
\endaligned
\right.
$$
Therefore, integrating $(\widetilde{P}(u,v)\cdot E)^2$, we obtain
$S(W_{\bullet, \bullet,\bullet}^{Y,\widetilde{S}};O)=\frac{13}{24}+F_O$.

If $O\not\in\widetilde{C}\cup\widetilde{R}$, then $F_O=0$.
Similarly, if $O\in\widetilde{C}$, then $O\not\in\widetilde{R}$, which gives
$$
F_O=\frac{1}{2}\int_0^1\int_2^{3-u}(6-2u-2v)(v-2)dvdu=\frac{1}{24}.
$$
Finally, if $O\in\widetilde{R}$, then $O\not\in\widetilde{C}$, which gives
$$
F_O\leqslant\frac{1}{2}\int_0^1\int_{2-u}^{2}2(4-2u-v)(v+u-2)dvdu+\frac{1}{2}\int_0^1\int_{2}^{3-u}2(6-2u-2v)(v+u-2)dvdu=\frac{7}{24}.
$$
Summarizing, we get $S(W_{\bullet, \bullet,\bullet}^{\widetilde{S},E};O)\leqslant\frac{5}{6}$ for every point $O\in E$,
which contradicts \eqref{equation:2-6-final}, because $S_X(S)=\frac{5}{12}$ and $S(W_{\bullet,\bullet}^{\widetilde{S}};E)=\frac{17}{12}<2$.
This shows that $X$ is K-stable.

\begin{corollary}
\label{corollary:2-6-final}
All smooth Fano threefolds in the~family \textnumero 2.6 are K-stable.
\end{corollary}

\section{Family \textnumero 2.7}
\label{section:2-7}

Now, let us fix three smooth quadric hypersurfaces  $\mathscr{Q}$, $\mathscr{Q}^\prime$ and $\mathscr{Q}^{\prime\prime}$ in $\mathbb{P}^4$
such that their~intersection is a~smooth curve of degree $8$ and genus $5$.
We set~$\mathscr{C}=\mathscr{Q}\cap\mathscr{Q}^\prime\cap\mathscr{Q}^{\prime\prime}$.
Let $\pi\colon X\to \mathscr{Q}$ be the~blow up of the~smooth curve $\mathscr{C}$.
Then  $X$ is a~smooth Fano~threefold, which is contained in the family~\textnumero 2.7.
Moreover, all smooth Fano threefolds in this family can be obtained in this way.
Note that $-K_X^3=14$ and $\mathrm{Aut}(X)$ is finite \cite{CheltsovShramovPrzyjalkowski}.

The~pencil generated by the~surfaces $\mathscr{Q}^\prime\vert_{\mathscr{Q}}$~and~$\mathscr{Q}^{\prime\prime}\vert_{\mathscr{Q}}$
gives a rational map $\mathscr{Q}\dasharrow\mathbb{P}^1$, which fits the following commutative diagram:
$$
\xymatrix{
&X\ar@{->}[ld]_{\pi}\ar@{->}[rd]^{\phi}&\\%
\mathscr{Q}\ar@{-->}[rr]&&\mathbb{P}^1}
$$
where $\phi$ is a~fibration into quartic del Pezzo surfaces.
Let $E$ be the~$\pi$-exceptional surface,
and let $H=\pi^*(\mathcal{O}_{\mathbb{P}^4}(1)\vert_{\mathscr{Q}})$.
Then  $-K_X\sim 3H-E$, the~morphism $\phi$ is given by the~linear system $|2H-E|$, and $E\cong\mathscr{C}\times\mathbb{P}^1$.

\begin{lemma}
\label{lemma:2-7-Du-Val-surfaces}
Let $S$ be a fiber of the morphism $\phi$. Then  $S$ has at most Du Val singularities.
\end{lemma}

\begin{proof}
Since $E\vert_{S}$ is a smooth curve, the surface $S$ is smooth along $E\vert_{S}$,
which implies that it has at most isolated singularities. Hence, we conclude that $S$ is normal and irreducible.

Note that $S\cong\pi(S)$. Suppose that the singularities of the surface $\pi(S)$ are not Du~Val.
Then  it follows from \cite[Theorem~1]{Brenton} that $\pi(S)$ is a cone in $\mathbb{P}^4$ over a quartic elliptic curve.
Let $P$ be the vertex of the cone $\pi(S)$, and let $T_P$ be the hyperplane section of the quadric hypersurface $\mathscr{Q}$ that is singular at~$P$.
Then  $T_P$ contains all lines in $\mathscr{Q}$ that pass through~$P$,
which implies that $T_P$ contains $\pi(S)$.
This is impossible, since $T_P$ is a quadric cone.
\end{proof}

The goal of this section is to prove that $X$ is K-stable.
To do this, we fix a point $P\in X$.
By \cite{Fujita2019Crelle,Li}, to prove that $X$ is K-stable, it is enough to show that $\delta_P(X)>1$.

Let $S$ be the fiber of the morphism $\phi$ containing $P$.
Then  $S$ is a quartic del Pezzo~surface, and $S$ has at most Du Val singularity at $P$ by Lemma~\ref{lemma:2-7-Du-Val-surfaces}.
Moreover, if $S$ is singular at~$P$,
then  $P$ is a singular point of the surface $S$
of type $\mathbb{A}_1$, $\mathbb{A}_2$, $\mathbb{A}_3$, $\mathbb{A}_4$, $\mathbb{D}_4$ or $\mathbb{D}_5$, see \cite{Coray1988}.

\begin{lemma}
\label{lemma:2-7-quartic-del-Pezzo-delta}
If $\delta_P(S)>\frac{54}{55}$ or $P\in\mathrm{Sing}(S)$ and $\delta_P(S)>\frac{27}{28}$, then $\delta_P(X)>1$.
\end{lemma}

\begin{proof}
Take $u\in\mathbb{R}_{\geqslant 0}$. Since $S\sim 2H-E$, we have
$$
-K_X-uS\sim_{\mathbb{R}} (3-2u)H-(u-1)E\sim_{\mathbb{R}}\Big(\frac{3}{2}-u\Big)S+\frac{1}{2}E.
$$
Using this, we conclude that the divisor $-K_X-uS$ is pseudoeffective if and only if $u\leqslant \frac{3}{2}$.
For $u\leqslant \frac{3}{2}$, let $P(u)$ be the positive part of Zariski decomposition of the divisor $-K_X-uS$,
and let $N(u)$  be the negative part of Zariski decomposition of the divisor $-K_X-uS$.
Then
$$
P(u)=\left\{\aligned
&-K_X-uS\ \text{if $0\leqslant u\leqslant 1$}, \\
&(3-2u)H\ \text{if $1\leqslant u\leqslant \frac{3}{2}$},
\endaligned
\right.
$$
and
$$
N(u)=\left\{\aligned
&0\ \text{if $0\leqslant u\leqslant 1$}, \\
&(u-1)E\ \text{if $1\leqslant u\leqslant \frac{3}{2}$}.
\endaligned
\right.
$$
This gives
$$
S_{X}(S)=\frac{1}{14}\int_{0}^{\frac{3}{2}}\big(P(u)\big)^3du=\frac{1}{14}\int_{0}^{1}(14-12u)du+\frac{1}{14}\int_{1}^{\frac{3}{2}} 2(3-2u)^{3}du=\frac{33}{56}<1.
$$
Now, using \cite[Theorem~3.3]{AbbanZhuang} and \cite[Corollary 1.102]{Book}, we get
\begin{equation}
\label{equation:2-7-S-delta}
\delta_P(X)\geqslant \min\left\{\frac{1}{S_X(S)},\inf_{\substack{F/S\\ P\in C_S(F)}}\frac{A_S(F)}{S(W_{\bullet,\bullet}^S;F)}\right\},
\end{equation}
where the infimum is taken by all prime divisors $F$ over the~surface $S$ such that $P\in C_S(F)$,
and $S(W_{\bullet,\bullet}^S;F)$ can be computed using \cite[Corollary~1.108]{Book} as follows:
$$
S\big(W_{\bullet,\bullet}^S; F\big)=\frac{3}{14}\int_0^{\frac{3}{2}}\big(P(u)\big\vert_{S}\big)^2\mathrm{ord}_{F}\big(N(u)\big\vert_{S}\big)du+\frac{3}{14}\int_{0}^{\frac{3}{2}}\int_0^\infty \mathrm{vol}\big(P(u)\big\vert_{S}-vF\big)dvdu.
$$

Let $E_S=E|_S$. Then  $E_S\in |-2K_S|$ and $E_S\cong\mathscr{C}$. Moreover, the surface $S$ is smooth in a neighborhood of the curve $E_S$.
Furthermore, we have
$$
P(u)\big\vert_{S}=\left\{\aligned
&-K_S\ \text{if $0\leqslant u\leqslant 1$}, \\
&(3-2u)(-K_S)\ \text{if $1\leqslant u\leqslant \frac{3}{2}$},
\endaligned
\right.
$$
and
$$
N(u)\big\vert_{S}=\left\{\aligned
&0\ \text{if $0\leqslant u\leqslant 1$}, \\
&(u-1)E_S\ \text{if $1\leqslant u\leqslant \frac{3}{2}$}.
\endaligned
\right.
$$
Let $F$ be any prime divisor over $S$ such that $P\in C_S(F)$. Then
\begin{multline*}
\quad S\big(W_{\bullet,\bullet}^S; F\big)=\frac{3}{14}\int_{1}^{\frac{3}{2}}4(u-1)(3-2u)^{2}\mathrm{ord}_{F}\big(E_{S}\big)du+\\
+\frac{3}{14}\int_{0}^{1} \int_{0}^{\infty}\mathrm{vol}\big(-K_S-vF\big)dvdu+\\
\quad\quad\quad\quad\quad\quad\quad\quad\quad\quad+\frac{3}{14}\int_{1}^{\frac{3}{2}}\int_{0}^{\infty}\mathrm{vol}\big((3-2u)(-K_S)-vF\big)dvdu=\\
=\frac{\mathrm{ord}_{F}\big(E_{S}\big)}{56}+\frac{3}{14}\int_{0}^{\infty}\mathrm{vol}\big(-K_S-vF\big)dv+\frac{3}{14}\int_{1}^{\frac{3}{2}}(3-2u)^3\int_{0}^{\infty}\mathrm{vol}\big(-K_S-vF\big)dvdu=\\
=\frac{\mathrm{ord}_{F}\big(E_{S}\big)}{56}+\frac{3}{14}\int_{0}^{\infty}\mathrm{vol}\big(-K_S-vF\big)dv+\frac{3}{112}\int_{0}^{\infty}\mathrm{vol}\big(-K_S-vF\big)dv=\\
=\frac{\mathrm{ord}_{F}\big(E_{S}\big)}{56}+\frac{27}{112}\int_{0}^{\infty}\mathrm{vol}\big(-K_S-vF\big)dv=\frac{\mathrm{ord}_{F}\big(E_{S}\big)}{56}+\frac{27}{28}S_S(F)\leqslant\frac{A_{S}(F)}{56}+\frac{27A_{S}(F)}{28\delta_{P}(S)},
\end{multline*}
because log pair $(S,E_S)$ is log canonical. Therefore, if  $\delta_P(S)>\frac{54}{55}$, \eqref{equation:2-7-S-delta} gives $\delta_P(X)>1$.
Similarly, if $P\in\mathrm{Sing}(S)$, $P\not\in E_S$, so $\mathrm{ord}_{F}(E_{S})=0$, which implies that
$$
S\big(W_{\bullet,\bullet}^S; F\big)=\frac{27}{28}S_S(F)\leqslant\frac{27A_{S}(F)}{28\delta_{P}(S)}.
$$
Hence, in this case, it follows from \eqref{equation:2-7-S-delta} that $\delta_P(X)>1$ provided that $\delta_P(S)>\frac{27}{28}$.
\end{proof}

\begin{corollary}
\label{corollary:2-7-quartic-del-Pezzo-delta}
If $S$ is smooth, then $\delta_P(X)>1$.
\end{corollary}

\begin{proof}
If $S$ is smooth, then $\delta(S)=\frac{4}{3}$ by \cite[Lemma~2.12]{Book}. Now apply Lemma~\ref{lemma:2-7-quartic-del-Pezzo-delta}.
\end{proof}

\begin{corollary}
\label{corollary:2-7-quartic-del-Pezzo-delta-singular}
If $P$ is not contained in a line in $S$, then $\delta_P(X)>1$.
\end{corollary}

\begin{proof}
If $P$ is not contained in a line in $S$, then $P$ is a smooth point of the surface $S$,
and the blow up of the surface $S$ at this point is a (possibly singular) cubic surface in $\mathbb{P}^3$.
Thus, arguing exactly as in the end of the proof of \cite[Lemma~2.12]{Book}, we obtain $\delta_P(S)\geqslant \frac{3}{2}$,
which implies that $\delta_P(X)>1$ by Lemma~\ref{lemma:2-7-quartic-del-Pezzo-delta}.
\end{proof}

Now, let $T$ be a surface in the linear system $|H|$ such that $P\in T$, and let $\mathcal{Q}=\pi(T)$.
Then  $\mathcal{Q}$ is a hyperplane section of the hypersurface $\mathscr{Q}$, so both $\mathcal{Q}$ and $T$ are irreducible.
In~the following, we will choose $T$ such that the surface $\mathcal{Q}$ is smooth, so that $\mathcal{Q}\cong\mathbb{P}^1\times\mathbb{P}^1$.

\begin{lemma}
\label{lemma:2-7-S-singular-at-P-irreducible}
Suppose that $T$ is a general surface in the linear system $|H|$ such that $P\in T$.
Then  the (scheme) intersection $S\cap T$ is an irreducible reduced curve.
\end{lemma}

\begin{proof}
Let $\rho\colon\widetilde{S}\to S$ be a blow up of the quartic del Pezzo surface $S$ at the point $P$,
and let $Z$ be the proper transform of the curve $T\vert_{S}$ on the surface $\widetilde{S}$.
Then  $|Z|$ has no base points and gives the morphism $\eta\colon\widetilde{S}\to\mathbb{P}^3$
that fits the following commutative diagram:
$$
\xymatrix{
&\widetilde{S}\ar@{->}[ld]_{\rho}\ar@{->}[rd]^{\eta}&\\%
S\ar@{-->}[rr]&&\mathbb{P}^3}
$$
where $S\dasharrow\mathbb{P}^3$ is a projection from $P$. Moreover, if $P$ is a smooth point of the surface~$S$,
then $Z^2=3$, and the image of the morphism $\eta$ is an irreducible cubic surface in $\mathbb{P}^3$.
Similarly, if $P$ is a singular point of the surface $S$, then we have \mbox{$Z^2=4-\mathrm{mult}_P(S)=2$},
and the image of the morphism $\eta$ is an irreducible quadric surface.
Therefore, we conclude that the curve $Z$ must be irreducible and reduced (by Bertini theorem),
which implies that the~intersection $S\cap T$ is also irreducible and reduced.
\end{proof}

\begin{remark}
\label{remark:2-7-S-T-irreducible}
Suppose that $S$ is singular at $P$, and $T$ is a general surface in $|H|$ that passes through the~point $P$.
Then, choosing appropriate coordinates $[x:y:z:t:w]$ on $\mathbb{P}^4$,
we~may assume that $\pi(P)=[0:0:0:0:1]$, and the surface $\pi(S)$ is given in $\mathbb{P}^4$ by
$$
\left\{\aligned
&at^2+btx+f_2(x,y,z)=0, \\
&wt=g_2(x,y,z), \\
\endaligned
\right.
$$
where $a$ and $b$ are complex numbers, $f_2(x,y,z)$ and $g_2(x,y,z)$ are non-zero quadratic homogeneous polynomials.
In the~chart $w\ne 0$, the surface $\pi(S)$ is given~by
$$
\left\{\aligned
&at^2+btx+f_2(x,y,z)=0, \\
&t=g_2(x,y,z), \\
\endaligned
\right.
$$
where now we consider $x$, $y$, $z$, $t$ as affine coordinates on $\mathbb{C}^4$.
Then $\pi(S)\cap\mathcal{Q}$ is cut out on the~surface $\pi(S)$~by $c_1x+c_2y+c_3z+c_4t=0$,
where $c_1$, $c_2$, $c_3$, $c_4$ are general numbers.
The affine part of the surface $\pi(S)$ is isomorphic to the hypersurface in $\mathbb{C}^3$ given by
$$
ag_2^2(x,y,z)+bxg_2(x,y,z)+f_2(x,y,z)=0,
$$
and the affine part of the curve $\pi(S)\cap\mathcal{Q}$ is cut out by \mbox{$c_1x+c_2y+c_3z+c_4g_2(x,y,z)=0$}.
If $P$ is a singular point of the surface $S$ of type $\mathbb{D}_4$ or $\mathbb{D}_5$,
then $S\cap T$ has an ordinary cusp at the point $P$, which easily implies that the~intersection $S\cap T$ is reduced and irreducible.
Similarly, if $P$ is a Du Val singular point of the surface $S$ of type $\mathbb{A}_1$, $\mathbb{A}_2$, $\mathbb{A}_3$ or $\mathbb{A}_4$,
then  the~intersection $S\cap T$ has an isolated ordinary double singularity at $P$.
\end{remark}

Observe that the morphism $\pi\colon X\to\mathscr{Q}$ induces a birational morphism $\varpi\colon T\to \mathcal{Q}$,
and the morphism $\phi\colon X\to\mathbb{P}^1$ induces a fibration $\varphi\colon T\to\mathbb{P}^1$ that both fit the~following commutative diagram:
$$
\xymatrix{
&T\ar@{->}[ld]_{\varpi}\ar@{->}[rd]^{\varphi}&\\%
\mathcal{Q}\ar@{-->}[rr]&&\mathbb{P}^1}
$$
where $\mathcal{Q}\dasharrow\mathbb{P}^1$ is a map given by the pencil generated by the curves
$\mathscr{Q}^\prime\vert_{\mathcal{Q}}$~and~$\mathscr{Q}^{\prime\prime}\vert_{\mathcal{Q}}$.
In~the~following, we will always choose $T\in|H|$ such that the quadric surface $\mathcal{Q}$ is smooth,
and $T$ is either smooth or has one isolated ordinary singularity,
which would imply that a~general fiber of the induced fibration $\varphi\colon T\to\mathbb{P}^1$ is a smooth elliptic curve.
Let $\mathcal{C}=S\vert_{T}$.
Then $\mathcal{C}$ is the fiber of the (elliptic) fibration $\varphi$ that contains the~point $P$.

Let $u$ be a non-negative real number.
Then $-K_X-uT\sim_{\mathbb{R}}(3-u)H-E\sim_{\mathbb{R}}(1-u)H+S$,
which implies that $-K_X-uT$ is nef $\iff$ $-K_X-uT$ is pseudoeffective $\iff$ $u\in[0,1]$.
Integrating, we get $S_{X}(T)=\frac{9}{28}<1$. For simplicity, we let $P(u)=-K_X-uT$.

\begin{lemma}
\label{lemma:2-7-S-singular-at-P}
Suppose that $S$ is singular at $P$. Then $\delta_P(X)>1$.
\end{lemma}

\begin{proof}
Now, let us choose $T\in|H|$ such that $T$ is a general surface in $|H|$ that contains~$P$.
Then $T$ and $\mathcal{Q}$ are smooth, and $\varpi$ is a blow up of the~eight intersection points~$\mathscr{C}\cap \mathcal{Q}$.
Moreover, by Lemma~\ref{lemma:2-7-S-singular-at-P-irreducible}, the curve $\mathcal{C}$ is an irreducible singular curve of arithmetic genus~$1$.
Thus, we have $P=\mathrm{Sing}(\mathcal{C})$. Furthermore, using Remark~\ref{remark:2-7-S-T-irreducible}, we see that
\begin{itemize}
\item either $\mathcal{C}$ has an isolated ordinary double singularity at $P$,
\item or the curve $\mathcal{C}$ has an ordinary cusp at the point $P$.
\end{itemize}

Recall that $\mathcal{Q}$ is a smooth quadric surface, so that it contains exactly two lines that pass though $\pi(P)$.
Since $T$ is chosen to be general, these lines are disjoint from   $\mathscr{C}\cap \mathcal{Q}$.
Denote by $L_1$ and $L_2$ the proper transforms of these lines on $T$. Then
$$
P(u)\big\vert_{T}\sim_{\mathbb{R}}\big((1-u)H+S\big)\big\vert_{T}\sim_{\mathbb{R}}(1-u)(L_1+L_2)+\mathcal{C}.
$$

Now, we let $\sigma\colon\widetilde{T}\to T$ be the blow up of the~point $P$,
we let $\mathbf{e}$ be the $\sigma$-exceptional~curve,
and we denote by $\widetilde{\mathcal{C}}$, $\widetilde{L}_1$, $\widetilde{L}_2$ the~proper transforms on $\widetilde{T}$ of the curves  $\mathcal{C}$, $L_1$, $L_2$,
respectively. Take a non-negative real number $v$. Then
\begin{equation}
\label{equation:2-7-C-singular-blow-up}
\sigma^*\big(P(u)\big\vert_{T}\big)-v\mathbf{e}\sim_{\mathbb{R}}\widetilde{\mathcal{C}}+(1-u)(\widetilde{L}_{1}+\widetilde{L}_{2})+(4-2u-v)\mathbf{e}.
\end{equation}
Note that the curves $\widetilde{\mathcal{C}}$, $\widetilde{L}_1$, $\widetilde{L}_2$ are disjoint.
Moreover, we have $\widetilde{L}_{1}^2=\widetilde{L}_{2}^2=-1$ and $\widetilde{\mathcal{C}}^2=-4$.
Thus, using~\eqref{equation:2-7-C-singular-blow-up}, we see that $\sigma^*(P(u)\vert_{T})-v\mathbf{e}$ is pseudoeffective $\iff$ $v\leqslant 4-2u$.

Let $\widetilde{P}(u,v)$ be the positive part of Zariski decomposition of the divisor $\sigma^*(P(u)\vert_{T})-v\mathbf{e}$,
and let $\widetilde{N}(u,v)$ be its negative part. Then it follows from \cite[Remark~1.113]{Book} that
\begin{equation}
\label{equation:2-7-C-singular-node}
\delta_P(X)\geqslant\min\left\{\frac{1}{S_X(T)},
\frac{2}{S(W_{\bullet,\bullet}^{\widetilde{T}};\mathbf{e})},\inf_{O\in \mathbf{e}}\frac{1}{S(W_{\bullet, \bullet,\bullet}^{\widetilde{T},\mathbf{e}};O)}\right\},
\end{equation}
where $S(W_{\bullet,\bullet}^{\widetilde{T}};\mathbf{e})$ and $S(W_{\bullet, \bullet,\bullet}^{\widetilde{T},\mathbf{e}};O)$ are defined  in \cite[\S~1.7]{Book},
and these two numbers can be computed using formulas described in \cite[Remark~1.113]{Book}.
Namely, we have
$$
S\big(W_{\bullet,\bullet}^{\widetilde{T}};\mathbf{e}\big)=\frac{3}{14}\int_0^1\int_0^{4-2u}\big(\widetilde{P}(u,v)\big)^2dvdu
$$
and
$$
S\big(W_{\bullet, \bullet,\bullet}^{\widetilde{T},\mathbf{e}};O\big)=\frac{3}{14}\int_0^1\int_0^{4-2u}\Big(\big(\widetilde{P}(u,v)\cdot\mathbf{e}\big)\Big)^2dvdu+\frac{3}{7}\int_0^1\int_0^{4-2u}\big(\widetilde{P}(u,v)\cdot\mathbf{e}\big)\mathrm{ord}_O\Big(\widetilde{N}(u,v)\big|_{\mathbf{e}}\Big)dvdu,
$$
where $O$ is a~point in $\mathbf{e}$.
Moreover, using~\eqref{equation:2-7-C-singular-blow-up}, we compute $\widetilde{P}(u,v)$ and $\widetilde{N}(u,v)$ as follows:
$$
\widetilde{P}(u,v)=\left\{\aligned
&\widetilde{\mathcal{C}}+(1-u)(\widetilde{L}_{1}+\widetilde{L}_{2})+(4-2u-v)\mathbf{e}\ \text{if $0\leqslant v\leqslant 2-2u$}, \\
&\frac{4-2u-v}{2}\widetilde{\mathcal{C}}+(1-u)(\widetilde{L}_{1}+\widetilde{L}_{2})+(4-2u-v)\mathbf{e}\ \text{if $2-2u\leqslant v\leqslant 3-u$}, \\
&\frac{4-2u-v}{2}\big(\widetilde{\mathcal{C}}+2(\widetilde{L}_{1}+\widetilde{L}_{2})+2\mathbf{e}\big)\ \text{if $3-u\leqslant v\leqslant 4-2u$}, \\
\endaligned
\right.
$$
and
$$
\widetilde{N}(u,v)=\left\{\aligned
&0\ \text{if $0\leqslant v\leqslant 2-2u$}, \\
&\frac{v-2+2u}{2}\widetilde{\mathcal{C}}\ \text{if $2-2u\leqslant v\leqslant 3-u$}, \\
&\frac{v-2+2u}{2}\widetilde{\mathcal{C}}+(v-3+u)(\widetilde{L}_{1}+\widetilde{L}_{2}) \ \text{if $3-u\leqslant v\leqslant 4-2u$}.\\
\endaligned
\right.
$$
Thus, we have
$$
\big(\widetilde{P}(u,v)\big)^2=\left\{\aligned
&2(1-u)(5-u)-v^{2}\ \text{if $0\leqslant v\leqslant 2-2u$}, \\
&2(1-u)(7-3u-2v)\ \text{if $2-2u\leqslant v\leqslant 3-u$}, \\
&2(4-2 u-v)^{2}\ \text{if $3-u\leqslant v\leqslant 4-2u$}, \\
\endaligned
\right.
$$
and
$$
\widetilde{P}(u,v)\cdot\mathbf{e}=\left\{\aligned
&v\ \text{if $0\leqslant v\leqslant 2-2u$}, \\
&2(1-u)\ \text{if $2-2u\leqslant v\leqslant 3-u$}, \\
&2(4-2u-v)\ \text{if $3-u\leqslant v\leqslant 4-2u$}. \\
\endaligned
\right.
$$
Now, integrating, we get $S(W_{\bullet,\bullet}^{\widetilde{T}};\mathbf{e})=\frac{51}{28}<2$.

Let $O$ be an arbitrary point in $\mathbf{e}$.
If $O\not\in\widetilde{L}_1\cup\widetilde{L}_2\cup\widetilde{\mathcal{C}}$, then we compute $S(W_{\bullet, \bullet,\bullet}^{\widetilde{T},\mathbf{e}};O)=\frac{4}{7}$.
Similarly, if $O\in\widetilde{L}_1\cup\widetilde{L}_2$, then $S(W_{\bullet, \bullet,\bullet}^{\widetilde{T},\mathbf{e}};O)=\frac{17}{28}$.
Finally, if $O\in\widetilde{\mathcal{C}}$, then
\begin{multline*}
S\big(W_{\bullet, \bullet,\bullet}^{\widetilde{T},\mathbf{e}};O\big)=\frac{4}{7}+\frac{3}{7}\int_{0}^{1}\int_{2-2u}^{3-u} 2(1-u)\frac{v-2+2 u}{2}\mathrm{ord}_O\big(\widetilde{\mathcal{C}}\big\vert_{\mathbf{e}}\big)dvdu+\\
+\frac{3}{7}\int_{0}^{1}\int_{3-u}^{4-2u} 2(4-2u-v)\frac{v-2+2 u}{2}\mathrm{ord}_O\big(\widetilde{\mathcal{C}}\big\vert_{\mathbf{e}}\big)dvdu=\frac{4}{7}+\frac{17}{56}\mathrm{ord}_O\big(\widetilde{\mathcal{C}}\big\vert_{\mathbf{e}}\big).
\end{multline*}
Hence, if $\widetilde{\mathcal{C}}$ intersects $\mathbf{e}$ transversally, then $S(W_{\bullet, \bullet,\bullet}^{\widetilde{T},\mathbf{e}};O)<1$, so that $\delta_P(X)>1$ by \eqref{equation:2-7-C-singular-node}.

Therefore, to complete the proof of the lemma, we may assume that $\widetilde{\mathcal{C}}$ is tangent to $\mathbf{e}$.
This means that $\mathcal{C}$ has a cusp at $P$, and the intersection $\widetilde{\mathcal{C}}\cap\mathbf{e}$ consists of a single point.

Now, as in the proof of Lemma~\ref{lemma:2-4-C-cuspidal}, we consider the~following commutative diagram:
$$
\xymatrix{
\widetilde{T}\ar@{->}[d]_{\sigma}&&\overline{T}\ar@{->}[ll]_{\gamma}\ar@{->}[d]^{\upsilon}\\%
T && \widehat{T}\ar@{->}[ll]_{\varsigma}}
$$
where $\gamma$ is a~composition of the blow up of the point $\widetilde{\mathcal{C}}\cap\mathbf{e}$
with the blow up of the unique intersection point of the~proper transforms of the curves $\widetilde{\mathcal{C}}$ and $\mathbf{e}$,
$\upsilon$ is the~birational contraction of all~$(\sigma\circ\gamma)$-exceptional curves that are not $(-1)$-curve,
and $\varsigma$ is the~birational contraction of the~proper transform of the~unique $\gamma$-exceptional curve that is $(-1)$-curve.
Then $\widehat{T}$ has two singular points:
\begin{enumerate}
\item a~cyclic quotient singularity of type $\frac{1}{2}(1,1)$, which we denote by $O_2$;
\item a~cyclic quotient singularity of~type~$\frac{1}{3}(1,1)$, which we denote by $O_3$.
\end{enumerate}

Let $\mathbf{f}$ be the~$\varsigma$-exceptional curve,
let $\widehat{\mathcal{C}}$ be the~proper transform on $\widehat{T}$ of the~curve $\mathcal{C}$,
and let $\widehat{L}_1$ and $\widehat{L}_2$ be the~proper transforms on $\widehat{T}$ of the~curves $L_1$ and $L_2$, respectively.
Then~the~curves $\mathbf{f}$, $\widehat{\mathcal{C}}$, $\widehat{L}_1$, $\widehat{L}_2$ are smooth,
and the~curve $\mathbf{f}$ contains both points $O_2$~and~$O_3$, which are not contained in  $\widehat{\mathcal{C}}$.
Moreover, we have
$$
\widehat{L}_1\cap\widehat{L}_2=\widehat{L}_1\cap\mathbf{f}=\widehat{L}_2\cap\mathbf{f}=O_3.
$$
Furthermore, we have $A_T(\mathbf{f})=5$, $\varsigma^*(\mathcal{C})\sim\widehat{\mathcal{C}}+6\mathbf{f}$,
$\varsigma^*(L_1)\sim\widehat{L}_1+2\mathbf{f}$, $\varsigma^*(L_2)\sim\widehat{L}_2+2\mathbf{f}$,
and the~intersection form of the~curves $\mathbf{f}$, $\widehat{\mathcal{C}}$, $\widehat{L}_1$ and $\widehat{L}_2$ is given in the~following table:
\begin{center}
\renewcommand\arraystretch{1.4}
\begin{tabular}{|c||c|c|c|c|}
 \hline
  & $\mathbf{f}$ & $\widehat{\mathcal{C}}$ & $\widehat{L}_1$ & $\widehat{L}_2$ \\
  \hline
  \hline
$\mathbf{f}$             & $-\frac{1}{6}$ & $1$ &  $\frac{1}{3}$ & $\frac{1}{3}$ \\
  \hline
$\widehat{\mathcal{C}}$  & $1$            &  $-6$ & $0$ & $0$\\
  \hline
$\widehat{L}_1$          & $\frac{1}{3}$  &  $0$ & $-\frac{2}{3}$ & $\frac{1}{3}$ \\
  \hline
$\widehat{L}_2$          & $\frac{1}{3}$  &  $0$ & $\frac{1}{3}$ & $-\frac{2}{3}$\\
 \hline
\end{tabular}
\end{center}

For a non-negative real number $v$, we have
$$
\varsigma^*\big(P(u)\big\vert_{T}\big)-v\mathbf{f}\sim_{\mathbb{R}}\widetilde{\mathcal{C}}+(1-u)(\widetilde{L}_{1}+\widetilde{L}_{2})+(10-4u-v)\mathbf{f},
$$
which implies that the divisor $\varsigma^*(P(u)\vert_{T})-v\mathbf{f}$ is pseudoeffective if and only if $v\leqslant 10-4u$,
because the~intersection form of the~curves  $\widehat{\mathcal{C}}$, $\widehat{L}_1$, $\widehat{L}_2$ is negative definite.

Let $\widehat{P}(u,v)$ be the positive part of Zariski decomposition of the divisor $\varsigma^*(P(u)\vert_{T})-v\mathbf{f}$,
and let $\widehat{N}(u,v)$ be its negative part.
Set $\Delta_\mathbf{f}=\frac{1}{2}O_2+\frac{2}{3}O_3$.
By \cite[Remark~1.113]{Book}, we~get
\begin{equation}
\label{equation:2-7-C-singular-cusp}
\delta_P(X)\geqslant\min\left\{\frac{1}{S_X(T)},
\frac{5}{S(W_{\bullet,\bullet}^{\widehat{T}};\mathbf{f})},\inf_{O\in \mathbf{f}}\frac{1-\mathrm{ord}_O\big(\Delta_\mathbf{f}\big)}{S(W_{\bullet, \bullet,\bullet}^{\widehat{T},\mathbf{f}};O)}\right\},
\end{equation}
where $S(W_{\bullet,\bullet}^{\widehat{T}};\mathbf{f})$ and $S(W_{\bullet, \bullet,\bullet}^{\widehat{T},\mathbf{f}};O)$ are defined as $S(W_{\bullet,\bullet}^{\widetilde{T}};\mathbf{e})$ and $S(W_{\bullet, \bullet,\bullet}^{\widetilde{T},\mathbf{e}};O)$ used earlier.
Moreover, it follows from \cite[Remark~1.113]{Book} that
$$
S\big(W_{\bullet,\bullet}^{\widehat{T}};\mathbf{f}\big)=\frac{3}{14}\int_0^1\int_0^{10-4u}\big(\widehat{P}(u,v)\big)^2dvdu
$$
and
$$
S\big(W_{\bullet, \bullet,\bullet}^{\widehat{T},\mathbf{f}};O\big)=\frac{3}{14}\int_0^1\int_0^{10-4u}\Big(\big(\widehat{P}(u,v)\cdot\mathbf{f}\big)\Big)^2dvdu+\frac{3}{7}\int_0^1\int_0^{10-4u}\big(\widehat{P}(u,v)\cdot\mathbf{f}\big)\mathrm{ord}_O\Big(\widehat{N}(u,v)\big|_{\mathbf{f}}\Big)dvdu.
$$
Moreover, we compute $\widehat{P}(u,v)$ and $\widehat{N}(u,v)$ as follows:
$$
\widehat{P}(u,v)=\left\{\aligned
&\widetilde{\mathcal{C}}+(1-u)(\widetilde{L}_{1}+\widetilde{L}_{2})+(10-4u-v)\mathbf{f}\ \text{if $0\leqslant v\leqslant 4-4u$}, \\
&\frac{10-4u-v}{6}\widetilde{\mathcal{C}}+(1-u)(\widetilde{L}_{1}+\widetilde{L}_{2})+(10-4u-v)\mathbf{f}\ \text{if $4-4u\leqslant v\leqslant 9-3u$}, \\
&\frac{10-4u-v}{6}\big(\widehat{\mathcal{C}}+6(\widehat{L}_{1}+\widehat{L}_{2})+6\mathbf{f}\big)\ \text{if $9-3u\leqslant v\leqslant 10-4u$}, \\
\endaligned
\right.
$$
and
$$
\widehat{N}(u,v)=\left\{\aligned
&0\ \text{if $0\leqslant v\leqslant 4-4u$}, \\
&\frac{v-4+4u}{6}\widehat{\mathcal{C}}\ \text{if $4-4u\leqslant v\leqslant 9-3u$}, \\
&\frac{v-4+4u}{6}\widehat{\mathcal{C}}+(v-9+3u)(\widetilde{L}_{1}+\widetilde{L}_{2})\ \text{if $9-3u\leqslant v\leqslant 10-4u$}. \\
\endaligned
\right.
$$
This gives
$$
\big(\widehat{P}(u,v)\big)^2=\left\{\aligned
&2(1-u)(5-u)-\frac{v^{2}}{6}\ \text{if $0\leqslant v\leqslant 4-4u$}, \\
&\frac{2(1-u)(19-7u-2v)}{3}\ \text{if $4-4u\leqslant v\leqslant 9-3u$}, \\
&\frac{2(10-4u-v)^{2}}{3}\ \text{if $9-3u\leqslant v\leqslant 10-4u$}, \\
\endaligned
\right.
$$
and
$$
\widehat{P}(u,v)\cdot\mathbf{f}=\left\{\aligned
&\frac{v}{6}\ \text{if $0\leqslant v\leqslant 4-4u$}, \\
&\frac{2(1-u)}{3}\ \text{if $4-4u\leqslant v\leqslant 9-3u$}, \\
&\frac{2(10-4 u-v)}{3}\ \text{if $9-3u\leqslant v\leqslant 10-4u$}. \\
\endaligned
\right.
$$
Now, integrating, we get $S(W_{\bullet,\bullet}^{\widehat{T}};\mathbf{f})=\frac{135}{28}<A_{T}(\mathbf{f})=5$.

Let $O$ be a point in $\mathbf{f}$.
If $O\not\in\widehat{L}_1\cup\widehat{L}_2\cup\widehat{\mathcal{C}}$, then
\begin{multline*}
\quad\quad\quad S\big(W_{\bullet, \bullet,\bullet}^{\widehat{T},\mathbf{f}};O\big)=\frac{3}{14}\int_{0}^{1}\int_{0}^{4-4u}\left(\frac{v}{6}\right)^{2}dvdu+\frac{3}{14}\int_{0}^{1}\int_{4-4u}^{9-3u}\left(\frac{2(1-u)}{3}\right)^{2}dvdu+\\
+\frac{3}{14}\int_{0}^{1}\int_{9-3u}^{10-4 u}\left(\frac{2(10-4 u-v)}{3}\right)^{2}dvdu=\frac{13}{63}.\quad\quad\quad
\end{multline*}
Similarly, if $O=\mathbf{f}\cap\widetilde{\mathcal{C}}$, then
\begin{multline*}
\quad\quad\quad S\big(W_{\bullet, \bullet,\bullet}^{\widehat{T},\mathbf{f}};O\big)=\frac{13}{63}+
\frac{3}{7}\int_{0}^{1}\int_{4-4u}^{9-3u}\frac{2(1-u)}{3}\times\frac{v-4+4 u}{6}dvdu+\\
+\frac{3}{7}\int_{0}^{1}\int_{9-3u}^{10-4u}\frac{2(10-4 u-v)}{3}\times\frac{v-4+4u}{6}dvdu=\frac{13}{63}+\frac{193}{504}=\frac{33}{56}.
\end{multline*}
Likewise, if $O=O_3$, we compute $S(W_{\bullet, \bullet,\bullet}^{\widehat{T},\mathbf{f}};O)=\frac{3}{14}$.
So, using \eqref{equation:2-7-C-singular-cusp}, we get $\delta_P(X)>1$.
\end{proof}

Thus, to prove that $\delta_P(X)>1$, we may assume that $S$ is singular, but $P\not\in\mathrm{Sing}(S)$.

\begin{lemma}
\label{lemma:2-7-S-smooth-at-P-not-in-E}
Suppose that $P\not\in E$. Then $\delta_P(X)>1$.
\end{lemma}

\begin{proof}
Recall that $E$ is the $\pi$-exceptional surface.
Using~Corollary~\ref{corollary:2-7-quartic-del-Pezzo-delta-singular}, we may assume that $S$ contains a line $L$ that passes through $P$.
Then $\pi(L)$ is a line in $\mathscr{Q}$, $\pi(L)\cap\mathscr{C}\ne\varnothing$, and one of the following cases holds:
\begin{itemize}
\item[Case 1:] the line $\pi(L)$ intersects the curve $\mathscr{C}$ transversally at $2$ points,
\item[Case 2:] the line $\pi(L)$ is tangent to the curve $\mathscr{C}$ at their single intersection point.
\end{itemize}

Now, let us choose $T$ to be a sufficiently general surface in $|H|$ that passes through $L$.
If~the~intersection $\pi(L)\cap\mathscr{C}$ consists of two points, then $\varpi\colon T\to \mathcal{Q}$ is a blow up of eight distinct points of the~transversal intersection $\mathcal{Q}\cap\mathscr{C}$,
which implies that  $T$ is smooth.
On~the other hand, if $L\cap\mathscr{C}$ consists of one point, then $T$ has one ordinary double point, which is not contained in the curve $L$.
We~have $\mathcal{C}=S\vert_{T}=L+Z$,
where $Z$ is a smooth rational irreducible curve such that $\pi(Z)$ is a smooth twisted cubic in $\mathscr{Q}$. The twisted cubic curve $\pi(Z)$ in $\mathscr{Q}$
intersects the curve $\mathscr{C}$
transversally by six distinct points, which we denote by $Q_3,Q_4,Q_5,Q_6,Q_7,Q_8$.
Moreover, if  $\pi(L)\cap\mathscr{C}$ consists of two distinct points, we denote these points by $Q_1$ and~$Q_2$.
Likewise, if $\pi(L)\cap\mathscr{C}$ consists of one point, we let $Q_1=Q_2=\pi(Z)\cap \mathscr{C}$. Then
\begin{itemize}
\item[Case 1:] the morphism $\varpi\colon T\to\mathcal{Q}$ is the blow up of the points $Q_1,Q_2,\ldots,Q_8$,
\item[Case 2:]  the morphism $\varpi\colon T\to\mathcal{Q}$  is a composition of the blow up of the points $Q_3,\ldots,Q_8$
with a weighted blow up with weights $(1,2)$ of the point $Q_1=Q_2$.
\end{itemize}

Since $T$ is a general surface in $|H|$ that contains the line $L$, we may assume that $P\not\in Z$.
Likewise, we may assume further that $Z$ is contained in the smooth locus of the surface~$T$.
Moreover, we may also assume that the~quadric surface $\mathcal{Q}$ does not contain lines that pass through one point in the~set $\{Q_1,Q_2,\pi(P)\}$
and one point in $\{Q_3,Q_4,Q_5,Q_6,Q_7,Q_8\}$.
Indeed, let $\mathcal{Q}^\prime$, $\mathcal{Q}^{\prime\prime}$, $\mathcal{Q}^{\prime\prime\prime}$ be the~hyperplane sections of the quadric $\mathscr{Q}$
that are singular at the points  $Q_1$, $Q_2$, $\pi(P)$, respectively.
Then $\mathcal{Q}^\prime$, $\mathcal{Q}^{\prime\prime}$, $\mathcal{Q}^{\prime\prime\prime}$ are cones,
$\pi(L)\subset\mathcal{Q}^\prime\cap\mathcal{Q}^{\prime\prime}\cap\mathcal{Q}^{\prime\prime\prime}$,
and every line in $\mathscr{Q}$ containing a point in $\{Q_1,Q_2,\pi(P)\}$ is a ruling of one of these cones.
On the other hand, we have $\mathscr{C}\not\subset\mathcal{Q}^\prime\cup\mathcal{Q}^{\prime\prime}\cup\mathcal{Q}^{\prime\prime\prime}$,
because $\mathscr{C}$ is not contained in a hyperplane.
This implies that the quadric threefold $\mathscr{Q}$ contains at most finitely many lines that pass through a~point in $\{Q_1,Q_2,\pi(P)\}$ and a~point in $\mathscr{C}\setminus\{Q_1,Q_2\}$.
Therefore, we can choose the~surface $T\in|H|$ such that $L\subset T$, but $\mathcal{Q}=\pi(T)$ does not contain any of these lines.

Let us identify $\mathcal{Q}=\mathbb{P}^1\times\mathbb{P}^1$ such that the~line $\pi(L)$ is a curve in $\mathcal{Q}$ of degree $(0,1)$.
Denote by $\mathbf{e}_1,\ldots,\mathbf{e}_8$ the $\varpi$-exceptional curves such~that $\varpi(\mathbf{e}_1)=Q_1,\ldots,\varpi(\mathbf{e}_8)=Q_8$.
Let $\mathbf{g}_3,\ldots,\mathbf{g}_8$ be the~strict transforms on $T$ of the curves in $\mathcal{Q}$ of degree $(1,0)$
that~pass through the points $Q_3,\ldots,Q_8$, respectively.
Then $L,Z,\mathbf{e}_1,\ldots,\mathbf{e}_8,\mathbf{g}_3,\ldots,\mathbf{g}_8$ are smooth irreducible rational curves.
In Case 1, their intersections are~given in the following table:
\begin{center}
\renewcommand\arraystretch{1.3}
\begin{tabular}{|c||c|c|c|c|c|c|c|c|c|c|c|c|c|c|c|c|c|}
 \hline
  & $L$  &  $Z$  &  $\mathbf{e}_1$  & $\mathbf{e}_2$   & $\mathbf{e}_3$   & $\mathbf{e}_4$   & $\mathbf{e}_5$  & $\mathbf{e}_6$  &  $\mathbf{e}_7$  &  $\mathbf{e}_8$  &
  $\mathbf{g}_3$  & $\mathbf{g}_4$  & $\mathbf{g}_5$  & $\mathbf{g}_6$  & $\mathbf{g}_7$  & $\mathbf{g}_8$ \\
  \hline
$L$              & $-2$ & $2$ & $1$ & $1$ & $0$ & $0$ & $0$ & $0$ & $0$ & $0$ & $1$ & $1$ & $1$ & $1$ & $1$ & $1$ \\
  \hline
$Z$              & $2$ & $-2$ & $0$ & $0$ & $1$& $1$& $1$& $1$& $1$& $1$ & $0$& $0$& $0$& $0$& $0$& $0$\\
  \hline
 $\mathbf{e}_1$  & $1$ & $0$& $-1$ & $0$& $0$& $0$& $0$& $0$& $0$& $0$& $0$& $0$& $0$& $0$& $0$& $0$\\
  \hline
 $\mathbf{e}_2$  & $1$ & $0$& $0$ & $-1$& $0$& $0$& $0$& $0$& $0$& $0$& $0$& $0$& $0$& $0$& $0$& $0$\\
  \hline
$\mathbf{e}_3$   & $0$ & $1$& $0$ & $0$& $-1$ & $0$& $0$& $0$& $0$& $0$& $1$& $0$& $0$& $0$& $0$& $0$\\
  \hline
$\mathbf{e}_4$   & $0$& $1$&  $0$& $0$ & $0$& $-1$& $0$& $0$& $0$& $0$& $0$& $1$& $0$& $0$& $0$& $0$\\
  \hline
$\mathbf{e}_5$   & $0$& $1$& $0$& $0$ & $0$& $0$& $-1$& $0$& $0$& $0$& $0$& $0$& $1$& $0$& $0$& $0$\\
  \hline
$\mathbf{e}_6$   & $0$& $1$&$0$ & $0$& $0$& $0$& $0$& $-1$& $0$& $0$& $0$& $0$& $0$& $1$& $0$& $0$\\
  \hline
$\mathbf{e}_7$   & $0$& $1$& $0$& $0$& $0$& $0$& $0$& $0$& $-1$& $0$& $0$& $0$& $0$& $0$& $1$& $0$\\
  \hline
$\mathbf{e}_8$   & $0$& $1$& $0$&$0$ & $0$& $0$& $0$& $0$& $0$& $-1$& $0$& $0$& $0$& $0$& $0$& $1$\\
  \hline
$\mathbf{g}_3$  & $1$& $0$& $0$&$0$ & $1$&$0$ & $0$& $0$& $0$& $0$& $-1$& $0$&$0$ & $0$& $0$ &$0$ \\
  \hline
$\mathbf{g}_4$  & $1$& $0$&$0$ &$0$ & $0$& $1$& $0$& $0$&$0$ & $0$& $0$& $-1$& $0$& $0$& $0$& $0$\\
  \hline
$\mathbf{g}_5$  & $1$& $0$& $0$& $0$& $0$& $0$& $1$&$0$ & $0$& $0$& $0$& $0$& $-1$& $0$& $0$&$0$ \\
  \hline
$\mathbf{g}_6$  & $1$& $0$& $0$&$0$ & $0$& $0$& $0$& $1$& $0$&$0$ & $0$& $0$& $0$& $-1$& $0$& $0$\\
  \hline
$\mathbf{g}_7$  & $1$& $0$&$0$ & $0$& $0$& $0$& $0$& $0$& $1$& $0$& $0$& $0$& $0$& $0$& $-1$& $0$\\
  \hline
$\mathbf{g}_8$  & $1$& $0$& $0$& $0$& $0$& $0$& $0$&$0$ & $0$&$1$ & $0$& $0$& $0$& $0$& $0$& $-1$\\
 \hline
\end{tabular}
\end{center}
In Case 2, we have $\mathbf{e}_1=\mathbf{e}_2$, and  $\mathbf{e}_1$ contains the singular point of $T$, so that $\mathbf{e}_1^2=-\frac{1}{2}$.
The remaining intersection numbers are exactly the same as in Case 1.

Observe that $P\not\in Z\cup\mathbf{g}_3\cup\mathbf{g}_4\cup\mathbf{g}_5\cup\mathbf{g}_6\cup\mathbf{g}_7\cup\mathbf{g}_8\cup\mathbf{e}_1\cup\mathbf{e}_2$, since $P\not\in E$ by assumption.

Recall that $P(u)=-K_X-uT$  is nef $\iff$ $P(u)$ is pseudoeffective $\iff$ $u\in[0,1]$.
Let $v$ be a non-negative real number. Then, in both Cases 1 and 2, we have
\begin{equation}
\label{equation:2-7-S-smooth-at-P-not-in-E-equivalence}
P(u)\big\vert_{T}-vL\sim_{\mathbb{R}}\frac{9-5u-4v}{4}L+\frac{3+u}{4}Z+\frac{5-5u}{4}\big(\mathbf{e}_1+\mathbf{e}_2\big)+\frac{1-u}{4}\sum_{i=3}^8\mathbf{g}_i,
\end{equation}
which implies that the divisor $P(u)\vert_{T}-vL$ is pseudoeffective $\iff$ $v\leqslant \frac{9-5u}{4}$.

Let $P(u,v)$ be the positive part of Zariski decomposition of the divisor $P(u)\vert_{T}-vL$,
and let $N(u,v)$ be its negative part. Then it follows from \cite[Theorem 1.112]{Book} that
\begin{equation}
\label{equation:S-smooth-at-P-not-in-E-inequality}
\delta_P(X)\geqslant\min\left\{\frac{1}{S_X(T)},\frac{1}{S(W_{\bullet,\bullet}^{T};L)},\frac{1}{S(W_{\bullet, \bullet,\bullet}^{T,L};P)}\right\},
\end{equation}
where
$$
S\big(W_{\bullet,\bullet}^{T};L\big)=\frac{3}{14}\int_0^1\int_0^{\frac{9-5u}{4}}\big(P(u,v)\big)^2dvdu
$$
and
$$
S\big(W_{\bullet, \bullet,\bullet}^{T,L};P\big)=\frac{3}{14}\int_0^1\int_0^{\frac{9-5u}{4}}\Big(\big(P(u,v)\cdot L\big)\Big)^2dvdu+\frac{3}{7}\int_0^1\int_0^{\frac{9-5u}{4}}\big(P(u,v)\cdot L\big)\mathrm{ord}_P\Big(N(u,v)\big|_{L}\Big)dvdu.
$$

Let us compute $S(W_{\bullet,\bullet}^{T};L)$ and $S(W_{\bullet, \bullet,\bullet}^{T,L};P)$.
If $0\leqslant u\leqslant \frac{1}{3}$, then, using~\eqref{equation:2-7-S-smooth-at-P-not-in-E-equivalence}, we get
$$
P(u,v)=\left\{\aligned
&\frac{9-5u-4v}{4}L+\frac{3+u}{4}Z+\frac{5-5u}{4}\big(\mathbf{e}_1+\mathbf{e}_2\big)+\frac{1-u}{4}\sum_{i=3}^8\mathbf{g}_i\ \text{if $0\leqslant v\leqslant 1$}, \\
&\frac{9-5u-4v}{4}\big(L+\mathbf{e}_1+\mathbf{e}_2\big)+\frac{3+u}{4}Z+\frac{1-u}{4}\sum_{i=3}^8\mathbf{g}_i\ \text{if $1\leqslant v\leqslant \frac{3-3u}{2}$}, \\
&\frac{9-5u-4v}{4}\big(L+Z+\mathbf{e}_1+\mathbf{e}_2\big)+\frac{1-u}{4}\sum_{i=3}^8\mathbf{g}_i\ \text{if $\frac{3-3u}{2}\leqslant v\leqslant 2-u$}, \\
&\frac{9-5 u-4 v}{4}\Big(L+Z+\mathbf{e}_1+\mathbf{e}_2+\sum_{i=3}^8\mathbf{g}_i\Big)\ \text{if $2-u\leqslant v\leqslant \frac{9-5u}{4}$}, \\
\endaligned
\right.
$$
$$
N(u,v)=\left\{\aligned
&0\ \text{if $0\leqslant v\leqslant 1$}, \\
&(v-1)\big(\mathbf{e}_1+\mathbf{e}_2\big) \ \text{if $1\leqslant v\leqslant \frac{3-3u}{2}$}, \\
&(v-1)\big(\mathbf{e}_1+\mathbf{e}_2\big)+\frac{2v+3u-3}{2}Z\ \text{if $\frac{3-3u}{2}\leqslant v\leqslant 2-u$}, \\
&(v-1)\big(\mathbf{e}_1+\mathbf{e}_2\big)+\frac{2v+3u-3}{2}Z+(v-2+u)\sum_{i=3}^{8}\mathbf{g}_{i}\ \text{if $2-u\leqslant v\leqslant \frac{9-5u}{4}$}, \\
\endaligned
\right.
$$
$$
\big(P(u,v)\big)^2=\left\{\aligned
&2u^2+(2v-12)u-2v^2-2v+10 \ \text{if $0\leqslant v\leqslant 1$}, \\
&2u^2+(2v-12)u-6v+12\ \text{if $1\leqslant v\leqslant \frac{3-3u}{2}$}, \\
&\frac{13u^2+16uv+4v^2-42u-24v+33}{2}\ \text{if $\frac{3-3u}{2}\leqslant v\leqslant 2-u$}, \\
&\frac{(9-5u-4v)^{2}}{2}\ \text{if $2-u\leqslant v\leqslant \frac{9-5u}{4}$}, \\
\endaligned
\right.
$$
and
$$
P(u,v)\cdot L=\left\{\aligned
&1-u+2v\ \text{if $0\leqslant v\leqslant 1$}, \\
&3-u\ \text{if $1\leqslant v\leqslant \frac{3-3u}{2}$}, \\
&6-4u-2v\ \text{if $\frac{3-3u}{2}\leqslant v\leqslant 2-u$}, \\
&2(9-5u-4v)\ \text{if $2-u\leqslant v\leqslant \frac{9-5u}{4}$}, \\
\endaligned
\right.
$$
Similarly, if $\frac{1}{3}\leqslant u\leqslant 1$, then, using~\eqref{equation:2-7-S-smooth-at-P-not-in-E-equivalence}, we get
$$
P(u,v)=\left\{\aligned
&\frac{9-5u-4v}{4}L+\frac{3+u}{4}Z+\frac{5-5u}{4}\big(\mathbf{e}_1+\mathbf{e}_2\big)+\frac{1-u}{4}\sum_{i=3}^8\mathbf{g}_i\ \text{if $0\leqslant v\leqslant \frac{3-3u}{2}$}, \\
&\frac{9-5u-4v}{4}(L+Z)+\frac{5-5u}{4}\big(\mathbf{e}_1+\mathbf{e}_2\big)+\frac{1-u}{4}\sum_{i=3}^8\mathbf{g}_i\ \text{if $\frac{3-3u}{2}\leqslant v\leqslant 1$}, \\
&\frac{9-5u-4v}{4}\big(L+Z+\mathbf{e}_1+\mathbf{e}_2\big)+\frac{1-u}{4}\sum_{i=3}^8\mathbf{g}_i\ \text{if $1\leqslant v\leqslant 2-u$}, \\
&\frac{9-5 u-4 v}{4}\Big(L+Z+\mathbf{e}_1+\mathbf{e}_2+\sum_{i=3}^8\mathbf{g}_i\Big)\ \text{if $2-u\leqslant v\leqslant \frac{9-5u}{4}$}, \\
\endaligned
\right.
$$
$$
N(u,v)=\left\{\aligned
&0\ \text{if $0\leqslant v\leqslant \frac{3-3u}{2}$}, \\
&\frac{2v+3u-3}{2}Z\ \text{if $\frac{3-3u}{2}\leqslant v\leqslant 1$}, \\
&(v-1)\big(\mathbf{e}_1+\mathbf{e}_2\big)+\frac{2v+3u-3}{2}Z\ \text{if $1\leqslant v\leqslant 2-u$}, \\
&(v-1)\big(\mathbf{e}_1+\mathbf{e}_2\big)+\frac{2v+3u-3}{2}Z+(v-2+u)\sum_{i=3}^{8}\mathbf{g}_{i}\ \text{if $2-u\leqslant v\leqslant \frac{9-5u}{4}$}, \\
\endaligned
\right.
$$
$$
\big(P(u,v)\big)^2=\left\{\aligned
&2u^2+(2v-12)u-2v^2-2v+10\ \text{if $0\leqslant v\leqslant \frac{3-3u}{2}$}, \\
&\frac{(1-u)(29-13u-16v)}{2}\ \text{if $\frac{3-3u}{2}\leqslant v\leqslant 1$}, \\
&\frac{13u^2+16uv+4v^2-42u-24v+33}{2}\ \text{if $1\leqslant v\leqslant 2-u$}, \\
&\frac{(9-5u-4v)^{2}}{2}\ \text{if $2-u\leqslant v\leqslant \frac{9-5u}{4}$}, \\
\endaligned
\right.
$$
and
$$
P(u,v)\cdot L=\left\{\aligned
&1-u+2v\ \text{if $0\leqslant v\leqslant \frac{3-3u}{2}$}, \\
&4-4u\ \text{if $\frac{3-3u}{2}\leqslant v\leqslant 1$}, \\
&6-4u-2v\ \text{if $1\leqslant v\leqslant 2-u$}, \\
&2(9-5u-4v)\ \text{if $2-u\leqslant v\leqslant \frac{9-5u}{4}$}, \\
\endaligned
\right.
$$
Therefore, we have $P\not\in\mathrm{Supp}(N(u,v))$, because $P\not\in Z\cup\mathbf{g}_3\cup\mathbf{g}_4\cup\mathbf{g}_5\cup\mathbf{g}_6\cup\mathbf{g}_7\cup\mathbf{g}_8\cup\mathbf{e}_1\cup\mathbf{e}_2$.
So, integrating $(P(u,v))^2$ and $(P(u,v)\cdot L)^2$, we get $S(W_{\bullet,\bullet}^{T};L)=\frac{423}{448}$
and $S(W_{\bullet, \bullet,\bullet}^{T,L};P)=\frac{79}{84}$, which implies that $\delta_P(X)>1$ by \eqref{equation:S-smooth-at-P-not-in-E-inequality}.
\end{proof}

By Lemma~\ref{lemma:2-7-S-smooth-at-P-not-in-E}, to show that $\delta_P(X)>1$, we may assume that $P\in E$. Then $\pi(P)\in\mathscr{C}$.
Now, let us choose $T$ to be a sufficiently general surface in $|H|$ that contains the point~$P$,
so~that $\mathcal{Q}$ is a sufficiently general hyperplane section of the quadric $\mathscr{Q}$ that contains $\pi(P)$.
Then $T$ is smooth, and $\varpi\colon T\to \mathcal{Q}$ is a blow up of eight  points of the~intersection $\mathcal{Q}\cap\mathscr{C}$.

Let $Q_1=\pi(P)$, let $Q_2,\ldots,Q_8$ be the remaining seven points of the~intersection~$\mathcal{Q}\cap\mathscr{C}$,
and let $\mathbf{e}_1,\ldots,\mathbf{e}_8$ be the~$\varpi$-exceptional curves such that $\varpi(\mathbf{e}_1)=Q_1,\ldots,\varpi(\mathbf{e}_8)=Q_8$.
For every $u\in [0,1]$, set
$$
t(u)=\inf\Big\{v\in \mathbb R_{\geqslant 0} \ \big|\ \text{the divisor $P(u)\big|_T-v\mathbf{e}_1$ is pseudo-effective}\Big\},
$$
and fix a~real number $v\in [0, t(u)]$.
Let $P(u,v)$ and $N(u,v)$ be the~positive and the~negative parts~of the~Zariski decomposition~of~the~$\mathbb{R}$-divisor \mbox{$P(u)|_T-v\mathbf{e}_1$},
respectively.
Then
\begin{equation}
\label{equation:2-7-P-in-E-inequality}
\delta_P(X)\geqslant\min\left\{\frac{1}{S_X(T)},\frac{1}{S(W_{\bullet,\bullet}^{T};\mathbf{e}_1)},\frac{1}{S(W_{\bullet, \bullet,\bullet}^{T,\mathbf{e}_1};P)}\right\}
\end{equation}
by \cite[Theorem 1.112]{Book},  where
$$
S\big(W_{\bullet,\bullet}^{T};\mathbf{e}_1\big)=\frac{3}{14}\int_0^1\int_0^{t(u)}\big(P(u,v)\big)^2dvdu
$$
and
\begin{multline*}
\quad \quad \quad \quad S\big(W_{\bullet, \bullet,\bullet}^{T,\mathbf{e}_1};P\big)=\frac{3}{14}\int_0^1\int_0^{t(u)}\Big(\big(P(u,v)\cdot\mathbf{e}_1\big)\Big)^2dvdu+\\
+\frac{3}{7}\int_0^1\int_0^{t(u)}\big(P(u,v)\cdot \mathbf{e}_1\big)\mathrm{ord}_P\Big(N(u,v)\big|_{\mathbf{e}_1}\Big)dvdu.\quad \quad \quad
\end{multline*}
Let us compute $S(W_{\bullet,\bullet}^{T};\mathbf{e}_1)$ and $S(W_{\bullet, \bullet,\bullet}^{T,\mathbf{e}_1};P)$.

Recall that $\varphi\colon T\to\mathbb{P}^1$ is an elliptic fibration, which is given by the linear system~$|-K_T|$.
As in the proof of Lemma~\ref{lemma:2-7-S-smooth-at-P-not-in-E}, let us identify $\mathcal{Q}=\mathbb{P}^1\times\mathbb{P}^1$.

\begin{lemma}
\label{lemma:2-7-elliptic-fibration}
Let $F$ be a curve in $|-K_T|$. Then $F$ is irreducible and reduced.
\end{lemma}

\begin{proof}
Suppose that $F$ is reducible or non-reduced. Then the~curve $\pi(F)$ is also reducible or non-reduced,
and every irreducible component of the curve $F$ is a smooth $(-2)$-curve.
But $\pi(F)$ is a curve in $\mathcal{Q}$ of degree $(2,2)$ that passes through the~points $Q_1,Q_2,\ldots,Q_8$.
Therefore, we have one of the following possibilities:
\begin{enumerate}
\item $\mathcal{Q}$ contains a line that passes through $Q_1$ and one point among  $Q_2,\ldots,Q_8$,
\item $\mathcal{Q}$ contains a line that passes through two points among  $Q_2,\ldots,Q_8$,
\item $\mathcal{Q}$ contains a conic that passes through $Q_1$ and three points among  $Q_2,\ldots,Q_8$.
\end{enumerate}

Recall that $\mathcal{Q}$ is a general hyperplane section of the quadric $\mathscr{Q}$ that contains $Q_1=\pi(P)$.
As we already mentioned in the proof of Lemma~\ref{lemma:2-7-S-smooth-at-P-not-in-E},
the quadric $\mathscr{Q}$ contains finitely many lines that pass through $Q_1$ and a~point in $\mathscr{C}\setminus Q_1$.
Thus, since $\mathcal{Q}$ is assumed to be general, the~quadric $\mathcal{Q}$ does not contain any of these lines,
so that $\mathcal{Q}$ does not contain a line that passes through $Q_1$ and a point among  $Q_2,\ldots,Q_8$.

Similarly, a parameter count implies that $\mathcal{Q}$ does not contain secant lines of the curve~$\mathscr{C}$,
so that $\mathcal{Q}$ does not contain a line that passes through two points among  $Q_2,\ldots,Q_8$,

Finally, we suppose that $\mathcal{Q}$ contains an irreducible conic $C$ that passes through $Q_1$ and three points among  $Q_2,\ldots,Q_8$.
Let $\rho\colon\mathscr{Q}\dasharrow\mathbb{P}^3$ be a linear projection from the point $Q_1$.
Then $\rho(\mathscr{C})$ is a curve of degree $7$, and the induced map $\mathscr{C}\dasharrow\rho(\mathscr{C})$ is an isomorphism,
because $\mathscr{C}$ is an intersection of quadrics.
Similarly, all points $\rho(Q_2),\ldots,\rho(Q_8)$ are distinct.
Then $\rho(C)$ is a three-secant line of the curve $\rho(\mathscr{C})$.
Note that $\rho(\mathscr{C})$ contains one-parameter family of three-secants  \cite[Appendix~A]{CheltsovShramov}.
But $\rho(\mathcal{Q})$ is a general plane in $\mathbb{P}^3$,
which implies that $\rho(\mathcal{Q})$ does not contain three-secant lines of the curve $\rho(\mathscr{C})$ --- a contradiction.
\end{proof}

\begin{corollary}
\label{corollary:2-7-elliptic-fibration}
Let $\gamma$ be a class in the~group $\mathrm{Pic}(T)$ such that $-K_T\cdot \gamma=1$ and $\gamma^2=-1$.
Then the linear system $|\gamma|$ consists of a unique $(-1)$-curve.
\end{corollary}

\begin{proof}
Apply the Riemann--Roch theorem, Serre duality and Lemma~\ref{lemma:2-7-elliptic-fibration}.
\end{proof}

Let us use Corollary~\ref{corollary:2-7-elliptic-fibration}, to describe infinitely many $(-1)$-curves in the~surface $T$.
Namely, let $\ell_1$ and $\ell_2$ be any curves in  $\mathcal{Q}=\mathbb{P}^1\times\mathbb{P}^1$ of degrees $(1,0)$ and $(0,1)$, respectively.
For $n\in\mathbb{Z}_{\geqslant 0}$ and  $i\in\{2,3,4,5,6,7,8\}$, let
$B_{n,1,1}$, $B_{n,1,2}$, $B_{n,2,i}$, $B_{n,3}$, $B_{n,4,i}$ be the~classes
$$
\varpi^*\big(a_1\ell_1+a_2\ell_2\big)-\sum_{i=1}^8b_i\mathbf{e}_i\in\mathrm{Pic}(T),
$$
where $a_1,a_2,b_1,\ldots,b_8$ are non-negative integers given in the following table:
\begin{center}
\renewcommand\arraystretch{1.4}
\begin{tabular}{|c||c|c|c|c|}
\hline & $a_1$ & $a_2$ & $b_1$ & $b_2,b_3,b_4,b_5,b_6,b_7,b_8$ \\
\hline
\hline
$B_{n,1,1}$ & $14n^{2}+7n+1$ & $14n^{2}+7n$ & $7n^{2}+7n+1$ & $\forall j$ $b_j=7n^{2}+3n$ \\
\hline
$B_{n,1,2}$ & $14n^{2}+7n$ & $14n^{2}+7n+1$ & $7n^{2}+7n+1$ & $\forall j$ $b_j=7n^{2}+3n$\\
\hline
\shortstack{{ }\\ $B_{n,2,i}$\\{ }} & \shortstack{{ }\\ $14n^{2}+13n+3$\\{ }} & \shortstack{{ }\\ $14n^{2}+13n+3$\\{ }} & \shortstack{{ }\\ $7n^{2}+10n+3$\\{ }} & \shortstack{{ }\\$b_i=7 n^{2}+6n+2$\\{$\forall j\ne i$ $b_j=7n^{2}+6n+1$}} \\
\hline
$B_{n,3}$ & $14n^{2}+21n+7$ & $14n^{2}+21 n+7$ & $7n^{2}+14 n+6$ & $\forall j$ $b_j=7 n^{2}+10n+3$\\
\hline
\shortstack{{ }\\ $B_{n,4,i}$\\{ }} & \shortstack{{ }\\ $14 n^{2}+29 n+15$\\{ }} & \shortstack{{ }\\ $14n^{2}+29 n+15$\\{ }} & \shortstack{{ }\\ $7n^{2}+18n+11$\\{ }} & \shortstack{{ }\\$b_i=7n^{2}+14n+6$\\{$\forall j\ne i$ $b_j=7n^{2}+14n+7$}} \\
\hline
\end{tabular}
\end{center}
By Corollary~\ref{corollary:2-7-elliptic-fibration}, each linear system
$|B_{n,1,1}|$, $|B_{n,1,2}|$, $|B_{n,2,i}|$, $|B_{n,3}|$, $|B_{n,4,i}|$ contains a~unique $(-1)$-curve.
Hence, we can identify the classes $B_{n,1,1}$, $B_{n,1,2}$, $B_{n,2,i}$, $B_{n,3}$, $B_{n,4,i}$ with $(-1)$-curves
in $|B_{n,1,1}|$, $|B_{n,1,2}|$, $|B_{n,2,i}|$, $|B_{n,3}|$, $|B_{n,4,i}|$, respectively.
Set
\begin{align*}
B_{n, 1}&=B_{n, 1, 1}+B_{n, 1, 2},\\
B_{n, 2}&=B_{n,2,2}+B_{n,2,3}+B_{n,2,4}+B_{n,2,5}+B_{n,2,6}+B_{n,2,7}+B_{n,2,8},\\
B_{n, 4}&=B_{n,4,2}+B_{n,4,3}+B_{n,4,4}+B_{n,4,5}+B_{n,4,6}+B_{n,4,7}+B_{n,4,8}.
\end{align*}
Note that irreducible components of each curve $B_{n,1}$, $B_{n,2}$, $B_{n,4}$ are disjoint $(-1)$-curves,
and $B_{n,1}\cap B_{n,2}=\varnothing$, $B_{n,2}\cap B_{n,3}=\varnothing$, $B_{n,3}\cap B_{n,4}=\varnothing$, $B_{n,4} \cap B_{n+1,1}=\varnothing$ for each $n\geqslant 0$.

Now, we let $I^\prime_{0,1}=[0,\frac{1}{3}]$ and $I^{\prime\prime}_{0,1}=[\frac{1}{3},\frac{3}{8}]$.
For every $n\in\mathbb{Z}_{> 0}$, we also let
\begin{align*}
I_{n, 1}^{\prime}&=\left[\frac{-1+4 n+14 n^{2}}{6 n+14 n^{2}}, \frac{1+13 n+21 n^{2}}{3+16 n+21 n^{2}}\right],\\
I_{n, 1}^{\prime \prime}&=\left[\frac{1+13n+21n^{2}}{3+16n+21n^{2}}, \frac{3+35n+49 n^{2}}{8+42 n+49 n^{2}}\right].
\end{align*}
For every $n\in\mathbb{Z}_{\geqslant 0}$, we let
\begin{align*}
I_{n,2}^{\prime}&=\left[\frac{3+35 n+49n^{2}}{8+42 n+49 n^{2}}, \frac{3+22 n+28 n^{2}}{6+26 n+28 n^{2}}\right],\\
I_{n,2}^{\prime\prime}&=\left[\frac{3+22n+28 n^{2}}{6+26 n+28 n^{2}}, \frac{2+7 n}{3+7 n}\right],\\
I_{n,3}^{\prime}&=\left[\frac{2+7 n}{3+7 n}, \frac{21+50 n+28 n^{2}}{26+54 n+28 n^{2}}\right],\\
I_{n, 3}^{\prime\prime}&=\left[\frac{21+50 n+28 n^{2}}{26+54 n+28 n^{2}}, \frac{39+91 n+49 n^{2}}{48+98 n+49 n^{2}}\right],\\
I_{n,4}^{\prime}&=\left[\frac{39+91 n+49 n^{2}}{48+98 n+49 n^{2}}, \frac{19+41 n+21 n^{2}}{23+44 n+21 n^{2}}\right],\\
I_{n,4}^{\prime\prime}&=\left[\frac{19+41 n+21 n^{2}}{23+44 n+21 n^{2}}, \frac{17+32 n+14 n^{2}}{20+34 n+14 n^{2}}\right].
\end{align*}
Set $I_{n,1}=I_{n,1}^\prime\cup I_{n,1}^{\prime\prime}$, $I_{n,2}=I_{n,2}^\prime\cup I_{n,2}^{\prime\prime}$, $I_{n,3}=I_{n,3}^\prime\cup I_{n,3}^{\prime\prime}$, $I_{n,4}=I_{n,4}^\prime\cup I_{n,4}^{\prime\prime}$.
Then
$$
[0,1)=\bigcup_{n \in \mathbb{Z}_{\geqslant 0}} \Big(I_{n, 1} \cup I_{n, 2} \cup I_{n, 3} \cup I_{n, 4}\Big),
$$
the intervals $I_{n,1}$, $I_{n,2}$, $I_{n,3}$, $I_{n,4}$ have positive volumes, and all their interiors are~disjoint.
Let us analyze $P(u,v)$ and $N(u,v)$ when $u$ is contained in one of these intervals.

First, we deal with $u\in I_{n,1}$.
If $u\in I_{n,1}^{\prime}$ and $v\in\left[0,\frac{2+14 n+28 n^{2}-u(1+14 n+28 n^{2})}{1+7 n+7 n^{2}}\right]$,
then
\begin{multline*}
P(u, v)=\frac{19+70 n+84 n^{2}-u\left(16+70 n+84 n^{2}\right)-v\left(8+28 n+21 n^{2}\right)}{8+28 n+21 n^{2}}\mathbf{e}_{1}+\\
+\frac{3+35n+49 n^{2}-u\left(8+42 n+49 n^{2}\right)}{8+28 n+21 n^{2}}B_{n, 1}+
\frac{1-4n-14n^{2}+u\left(6 n+14 n^{2}\right)}{8+28 n+21 n^{2}}B_{n, 2}
\end{multline*}
and $N(u,v)=0$. The same holds if $u\in I_{n,1}^{\prime\prime}$ and \mbox{$v\in\left[0,\frac{7+26 n+28n^{2}-u(6+26 n+28 n^{2})}{3+10 n+7 n^{2}}\right]$}.
Then
\begin{align*}
\big(P(u,v)\big)^{2}&=10-12u+2u^{2}-2v-v^{2},\\
P(u,v)\cdot\mathbf{e}_1&=1+v.
\end{align*}
Similarly, if $u\in I_{n,1}^\prime$ and $v\in\left[\frac{2+14 n+28 n^{2}-u\left(1+14 n+28 n^{2}\right)}{1+7 n+7 n^{2}}, \frac{7+26 n+28 n^{2}-u\left(6+26 n+28 n^{2}\right)}{3+10 n+7 n^{2}}\right]$, then
\begin{multline*}
P(u, v)=\frac{19+70n+84n^{2}-u(16+70 n+84 n^{2})-v(8+28 n+21 n^{2})}{8+28 n+21 n^{2}}\mathbf{e}_{1}+\\
+\frac{\big(19+70n+84n^{2}-u(16+70 n+84 n^{2})-v(8+28 n+21 n^{2})\big)(1+7n+7n^{2})}{8+28n+21n^{2}}B_{n,1}+\\
+\frac{1-4n-14n^{2}+u(6 n+14 n^{2})}{8+28 n+21 n^{2}}B_{n,2}\quad\quad\quad\quad\quad\quad
\end{multline*}
and
$$
N(u,v)=\left(u\left(1+14 n+28 n^{2}\right)+v\left(1+7n+7n^{2}\right)-2-14n-28n^{2}\right)B_{n,1}.
$$
Then
\begin{multline*}
\quad\quad \big(P(u,v)\big)^{2}=10-12 u+2 u^{2}-2 v-v^{2}+\\
+2\left(u\left(1+14 n+28 n^{2}\right)+v\left(1+7 n+7 n^{2}\right)-2-14n-28n^{2}\right)^{2}\quad\quad\quad\quad
\end{multline*}
and
\begin{multline*}
\quad\quad P(u,v)\cdot\mathbf{e}_1=5+56n+280n^2+588n^3+392n^4-\\
\quad\quad-2u(1+7n+7n^2)(1+14n+28n^2)-v(1+28n+126n^2+196n^3+98n^4).
\end{multline*}
Likewise, if $u\in I_{n,1}^{\prime\prime}$ and $v \in\left[\frac{7+26 n+28 n^{2}-u\left(6+26 n+28 n^{2}\right)}{3+10 n+7 n^{2}}, \frac{2+14 n+28 n^{2}-u\left(1+14 n+28 n^{2}\right)}{1+7 n+7 n^{2}}\right]$, then
\begin{multline*}
P(u, v)=\frac{(19+70 n+84 n^{2}-u(16+70 n+84 n^{2})-v(8+28 n+21 n^{2})}{8+28 n+21 n^{2}}\mathbf{e}_{1}+\\
+\frac{\big(19+70 n+84 n^{2}-u(16+70 n+84 n^{2})-v(8+28 n+21 n^{2})\big)(1+n)(3+7n)}{8+28 n+21 n^{2}}B_{n,2}+\\
+\frac{3+35 n+49 n^{2}-u\left(8+42 n+49 n^{2}\right)}{8+28 n+21 n^{2}}B_{n,1}\quad\quad\quad\quad
\end{multline*}
and
$$
N(u,v)=\left(u\left(6+26 n+28 n^{2}\right)+v\left(3+10 n+7 n^{2}\right)-7-26n-28n^{2}\right)B_{n,2}.
$$
Then
\begin{multline*}
\big(P(u,v)\big)^{2}=10-12 u+2 u^{2}-2 v-v^{2}+\\
+7\left(u\left(6+26 n+28 n^{2}\right)+v\left(3+10 n+7 n^{2}\right)-7-26n-28n^{2}\right)^{2}\quad\quad
\end{multline*}
and
\begin{multline*}
P(u,v)\cdot\mathbf{e}_1=148+1036 n+2751 n^{2}+3234 n^{3}+1372 n^{4}-\\
-14u(1+n)(1+2 n)(3+7 n)^{2}-v\left(62+420 n+994 n^{2}+980 n^{3}+343 n^{4}\right).\quad\quad
\end{multline*}
If $u \in I_{n, 1}^{\prime}$ and $v\in\left[\frac{7+26 n+28 n^{2}-u\left(6+26 n+28 n^{2}\right)}{3+10 n+7 n^{2}}, \frac{19+70 n+84 n^{2}-u\left(16+70 n+84 n^{2}\right)}{8+28 n+21 n^{2}}\right]$, then
\begin{multline*}
P(u,v)=\frac{19+70 n+84 n^{2}-u\left(16+70 n+84 n^{2}\right)-v\left(8+28 n+21 n^{2}\right)}{8+28 n+21 n^{2}}\mathbf{e}_1+\\
+\frac{\big(19+70 n+84 n^{2}-u\left(16+70 n+84 n^{2}\right)-v\left(8+28 n+21 n^{2}\right)\big)\left(1+7n+7 n^{2}\right)}{8+28 n+21 n^{2}}B_{n,1}+\\
+\frac{\big(19+70 n+84 n^{2}-u\left(16+70 n+84 n^{2}\right)-v\left(8+28 n+21 n^{2}\right)\big)(1+n)(3+7n)}{8+28 n+21 n^{2}}B_{n,2}
\end{multline*}
and
\begin{multline*}
N(u,v)=\left(u\left(1+14 n+28 n^{2}\right)+v\left(1+7 n+7 n^{2}\right)-2-14n-28n^{2}\right)B_{n,1}+ \\
+\left(u\left(6+26 n+28 n^{2}\right)+v\left(3+10 n+7 n^{2}\right)-7-26n-28n^{2}\right)B_{n,2}.
\end{multline*}
The same holds if $u\in I_{n,1}^{\prime\prime}$ and $v\in\left[\frac{2+14 n+28 n^{2}-u\left(1+14 n+28 n^{2}\right)}{1+7 n+7 n^{2}}, \frac{19+70 n+84 n^{2}-u\left(16+70 n+84 n^{2}\right)}{8+28 n+21 n^{2}}\right]$.
In both cases, we have
\begin{align*}
\big(P(u,v)\big)^{2}&=\left(19+70 n+84 n^{2}-u\left(16+70 n+84 n^{2}\right)-v\left(8+28 n+21 n^{2}\right)\right)^{2},\\
P(u,v)\cdot\mathbf{e}_1&=\left(8+28 n+21 n^{2}\right)\left(19+70 n+84 n^{2}-u\left(16+70 n+84 n^{2}\right)-v\left(8+28 n+21 n^{2}\right)\right).
\end{align*}
Hence, if $u\in I_{n,1}$, then
$$
t(u)=\frac{19+70 n+84 n^{2}-u\left(16+70 n+84 n^{2}\right)}{8+28 n+21 n^{2}}.
$$

Now, we deal with $u\in I_{n,2}$. If $u\in I_{n,2}^{\prime}$ and $v\in\left[0,\frac{7+26 n+28 n^{2}-u\left(6+26 n+28 n^{2}\right)}{3+10 n+7 n^{2}}\right]$, then
\begin{multline*}
P(u,v)=\frac{17+56 n+56 n^{2}-u\left(15+56 n+56 n^{2}\right)-7v(1+n)(1+2 n)}{7(1+n)(1+2n)}\mathbf{e}_{1}+ \\
+\frac{(1+n)(2+7 n)-u(1+n)(3+7n)}{7(1+n)(1+2n)} B_{n, 2}+\frac{u\left(8+42 n+49 n^{2}\right)-3-35n-49n^{2}}{7(1+n)(1+2n)}B_{n, 3}
\end{multline*}
and $N(u,v)=0$.
The same holds if $u\in I_{n,2}^{\prime\prime}$ and $v\in\left[0,\frac{15+42n+28n^{2}-u\left(14+42 n+28 n^{2}\right)}{6+14 n+7 n^{2}}\right]$.
Then
\begin{align*}
\big(P(u,v)\big)^{2}&=10-12u+2u^{2}-2v-v^{2},\\
P(u,v)\cdot\mathbf{e}_1&=v+1.
\end{align*}
If $u\in I_{n,2}^\prime$ and $v \in\left[\frac{7+26 n+28 n^{2}-u\left(6+26 n+28 n^{2}\right)}{3+10 n+7 n^{2}}, \frac{15+42 n+28 n^{2}-u\left(14+42 n+28 n^{2}\right)}{6+14 n+7 n^{2}}\right]$, then
\begin{multline*}
\quad\quad P(u,v)=\frac{17+56n+56 n^{2}-u(15+56 n+56 n^{2})-7v(1+n)(1+2n)}{7(1+n)(1+2n)}\mathbf{e}_{1}+\\
+\frac{(1+n)(3+7n)\big(17+56n+56 n^{2}-u(15+56 n+56 n^{2})-7v(1+n)(1+2n)\big)}{7(1+n)(1+2n)}B_{n, 2}+\\
+\frac{u\left(8+42 n+49 n^{2}\right)-3-35n-49n^{2}}{7(1+n)(1+2 n)}B_{n,3},\quad\quad\quad\quad\quad\quad
\end{multline*}
$$
N(u,v)=\big(u\left(6+26 n+28 n^{2}\right)+v\left(3+10 n+7n^{2}\right)-7-26n-28n^{2}\big)B_{n, 2},
$$
\begin{multline*}
\big(P(u,v)\big)^{2}=10-12 u+2 u^{2}-2 v-v^{2}+\\
+7\left(u\left(6+26 n+28 n^{2}\right)+v\left(3+10 n+7 n^{2}\right)-7-26n-28n^{2}\right)^{2},
\end{multline*}
\begin{multline*}
P(u,v)\cdot\mathbf{e}_1=148+1036 n+2751 n^{2}+3234 n^{3}+1372 n^{4}-\\
\quad\quad\quad-14u(1+n)(1+2 n)(3+7 n)^{2}-\\
-v\left(62+420 n+994 n^{2}+980 n^{3}+343 n^{4}\right).
\end{multline*}
Similarly, if $u\in I_{n,2}^{\prime\prime}$ and $v\in\left[\frac{15+42 n+28 n^{2}-u\left(14+42 n+28 n^{2}\right)}{6+14 n+7 n^{2}}, \frac{7+26 n+28 n^{3}-u\left(6+26 n+28 n^{2}\right)}{3+10 n+7 n^{2}}\right]$, then
\begin{multline*}
P(u,v)=\frac{17+56+56 n^{2}-u\left(15+56 n+56 n^{2}\right)-7v(1+n)(1+2 n)}{7(1+n)(1+2n)}\mathbf{e}_{1}+\\
+\frac{(6+14 n+7 n^{2})\big(17+56+56 n^{2}-u\left(15+56 n+56 n^{2}\right)-7v(1+n)(1+2 n)\big)}{7(1+n)(1+2 n)}B_{n, 3}+\\
+\frac{(1+n)(2+7n)-u(1+n)(3+7n)}{7(1+n)(1+2n)}B_{n,2},\quad\quad\quad\quad\quad\quad\quad\quad
\end{multline*}
$$
N(u,v)=\big(u\left(14+42n+28n^{2}\right)+v\left(6+14n+7n^{2}\right)-15-42n-28n^{2}\big)B_{n,3},
$$
\begin{multline*}
\big(P(u, v)\big)^{2}=10-12 u+2 u^{2}-2 v-v^{2}+\\
+\left(u\left(14+42 n+28 n^{2}\right)+v\left(6+14 n+7 n^{2}\right)-15-42n-28n^{2}\right)^{2},
\end{multline*}
\begin{multline*}
P(u,v)\cdot\mathbf{e}_1=7(1+n)(13+53 n+70n^{2}+28 n^{3})-\\
-14u(1+n)(1+2n)\left(6+14 n+7 n^{2}\right)-7v(1+n)^2\left(5+14 n+7 n^{2}\right).
\end{multline*}
Likewise, if $u\in I_{n,2}^{\prime}$ and $v\in\left[\frac{15+42 n+28 n^{2}-u\left(14+42 n+28 n^{2}\right)}{6+14 n+7 n^{2}}, \frac{17+56 n+56 n^{2}-u\left(15+56 n+56 n^{2}\right)}{7(1+n)(1+2 n)}\right]$, then
\begin{multline*}
P(u,v)=\frac{17+56n+56 n^{2}-u\left(15+56 n+56 n^{2}\right)-v(7(1+n)(1+2 n))}{7(1+n)(1+2n)}\mathbf{e}_{1}+\\
+\frac{(1+n)(3+7n)\big(17+56n+56 n^{2}-u\left(15+56 n+56 n^{2}\right)-7v(1+n)(1+2 n)\big)}{7(1+n)(1+2n)}B_{n, 2}+\\
+\frac{\left(6+14 n+7 n^{2}\right)\big(17+56n+56 n^{2}-u\left(15+56 n+56 n^{2}\right)-7v(1+n)(1+2n)\big)}{7(1+n)(1+2n)}B_{n,3}
\end{multline*}
and
\begin{multline*}
N(u,v)=\big(u(6+26 n+28 n^{2})+v(3+10 n+7 n^{2})-7-26n-28n^{2}\big) B_{n, 2}+\\
+\big(u\left(14+42 n+28 n^{2}\right)+v\left(6+14 n+7 n^{2}\right)-15-42n-28 n^{2}\big)B_{n,3}.
\end{multline*}
The same holds if $u\in I_{n,2}^{\prime \prime}$ and $v\in\left[\frac{7+26 n+28 n^{2}-u\left(6+26 n+28 n^{2}\right)}{3+10 n+7 n^{2}}, \frac{17+56 n+56 n^{2}-u\left(15+56 n+56 n^{2}\right)}{7(1+n)(1+2 n)}\right]$.
Moreover, in both cases, we have
$$
P(u,v)\cdot\mathbf{e}_1=14(1+n)(1+2 n)\left(17+56 n+56 n^{2}-u\left(15+56 n+56 n^{2}\right)-7v(1+n)(1+2 n)\right)
$$
and
$$
\big(P(u,v)\big)^{2}=2 \left(17+56 n+56 n^{2}-u\left(15+56 n+56 n^{2}\right)-7v(1+n)(1+2n)\right)^{2} .
$$
Thus, if $u\in I_{n,2}$, then
$$
t(u)=\frac{17+56n+56n^{2}-u\left(15+56n+56n^{2}\right)}{7(1+n)(1+2n)}.
$$

Now, we deal with $u\in I_{n,3}$.
If $u\in I_{n,3}^{\prime}$ and $v\in\left[0,\frac{15+42 n+28 n^{2}-u\left(14+42 n+28 n^{2}\right)}{6+14 n+7 n^{2}}\right]$, then
\begin{multline*}
P(u,v)=\frac{59+112 n+56 n^{2}-u\left(57+112 n+56n^{2}\right)-7v(1+n)(3+2n)}{7(1+n)(3+2n)}\mathbf{e}_{1}+ \\
+\frac{39+91 n+49 n^{2}-u\left(48+98 n+49 n^{2}\right)}{7(1+n)(3+2n)}B_{n, 3}+\frac{u(1+n)(3+7n)-(1+n)(2+7n)}{7(1+n)(3+2n)}B_{n,4}
\end{multline*}
and $N(u,v)=0$.
The same holds if $u\in I_{n,3}^{\prime\prime}$ and $v\in\left[0,\frac{31+58 n+28 n^{2}-u\left(30+58 n+28 n^{2}\right)}{11+18 n+7 n^{2}}\right]$. Then
\begin{align*}
\big(P(u,v)\big)^{2}&=10-12u+2u^{2}-2v-v^{2},\\
P(u,v)\cdot\mathbf{e}_1&=1+v.
\end{align*}
If $u\in I_{n,3}^\prime$ and $v \in\left[\frac{15+42 n+28 n^{2}-u\left(14+42 n+28 n^{2}\right)}{6+14 n+7 n^{2}}, \frac{31+58 n+28 n^{2}-u\left(30+58 n+28 n^{2}\right)}{11+18 n+7 n^{2}}\right]$, then
\begin{multline*}
P(u,v)=\frac{59+112n+56 n^{2}-u\left(57+112 n+56 n^{2}\right)-7v(1+n)(3+2n)}{7(1+n)(3+2n)}\mathbf{e}_{1}+\\
+\frac{\left(6+14n+7 n^{2}\right)\big(59+112n+56 n^{2}-u\left(57+112 n+56 n^{2}\right)-7v(1+n)(3+2n)\big)}{7(1+n)(3+2n)}B_{n, 3}+\\
+\frac{u(1+n)(3+7n)-(1+n)(2+7n)}{7(1+n)(3+2n)}B_{n,4}\quad\quad\quad\quad\quad\quad\quad\quad\quad\quad
\end{multline*}
$$
N(u,v)=\big(u\left(14+42 n+28 n^{2}\right)+v\left(6+14 n+7 n^{2}\right)-15-42n-28n^{2}\big)B_{n,3},
$$
\begin{multline*}
\big(P(u, v)\big)^{2}=10-12 u+2 u^{2}-2 v-v^{2}+\\
+\left(u\left(14+42 n+28 n^{2}\right)+v\left(6+14 n+7 n^{2}\right)-15-42n-28 n^{2}\right)^{2},
\end{multline*}
\begin{multline*}
P(u,v)\cdot\mathbf{e}_1=7(1+n)\big(13+53 n+70 n^{2}+28 n^{3}\big)-\\
-14u(1+n)(1+2n)\left(6+14 n+7 n^{2}\right)-7v(1+n)^2\left(5+14n+7n^{2}\right).
\end{multline*}
Similarly, if $u\in I_{n,3}^{\prime\prime}$ and $v \in\left[\frac{31+58 n+28 n^{2}-u\left(30+58 n+28 n^{2}\right)}{11+18 n+7 n^{2}}, \frac{15+42 n+28 n^{2}-u\left(14+42 n+28 n^{2}\right)}{6+14 n+7 n^{2}}\right]$, then
\begin{multline*}
P(u,v)=\frac{59+112n+56n^{2}-u(57+112n+56n^{2})-7v(1+n)(3+2n)}{7(1+n)(3+2n)}\mathbf{e}_{1}+\\
+\frac{(1+n)(11+7n)\big(59+112n+56n^{2}-u(57+112n+56n^{2})-7v(1+n)(3+2n)\big)}{7(1+n)(3+2n)}B_{n,4}+\\
+\frac{39+91n+49n^{2}-u\left(48+98n+49n^{2}\right)}{7(1+n)(3+2n)}B_{n,3},\quad\quad\quad\quad\quad
\end{multline*}
$$
N(u,v)=\big(u\left(30+58n+28n^{2}\right)+v\left(11+18n+7 n^{2}\right)-31-58n-28n^{2}\big)B_{n,4},
$$
\begin{multline*}
\big(P(u, v)\big)^{2}=10-12 u+2 u^{2}-2v-v^{2}+\\
+7\left(u\left(30+58n+28n^{2}\right)+v\left(11+18n+7n^{2}\right)-31-58n-28n^{2}\right)^{2},
\end{multline*}
\begin{multline*}
P(u,v)\cdot\mathbf{e}_1=2388+8372 n+10983 n^{2}+6370 n^{3}+1372 n^{4}-\\
-14u(1+n)^{2}(11+7n)(15+14n)-v\left(846+2772 n+3346 n^{2}+1764 n^{3}+343 n^{4}\right).
\end{multline*}
If $u\in I_{n,3}^{\prime}$ and $v\in\left[\frac{31+58 n+28 n^{2}-u\left(30+58 n+28 n^{2}\right)}{11+18 n+7 n^{2}}, \frac{59+112 n+56 n^{2}-u\left(57+112 n+56 n^{2}\right)}{7(1+n)(3+2 n)}\right]$, then
\begin{multline*}
P(u, v)=\frac{59+112 n+56 n^{2}-u\left(57+112 n+56 n^{2}\right)-7v(1+n)(3+2n)}{7(1+n)(3+2n)}\mathbf{e}_1+\\
+\frac{\left(6+14 n+7 n^{2}\right)\big(59+112 n+56 n^{2}-u\left(57+112 n+56 n^{2}\right)-7v(1+n)(3+2n)\big)}{7(1+n)(3+2n)}B_{n,3 }+\\
+\frac{(1+n)\left(11+7n\right)\big(59+112 n+56 n^{2}-u\left(57+112 n+56 n^{2}\right)-7v(1+n)(3+2n)\big)}{7(1+n)(3+2n)}B_{n, 4}
\end{multline*}
and
\begin{multline*}
N(u,v)=\left(u\left(14+42 n+28 n^{2}\right)+v\left(6+14 n+7 n^{2}\right)-15-42n-28n^{2}\right)B_{n,3}+\\
+\left(u\left(30+58n+28n^{2}\right)+v\left(11+18n+7n^{2}\right)-31-58n-28n^{2}\right)B_{n, 4}.
\end{multline*}
The same holds if $u\in I_{u,3}^{\prime\prime}$ and $v\in\left[\frac{15+42 n+28 n^{2}-u\left(14+42 n+28 n^{2}\right)}{6+14 n+7 n^{2}}, \frac{59+112 n+56 n^{2}-u\left(57+112 n+56 n^{2}\right)}{7(1+n)(3+2n)}\right]$.
In both cases, we have
$$
P(u,v)\cdot\mathbf{e}_1=14(1+n)(3+2n)\left(59+112n+56 n^{2}-u\left(57+112 n+56 n^{2}\right)-7v(1+n)(3+2 n)\right).
$$
and
$$
\big(P(u,v)\big)^{2}=2\left(59+112n+56n^{2}-u\left(57+112n+56n^{2}\right)-7v(1+n)(3+2n)\right)^{2}.
$$
Therefore, if  $u\in I_{n,3}$, then
$$
t(u)=\frac{59+112n+56n^{2}-u\left(57+112n+56n^{2}\right)}{7(1+n)(3+2n)}.
$$

Finally, we deal with $u\in I_{n,4}$.
If $u\in I_{n,4}^{\prime}$ and $v\in\left[0,\frac{31+58n+28n^{2}-u\left(30+58n+28n^{2}\right)}{11+18n+7n^{2}}\right]$, then
\begin{multline*}
P(u, v)=\frac{103+182 n+84 n^{2}-u\left(100+182 n+84 n^{2}\right)-v\left(36+56 n+21 n^{2}\right)}{36+56n+21 n^{2}}\mathbf{e}_{1}+\\
+\frac{17+32 n+14 n^{2}-u\left(20+34 n+14 n^{2}\right)}{36+56n+21n^{2}}B_{n,4}+\frac{u\left(48+98 n+49 n^{2}\right)-39-91n-49n^{2}}{36+56n+21 n^{2}}B_{n+1,1}
\end{multline*}
and $N(u,v)=0$. The same holds when $u\in I_{n,4}^{\prime\prime}$ and $v\in\left[0,\frac{44+70 n+28 n^{2}-u\left(43+70 n+28 n^{2}\right)}{15+21 n+7 n^{2}}\right]$.
In both cases, we compute
\begin{align*}
\big(P(u,v)\big)^{2}&=10-12u+2u^{2}-2 v-v^{2},\\
P(u,v)\cdot\mathbf{e}_1&=1+v.
\end{align*}
If $u\in I_{n,4}^\prime$ and $v \in\left[\frac{31+58 n+28 n^{2}-u\left(30+58 n+28 n^{2}\right)}{11+18 n+7 n^{2}}, \frac{44+70 n+28 n^{2}-u\left(43+70 n+28 n^{2}\right)}{15+21 n+7 n^{2}}\right]$, then
\begin{multline*}
P(u, v)=\frac{103+182n+84n^{2}-u(100+182n+84n^{2})-v(36+56n+21n^{2})}{36+56n+21n^{2}}\mathbf{e}_{1}+\\
+\frac{(1+n)(11+7n)\big(103+182n+84n^{2}-u(100+182n+84n^{2})-v(36+56n+21n^{2})\big)}{36+56n+21n^{2}}B_{n,4}+\\
+\frac{u(48+98n+49n^{2})-39-91n-49n^{2}}{36+56n+21n^{2}}B_{n+1,1},\quad\quad\quad\quad\quad\quad\quad
\end{multline*}
$$
N(u,v)=\left(u\left(30+58 n+28 n^{2}\right)+v\left(11+18 n+7 n^{2}\right)-31-58n-28n^{2}\right)B_{n,4},
$$
\begin{multline*}
\big(P(u,v)\big)^{2}=10-12 u+2 u^{2}-2 v-v^{2}+\\
+7\left(u\left(30+58 n+28 n^{2}\right)+v\left(11+18 n+7 n^{2}\right)-31-58n-28n^{2}\right)^{2},
\end{multline*}
\begin{multline*}
P(u,v)\cdot\mathbf{e}_1=2388+8372 n+10983 n^{2}+6370 n^{2}+1372 n^{4}-\\
-14u(1+n)^{2}(11+7n)(15+14 n)-\\
-v\left(846+2772 n+3346 n^{2}+1764 n^{3}+343 n^{4}\right).
\end{multline*}
If $u\in I_{n,4}^{\prime\prime}$ and $v \in\left[\frac{44+70 n+28 n^{2}-u\left(43+70 n+28 n^{2}\right)}{15+21 n+7 n^{2}}, \frac{31+58 n+28 n^{2}-u\left(30+58 n+28 n^{2}\right)}{11+18 n+7 n^{2}}\right]$, then
\begin{multline*}
P(u,v)=\frac{103+182n+84n^{2}-u(100+182n+84n^{2})-v(36+56n+21n^{2})}{36+56 n+21n^{2}}\mathbf{e}_{1}+\\
+\frac{(15+21n+7n^{2})\big(103+182n+84n^{2}-u(100+182n+84n^{2})-v(36+56n+21n^{2})\big)}{36+56 n+21n^{2}}B_{n+1,1}+\\
+\frac{17+32n+14n^{2}-u\left(20+34n+14n^{2}\right)}{36+56 n+21n^{2}}B_{n,4}\quad\quad\quad\quad\quad\quad\quad
\end{multline*}
and
$$
N(u,v)=\left(u\left(43+70 n+28 n^{2}\right)+v\left(15+21 n+7 n^{2}\right)-44-70n-28 n^{2}\right)B_{n+1,1}.
$$
Moreover, in this case, we have
\begin{multline*}
\big(P(u,v)\big)^{2}=10-12 u+2 u^{2}-2 v-v^{2}+\\
+2\left(u\left(43+70 n+28 n^{2}\right)+v\left(15+21 n+7 n^{2}\right)-44-70n-28n^{2}\right)^{2}\quad\quad
\end{multline*}
and
\begin{multline*}
P(u,v)\cdot\mathbf{e}_1=1321+3948 n+4396 n^{2}+2156 n^{3}+392 n^{4}-\\
-2u\left(15+21 n+7 n^{2}\right)\left(43+70 n+28 n^{2}\right)-\\
-v\left(449+1260 n+1302 n^{2}+588 n^{3}+98 n^{4}\right).
\end{multline*}
If $u\in I_{n,4}^{\prime}$ and $v\in\left[\frac{44+70 n+28 n^{2}-u\left(43+70 n+28 n^{2}\right)}{15+21 n+7 n^{2}},\frac{103+182 n+84 n^{2}-u\left(100+182 n+84 n^{2}\right)}{36+56 n+21 n^{2}}\right]$, then
\begin{multline*}
P(u,v)=\frac{103+182n+84n^{2}-u(100+182n+84n^{2})-v(36+56n+21n^{2})}{36+56n+21n^{2}}\mathbf{e}_1+\\
+\frac{(1+n)(11+7n)\big(103+182n+84n^{2}-u(100+182n+84n^{2})-v(36+56n+21n^{2})\big)}{36+56n+21n^{2}}B_{n,4}+\\
+\frac{(15+21 n+7n^{2})\big(103+182n+84n^{2}-u(100+182n+84n^{2})-v(36+56n+21n^{2})\big)}{36+56n+21n^{2}}B_{n+1,1}
\end{multline*}
and
\begin{multline*}
N(u,v) =\big(u(30+58 n+28 n^{2})+v(11+18 n+7 n^{2})-31-58n-28n^{2}\big)B_{n,4} +\\
+\big(u(43+70 n+28 n^{2})+v(15+21 n+7 n^{2})-44-70n-28n^{2}\big)B_{n+1,1}.
\end{multline*}
The same holds when $u\in I_{n,4}^{\prime\prime}$ and $v \in\left[\frac{31+58n+28 n^{2}-u\left(30+58n+28 n^{2}\right)}{11+18n+7 n^{2}},\frac{103+182n+84n^{2}-u\left(100+182 n+84n^{2}\right)}{36+56n+21n^{2}}\right]$.
Moreover, in both cases, we have
$$
P(u,v)\cdot\mathbf{e}_1=(36+56n+21n^{2})\big(103+182n+84n^{2}-u(100+182n+84n^{2})-v(36+56n+21n^{2})\big)
$$
and
$$
\big(P(u,v)\big)^{2}=\big(103+182n+84n^{2}-u(100+182n+84 n^{2})-v(36+56 n+21 n^{2})\big)^{2}
$$
Thus, if $u\in I_{n,4}$, then
$$
t(u)=\frac{103+182n+84n^{2}-u\left(100+182n+84n^{2}\right)}{36+56n+21n^{2}}.
$$

Now, we are ready to compute $S(W_{\bullet,\bullet}^{T};\mathbf{e}_1)$.
Namely, for every $i\in\{1,2,3,4\}$,~we set
$$
S_{n,i}=\frac{3}{14}\int_{I_{n,i}}\int_0^{t(u)}\big(P(u,v)\big)^2dvdu.
$$
Then
$$
S\big(W_{\bullet,\bullet}^{T};\mathbf{e}_1\big)=\sum_{n=0}^{\infty}\Big(S_{n,1}+S_{n,2}+S_{n,3}+S_{n,4}\Big).
$$
On the other hand, integrating, we get
$$
S_{n,1}=
\left\{\aligned
&\frac{84365}{114688}\ \text{if $n=0$}, \\
&\frac{(8+28n+21n^{2})A_{n,1}}{448 n^{4}(1+n)(2+7n)^{4}(3+7n)^{4}(4+7n)^{4} (1+7n+7n^{2})}\ \text{if $n\geqslant 1$}, \\
\endaligned
\right.
$$
where
\begin{multline*}
A_{n,1}=1536+109312 n+2935552 n^{2}+42681728 n^{3}+386407488 n^{4}+2335296292 n^{5}+\\
+9789648099 n^{6}+29038364761 n^{7}+61312905318 n^{8}+91454579804 n^{9}+\\
+94035837280 n^{10}+63317750608 n^{11}+25088413952 n^{12}+4427367168 n^{13}.\quad\quad
\end{multline*}
Similarly, we get
$$
S_{n,2}=\frac{(1+2n)A_{n,2}}{4(1+n)(2+7n)^{4}(3+7n)^{4}(4+7n)^{4} (6+14n+7n^{2})},
$$
where
\begin{multline*}
A_{n,2}=1618654+31459234 n+271069253 n^{2}+1362423916 n^{3}+4419070194 n^{4}+9654348284n^{5}+\\
+14368501182 n^{6}+14362052096 n^{7}+9209328422 n^{8}+3412762192n^9+553420896n^{10}
\end{multline*}
Likewise, we get
$$
S_{n,3}=\frac{(3+2n)A_{n,3}}{4(1+n)(3+7n)^{4}(6+7n)^{4}(8+7n)^{4}(11+7n)(6+14 n+7 n^{2})},
$$
where
\begin{multline*}
A_{n,3}=1167997914+15454923336 n+91492878645 n^{2}+319934133575 n^{3}+\\
+734395997090 n^{4}+1162203105378 n^{5}+1294197714054 n^{6}+1014406754242 n^{7}+\\
+548632346402 n^{8}+195059453722 n^{9}+41045383120 n^{10}+3873946272n^{11}.\quad\quad
\end{multline*}
Finally, we get
$$
S_{n,4}=\frac{(36+56 n+21 n^{2})A_{n,4}}{448(1+n)^{4}(6+7n)^{4}(8+7n)^{4}(10+7n)^{4}(11+7n)(15+21n+7 n^{2})},
$$
where
\begin{multline*}
A_{n,4}=365613573312+4021500121920 n+20341847967024 n^{2}+\\
+62650071283024 n^{3}+131072047236004 n^{4}+196698030664492 n^{5}+217823761840153 n^{6}+\\
+180219167765455 n^{7}+111395400841326 n^{8}+50802960251820 n^{9}+16615457209344 n^{10}+\\
+3690223711216 n^{11}+498816700928 n^{12}+30991570176 n^{13}.\quad\quad\quad\quad
\end{multline*}
Then, adding, we get
$$
S\big(W_{\bullet,\bullet}^{T};\mathbf{e}_1\big)=\sum_{n=0}^{\infty}\Big(S_{n,1}+S_{n,2}+S_{n,3}+S_{n,4}\Big)\approx 0.976712233\ldots<1.
$$

Finally, let us compute $S(W_{\bullet, \bullet,\bullet}^{T,\mathbf{e}_1};P)$.
For every $i\in\{1,2,3,4\}$, we set
\begin{align*}
M_{n,i}^\prime&=\frac{3}{14}\int_{I_{n,i}^\prime}\int_0^{t(u)}\Big(\big(P(u,v)\cdot\mathbf{e}_1\big)\Big)^2dvdu,\\
M_{n,i}^{\prime\prime}&=\frac{3}{14}\int_{I_{n,i}^{\prime\prime}}\int_0^{t(u)}\Big(\big(P(u,v)\cdot\mathbf{e}_1\big)\Big)^2dvdu.\quad\quad\quad\quad\quad\quad\quad
\end{align*}
Then
$$
S\big(W_{\bullet, \bullet,\bullet}^{T,\mathbf{e}_1};P\big)=
\sum_{n=0}^{\infty}\sum_{i=1}^{4}\Big(M_{n,i}^\prime+M_{n,i}^{\prime\prime}\Big)+\frac{3}{7}\int_{0}^1\int_0^{t(u)}\big(P(u,v)\cdot \mathbf{e}_1\big)\mathrm{ord}_P\Big(N(u,v)\big|_{\mathbf{e}_1}\Big)dvdu.
$$
On the other hand, integrating, we get
$$
M_{n,1}^\prime=
\left\{\aligned
&\frac{1403}{22268}\ \text{if $n=0$}, \\
&\frac{(1+n)A_{n,1}^\prime}{448n^{4}(1+3n)^{4}(3+7n)^{4}(1+7n+7 n^{2})}\ \text{if $n\geqslant 1$}, \\
\endaligned
\right.
$$
where
\begin{multline*}
A_{n,1}^\prime=1+81 n+2535 n^{2}+37209 n^{3}+301046 n^{4}+1459736 n^{5}+\\
+4420190 n^{6}+8425410 n^{7}+9821448 n^{8}+6392736 n^{9}+1778112 n^{10}.\quad\quad\quad\quad
\end{multline*}
Similarly, we get
$$
M_{n, 1}^{\prime\prime}=\frac{(1+7 n+7 n^{2})A_{n,1}^{\prime\prime}}{28(1+n)(1+3 n)^{4}(2+7 n)^{4}(3+7 n)^{4}(4+7 n)^{4}},
$$
where
\begin{multline*}
A_{n,1}^{\prime\prime}=480574+12906866 n+157271760 n^{2}+1149521334 n^{3}+5612285145 n^{4}+\\
+19278934535 n^{5}+47770884833 n^{6}+86016481159 n^{7} +111679016743 n^{8}+\\
+101939513907 n^{9}+62077730148 n^{10}+22635902898 n^{11}+3735591048 n^{12}.
\end{multline*}
Likewise, we have
$$
M_{n,2}^{\prime}=\frac{(6+14n+7n^{2})A_{n,2}^{\prime}}{224(1+n)(1+2 n)^{3}(2+7 n)^{4}(3+7 n)^{4}(4+7 n)^{4}},
$$
and
$$
M_{n,2}^{\prime\prime}=\frac{11780+111142 n+430951 n^{2}+875637 n^{3}+978656 n^{4}+566832 n^{5}+131712 n^{6}}{224(1+2 n)^{3}(3+7 n)^{4}( 6+14 n+7 n^{2})}.
$$
where
\begin{multline*}
A_{n,2}^{\prime}=1561176+35176776 n+356105548 n^{2}+2137950448 n^{3}+\\
+8458603286 n^{4}+23158717414 n^{5}+44778314889 n^{6}+61151030584 n^{7}+\\
+57807289939 n^{8}+36026947376 n^{9}+13321631568 n^{10}+2213683584 n^{11}.
\end{multline*}
Similarly, we have
$$
M_{n,3}^{\prime}=\frac{(11+7n)A_{n,3}^{\prime}}{224(1+n)^{3}(3+7 n)^{4}(13+14 n)^{4} (6+14 n+7 n^{2})},
$$
where
\begin{multline*}
A_{n,3}^{\prime}=13726028+164541190 n+859036123 n^{2}+2564002455 n^{3}+4823323519 n^{4}+\\
+5933644367 n^{5}+4776917782 n^{6}+2428774768 n^{7}+708314208 n^{8}+90354432 n^{9}.
\end{multline*}
Likewise, we have
$M_{n,3}^{\prime\prime}=\frac{(6+14 n+7 n^{2})A_{n,3}^{\prime\prime}}{224(1+n)^{3}(6+7 n)^{4}(8+7 n)^{4}(11+7 n)(13+14n)^{4}}$,
where
\begin{multline*}
A_{n,3}^{\prime\prime}=67760261208+703706084640 n+3313300067388 n^{2}+9335574166156 n^{3}+\\
+17489294547578 n^{4}+22873117200584 n^{5}+21308562209725 n^{6}+14139587568253 n^{7}+\\
+6548997703738 n^{8}+2016283621072 n^{9}+371345421216 n^{10}+30991570176 n^{11}.
\end{multline*}
Similarly, we see that
$$
M_{n,4}^{\prime}=\frac{(15+21n+7n^{2})A_{n,4}^{\prime}}{28(1+n)^{4}(6+7 n)^{4}(8+7 n)^{4}(11+7 n)(23+21 n)^{4}},
$$
where
\begin{multline*}
A_{n,4}^{\prime}=88135013250+967134809574 n+4853884596732 n^{2}+\\
+14732868828434 n^{3}+30120687035243 n^{4}+43697011451345 n^{5}+46124583653603 n^{6}+\\
+35692827118809 n^{7}+20096052100397 n^{8}+8028312817917 n^{9}+\\
+2160120347280 n^{10}+351456857766 n^{11}+26149137336 n^{12}.\quad\quad\quad\quad
\end{multline*}
Finally, we have
$$
M_{n,4}^{\prime\prime}=\frac{(11+7n)A_{n,4}^{\prime\prime}}{448(1+n)^{4}(10+7 n)^{4}(23+21 n)^{4} (15+21 n+7 n^{2})},
$$
where
\begin{multline*}
A_{n,4}^{\prime\prime}=7582266167+59702225967 n+210973884925 n^{2}+440580768679 n^{3}+\\
+602090743422 n^{4}+562572998512 n^{5}+363945674554 n^{6}+160955181870 n^{7}+\\
+46566357768 n^{8}+7957643904 n^{9}+609892416 n^{10}.\quad\quad\quad\quad\quad\quad\quad
\end{multline*}
Now, adding terms together, we see that
\begin{equation}
\label{equation:2-7-final}
S\big(W_{\bullet, \bullet,\bullet}^{T,\mathbf{e}_1};P\big)\leqslant 0.974+\frac{3}{7}\int_{0}^1\int_0^{t(u)}\big(P(u,v)\cdot \mathbf{e}_1\big)\mathrm{ord}_P\Big(N(u,v)\big|_{\mathbf{e}_1}\Big)dvdu.
\end{equation}

Now, for every $i\in\{1,2,3,4\}$ and any irreducible component $\ell$ of the curve $B_{n,i}$, we let
\begin{multline*}
\quad\quad\quad\quad\quad F_{n,i}=\frac{3}{7}\int_{0}^1\int_0^{t(u)}\big(P(u,v)\cdot \mathbf{e}_1\big)\mathrm{ord}_\ell\big(N(u,v)\big)\big(\ell\cdot\mathbf{e}_1\big)dvdu=\\
\quad=\frac{3}{7}\int_{I_{n,i-1}}\int_0^{t(u)}\big(P(u,v)\cdot \mathbf{e}_1\big)\mathrm{ord}_\ell\big(N(u,v)\big)\big(\ell\cdot\mathbf{e}_1\big)dvdu+\\
+\frac{3}{7}\int_{I_{n,i}}\int_0^{t(u)}\big(P(u,v)\cdot \mathbf{e}_1\big)\mathrm{ord}_\ell\big(N(u,v)\big)\big(\ell\cdot\mathbf{e}_1\big)dvdu+\\
+\frac{3}{7}\int_{I_{n,i+1}}\int_0^{t(u)}\big(P(u,v)\cdot \mathbf{e}_1\big)\mathrm{ord}_\ell\big(N(u,v)\big)\big(\ell\cdot\mathbf{e}_1\big)dvdu,\quad\quad\quad\quad\quad\quad
\end{multline*}
where we set $I_{0,0}=\varnothing$, $I_{n,5}=I_{n+1,1}$ for $n\geqslant 0$, and $I_{n,0}=I_{n-1,4}$ for $n\geqslant 1$.
Since irreducible components of the~curve $B_{n,i}$ are disjoint, we get
$$
\frac{3}{7}\int_{0}^1\int_0^{t(u)}\big(P(u,v)\cdot \mathbf{e}_1\big)\mathrm{ord}_P\Big(N(u,v)\big|_{\mathbf{e}_1}\Big)dvdu\leqslant \sum_{n=0}^{\infty}\sum_{i=1}^{4}F_{n,i}.
$$
On the other hand, each $F_{n,i}$ is not difficult to compute.
For instance, we have
\begin{multline*}
F_{0,1}=\frac{3}{7}\int_{0}^{\frac{1}{3}}\int_{2-u}^{\frac{7-6u}{3}}(v+u-2)(5-2u-v)dvdu+\\
+\frac{3}{7}\int_{0}^{\frac{1}{3}}\int_{\frac{7-6 u}{3}}^{\frac{19-16u}{8}}8(u+v-2)(19-16u-8v)dvdu+\\
+\frac{3}{7}\int_{\frac{1}{3}}^{\frac{3}{8}}\int_{2-u}^{\frac{19-16 u}{8}}8(u+v-2)(19-16u-8v)dvdu=\frac{281}{32256}.
\end{multline*}
Similarly, we see that
$$
F_{n,1}=\frac{3(1+7 n+7 n^{2} )^{2}}{2 n^{2}(1+3 n)(-1+7 n)(1+7 n)(2+7 n)(3+7 n)^{2}(4+7 n)(2+21 n)}
$$
for $n\geqslant 1$. Likewise, we get
$$
F_{n,2}=
\left\{\aligned
&\frac{5}{3584}\ \text{if $n=0$}, \\
&\frac{1+n}{112 n(1+2 n)(1+3 n)(2+7 n)(3+7 n)^{2}(4+7 n)}\ \text{if $n\geqslant 1$}. \\
\endaligned
\right.
$$
Likewise, for every $n\geqslant 0$, we have
$$
F_{n,3}=\frac{15 (6+14n+7n^{2})^{2}}{4(1+n)^{2}(1+2 n)(2+7 n)(3+7 n)^{2}(4+7 n)(6+7 n)(8+7 n)(13+14 n)}
$$
and
$$
F_{n,4}=\frac{(11+7 n)^{2}}{112(1+n)^{2}(3+7 n)(6+7 n)(8+7 n)(10+7 n)(13+14 n)(23+21 n)}.
$$
Now, one can easily check that the total sum of all $F_{n,1}$, $F_{n,2}$, $F_{n,3}$, $F_{n,4}$ is at most $0.014$.
This and \eqref{equation:2-7-final} give $S(W_{\bullet, \bullet,\bullet}^{T,\mathbf{e}_1};P)\leqslant 0.974+0.014=0.988$.
Using \eqref{equation:2-7-P-in-E-inequality}, we get $\delta_P(X)>1$.

\begin{corollary}
\label{corollary:2-7-final}
All smooth Fano threefolds in the~family \textnumero 2.7 are K-stable.
\end{corollary}

\section*{Declarations}
\medskip
\noindent
\textbf{Competing interests.}
The authors declares that there is no conflict of interest.

\medskip
\noindent
\textbf{Financial Support.}
Cheltsov has been supported by EPSRC grant \textnumero EP/V054597/1.
Fujita has been supported by JSPS KAKENHI Grant \textnumero 22K03269.

This paper has been written during Cheltsov's stay at IHES (Bures-sur-Yvette, France).
He thanks the institute for the hospitality and excellent working conditions.

\end{document}